\documentclass[a4paper,11pt]{article}

\usepackage[T1]{fontenc}
\usepackage{lmodern}

\usepackage{dsfont}

\usepackage[latin1]{inputenc}
\usepackage{amsmath}
\usepackage{amsthm}
\usepackage{amssymb}
\usepackage{mathrsfs}
\usepackage{graphicx}
\usepackage[all]{xy}
\usepackage{hyperref}

\usepackage{makeidx}

\usepackage{stmaryrd}
\usepackage{caption}

\usepackage{abstract} 

\newtheorem{thm}{Theorem}[section]
\newtheorem{cor}[thm]{Corollary}
\newtheorem{claim}[thm]{Claim}
\newtheorem{fact}[thm]{Fact}

\newtheorem{lemma}[thm]{Lemma}
\newtheorem{prop}[thm]{Proposition}

\theoremstyle{definition}
\newtheorem{definition}[thm]{Definition}
\newtheorem{ex}[thm]{Example}
\newtheorem{remark}[thm]{Remark}
\newtheorem{question}[thm]{Question}
\newtheorem{problem}[thm]{Problem}
\newtheorem{conj}[thm]{Conjecture}

\def\rquotient#1#2{%
	\makeatletter
	\raise.3ex\hbox{$#1$}/\lower.3ex\hbox{$#2$}%
	\makeatother
}	

\makeatletter
\newcommand{\subjclass}[2][2010]{%
	\let\@oldtitle\@title%
	\gdef\@title{\@oldtitle\footnotetext{#1 \emph{Mathematics subject classification.} #2}}%
}
\newcommand{\keywords}[1]{%
	\let\@@oldtitle\@title%
	\gdef\@title{\@@oldtitle\footnotetext{\emph{Key words and phrases.} #1.}}%
}
\makeatother

\newcommand{\Address}{{
		\bigskip
		\small
		
\noindent\textsc{Universit\'e de Montpellier\\ 
Institut Math\'ematiques Alexander Grothendieck\\
Place Eug\`ene Bataillon\\
34090 Montpellier (France)}\par\nopagebreak
\noindent \textit{E-mail address}: \texttt{anthony.genevois@umontpellier.fr}
		
}}

\makeindex

\title{Beyond graph products and cactus groups: quandle products of groups}
\date{\today}
\author{Anthony Genevois}

\subjclass{Primary 20F65. Secondary 20F10, 20E22, 05C25.}
\keywords{graph product, wreath product, cactus group, quasi-median graph}

\begin{document}

\maketitle

\begin{abstract}
In this paper, we introduce and initiate the study of quandle products of groups, a family of groups that includes graph products of groups, cactus groups, wreath products, and the recently introduced trickle groups. Our approach is geometric: we show that quandle products admit quasi-median Cayley graphs; and, then, we exploit this geometry to deduce various valuable information about quandle products. 
\end{abstract}

\small
\tableofcontents
\normalsize

\section{Introduction}

\noindent
In group theory, an utopian goal is to classify all the groups into a well-organised collection of families. Of course, in full generality or even when restricted to finitely generated groups, nobody expects this goal to be reached. However, we can tend to this objective in the small part of the world of groups that we have explored so far. For this, a reasonable strategy is:
\begin{description}
	\item[(Step 1)] Identify a large collection of groups sharing some similarities (in their constructions or in the properties they satisfy).
	\item[(Step 2)] Propose a formal framework, as simple as possible, that encompasses all these groups.
	\item[(Step 3)] Justify that the formalism is meaningful by producing unified proofs.
\end{description}
The second step is the most delicate one. For instance, it has not been completed for Thompson-like and braided-like groups. 

\medskip \noindent
In this article, we apply this program to a family of groups whose constructions are inspired by \emph{cactus groups}.

\paragraph{Step 1.} Formally, given an integer $n \geq 2$, the \emph{cactus group} $J_n$ is defined by generators $s_{I}$, with $I$ a subinterval of $[n]:=\{1, \ldots, n\}$ of length $\geq 2$, submitted to the relations:
\begin{itemize}
	\item $s_{I}^2=1$ for every $I \subset [n]$;
	\item $s_{I}s_{J} = s_{I}s_{J}$  for all $I,J \subset [n]$ satisfying $I \cap J = \emptyset$;
	\item $s_{I}s_{J} = s_{J}s_{J \ast I}$ for all $I,J \subset [n]$ satisfying $I \subset J$, where $J \ast I$ denotes the image of $I$ under the central symmetry in $J$.
\end{itemize}
From a more visual point of view, it follows essentially from the presentation given above that the elements of $J_n$ can be represented by \emph{braided-like pictures}. 

\medskip \noindent
\begin{minipage}{0.3\linewidth}
\includegraphics[width=0.95\linewidth]{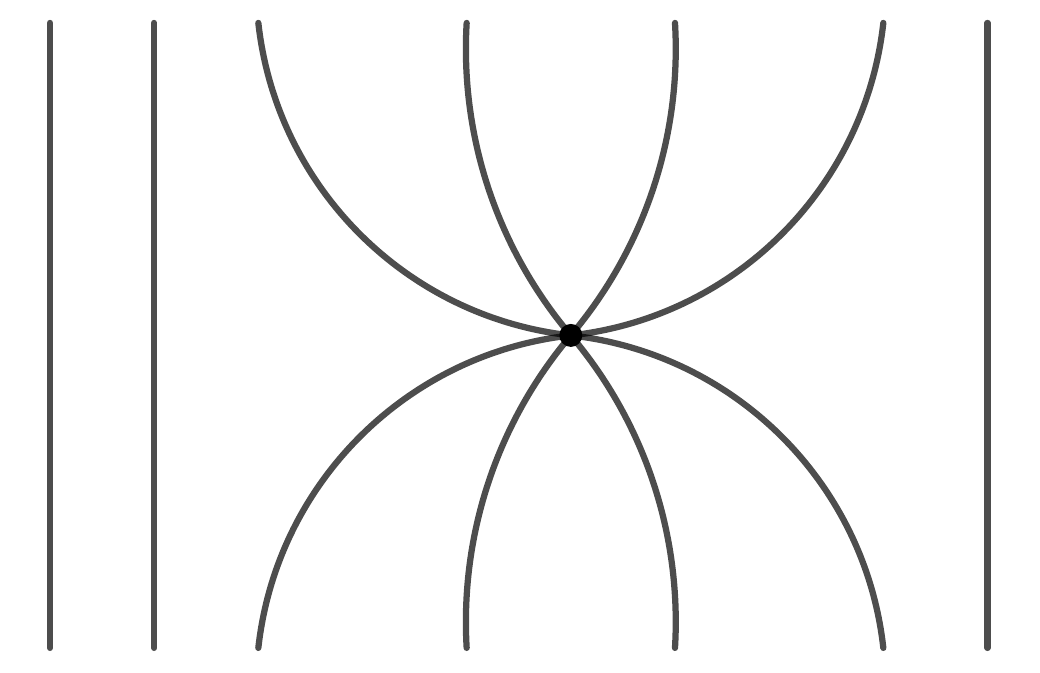}
\end{minipage}
\begin{minipage}{0.68\linewidth}
Given $n$ vertical strands and two indices $1 \leq p<q \leq n$, the generator $s_{[p,q]}$ centrally inverts the strands numbered from $p$ to $q$, all the strands meeting simultaneously at a unique node during the process. For instance, the picture on the left illustrates the generator $s_{[3,6]}$ in the cactus group $J_7$. 
\end{minipage}

\medskip \noindent
Then, the relations given above can be visualised as follows:

\medskip \noindent
\includegraphics[width=\linewidth]{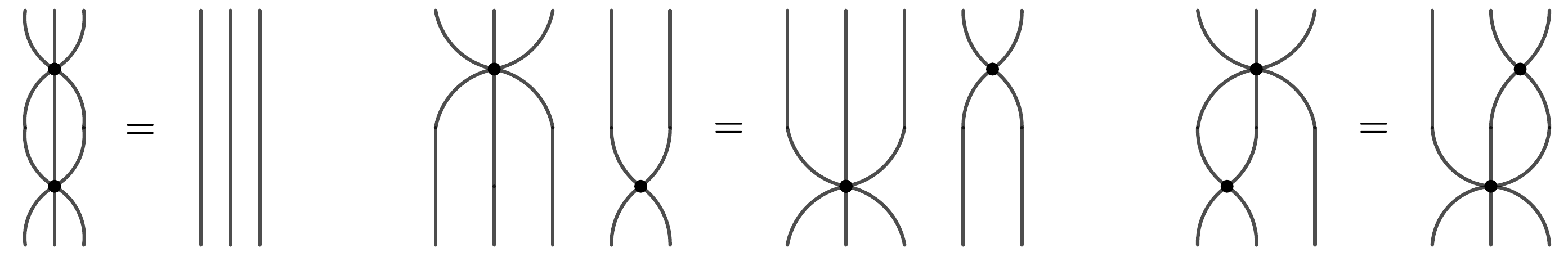}

\medskip \noindent
In practice, it will be more convenient to draw braided-like pictures with strands meeting at an interval instead of a single node.
\begin{center}
\includegraphics[width=0.5\linewidth]{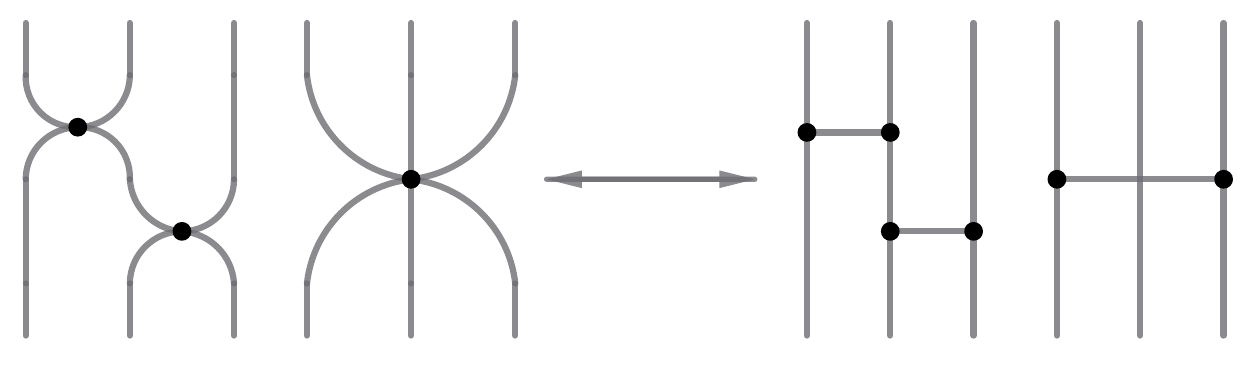}
\end{center}
Thus, an element of the cactus group is just a stack of intervals submitted to a few Reidemeister-like relations, namely:
\begin{center}
\includegraphics[width=0.5\linewidth]{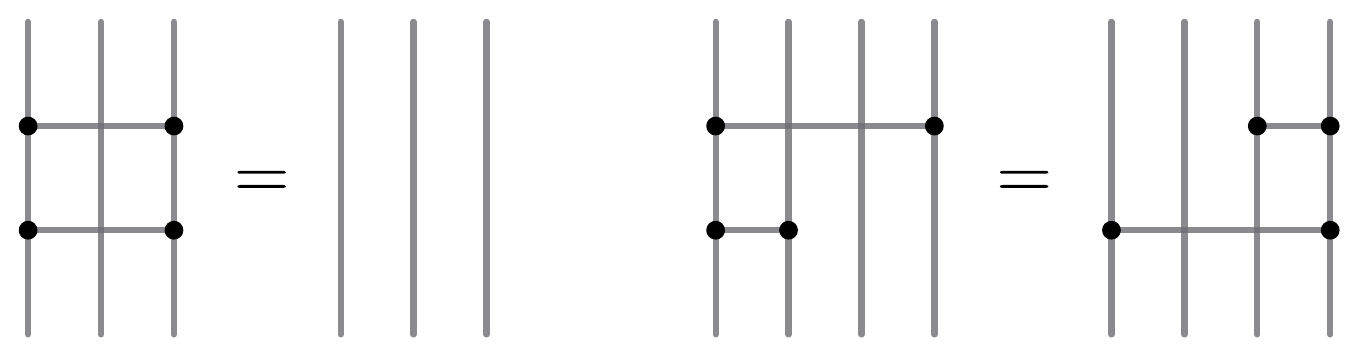}
\end{center}
where we omit, as we will keep doing it, the commutation relations between disjoint intervals, which are implicit.

\medskip \noindent
\begin{minipage}{0.2\linewidth}
\begin{center}
\includegraphics[width=0.9\linewidth]{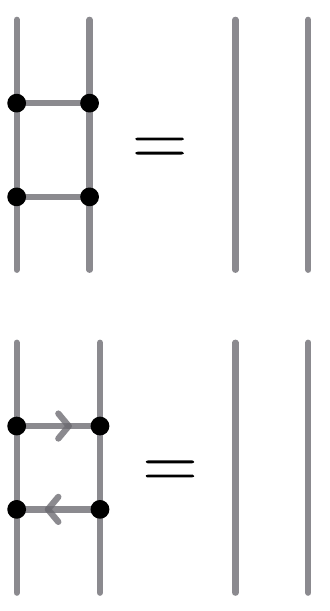}
\end{center}
\end{minipage}
\begin{minipage}{0.78\linewidth}
By restricting ourselves with intervals $[p,p+1] \subset [n]$, submitted to same Reidemeister moves, one recover the so-called \emph{flat braid group} $\mathrm{FB}_n$, also known as \emph{twin group} $T_n$. As shown by its presentation 
$$\langle s_p \ (1 \leq p \leq n-1) \mid s_p^2=1 \text{ for every } p, \ [s_p,s_q]=1 \text{ for all } |p-q| \geq 2 \rangle,$$
$\mathrm{FB}_n$ coincides with the right-angled Coxeter group $C(P_{n-2}^\mathrm{opp})$, where $P_{n-2}^\mathrm{opp}$ denotes the opposite graph of the path $P_{n-2}$ of length $n-2$. The corresponding right-angled Artin group can also be described similarly by decorating intervals with orientations.
\end{minipage}

\medskip \noindent
As another example, consider an infinite oriented interval $(- \infty, + \infty)$ and small intervals $(p-1/2,p+1/2)$, $p \in \mathbb{Z}$, submitted to the moves:
\begin{center}
\includegraphics[width=\linewidth]{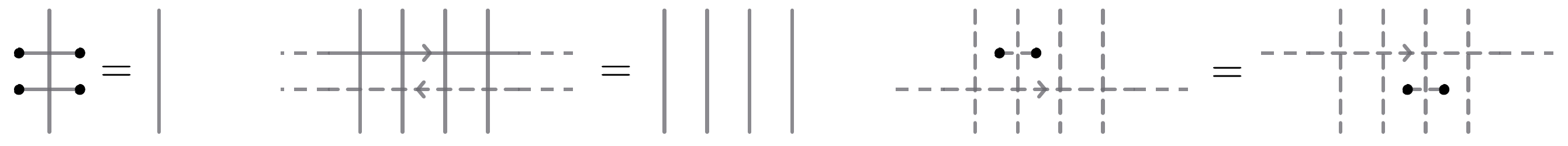}
\end{center}
Then, one gets a group with 
$$\langle t, a_i \ (i \in \mathbb{Z}) \mid a_i^2=1 \text{ and }  a_it = ta_{i+1} \text{ for every } i \rangle$$
as a presentation, which corresponds to the well-known lamplighter group $\mathbb{Z}_2 \wr \mathbb{Z}:= \left( \bigoplus_\mathbb{Z} \mathbb{Z}_2 \right) \rtimes \mathbb{Z}$. By decorating the small intervals with coefficients coming from a group $F$, we can get the more general wreath product $F \wr \mathbb{Z}$ by modifying slightly our Reidemeister moves:
\begin{center}
\includegraphics[width=\linewidth]{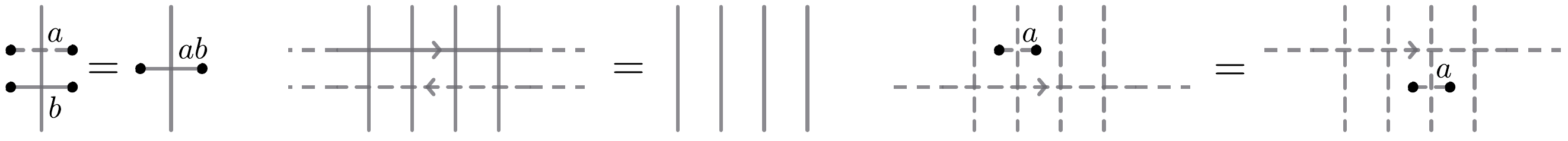}
\end{center}

\noindent
Playing with intervals is convenient for drawings, but the construction naturally extends to subsets of an arbitrary set. For instance, we can adapt the definition of the cactus group $J_n$ by organising the $n$ strands along a cylinder instead of a plane. The group thus obtained is known as the affine cactus group \cite{AffineCactus}. We can also imagine higher-dimensional cactus groups. Given an integer $n \geq 2$ and a dimension $d \geq 1$, consider the cube $[n]^d$. Every subcube $Q \subset [n]^d$ is naturally endowed with a central symmetry $\sigma_Q$ (which can be described as the product of $d$ reflections). Then, we naturally get a cactus group over $[n]^d$ defined by
$$\left\langle s_Q \ (Q \subset [n]^d \text{ subcube}) \mid \begin{array}{l} s_Q^2=1, \ [s_Q,s_R]=1 \text{ for all } Q \cap R = \emptyset, \\ s_Qs_R = s_Rs_{\sigma_R(Q)} \text{ for all } Q \subset R \end{array} \right\rangle$$
By generalising the construction above to arbitrary sets and subsets, one easily recovers generalised cactus groups (see Section~\ref{section:cactus}) as well as arbitrary graph products and arbitrary permutational wreath products. 

\medskip \noindent
We can keep playing with the construction and produce new groups. For instance, by decorating the intervals in the construction of cactus groups with orientations, we are naturally led to introduce the following Reidemeister moves:
\begin{center}
\includegraphics[width=0.5\linewidth]{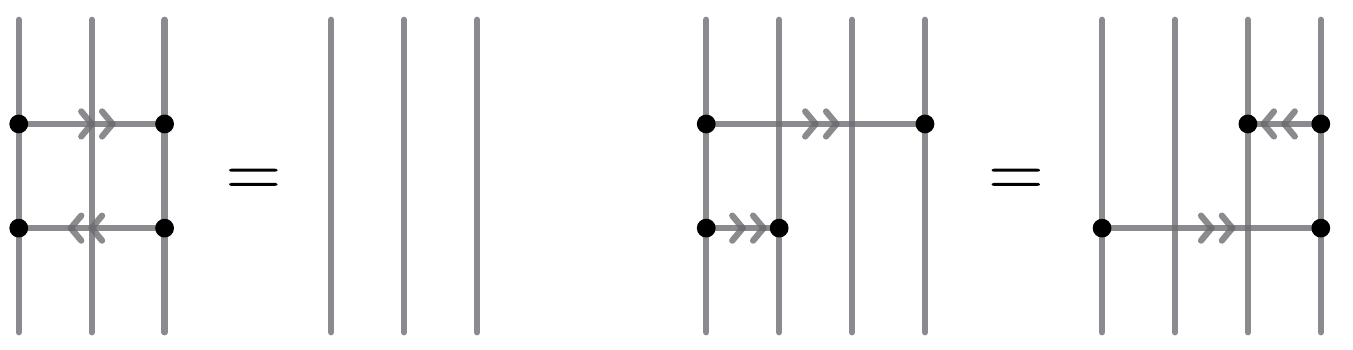}
\end{center}
This leads to the group, which we refer to as the \emph{oriented cactus group}, admitting
$$\left\langle s_I \ (I \subset [n] \text{ subinterval}) \mid \begin{array}{l} [s_I,s_J]=1 \text{ when } I \cap J= \emptyset \\ s_Is_J = s_Js_{J \ast I}^{-1} \text{ when } I \subset J \end{array} \right\rangle$$
as a presentation, where $J \ast I$ still denotes the image of $I$ under the central symmetry in $J$. We can also mix the intervals and moves used to define cactus groups and lamplighters:
\begin{center}
\includegraphics[width=0.6\linewidth]{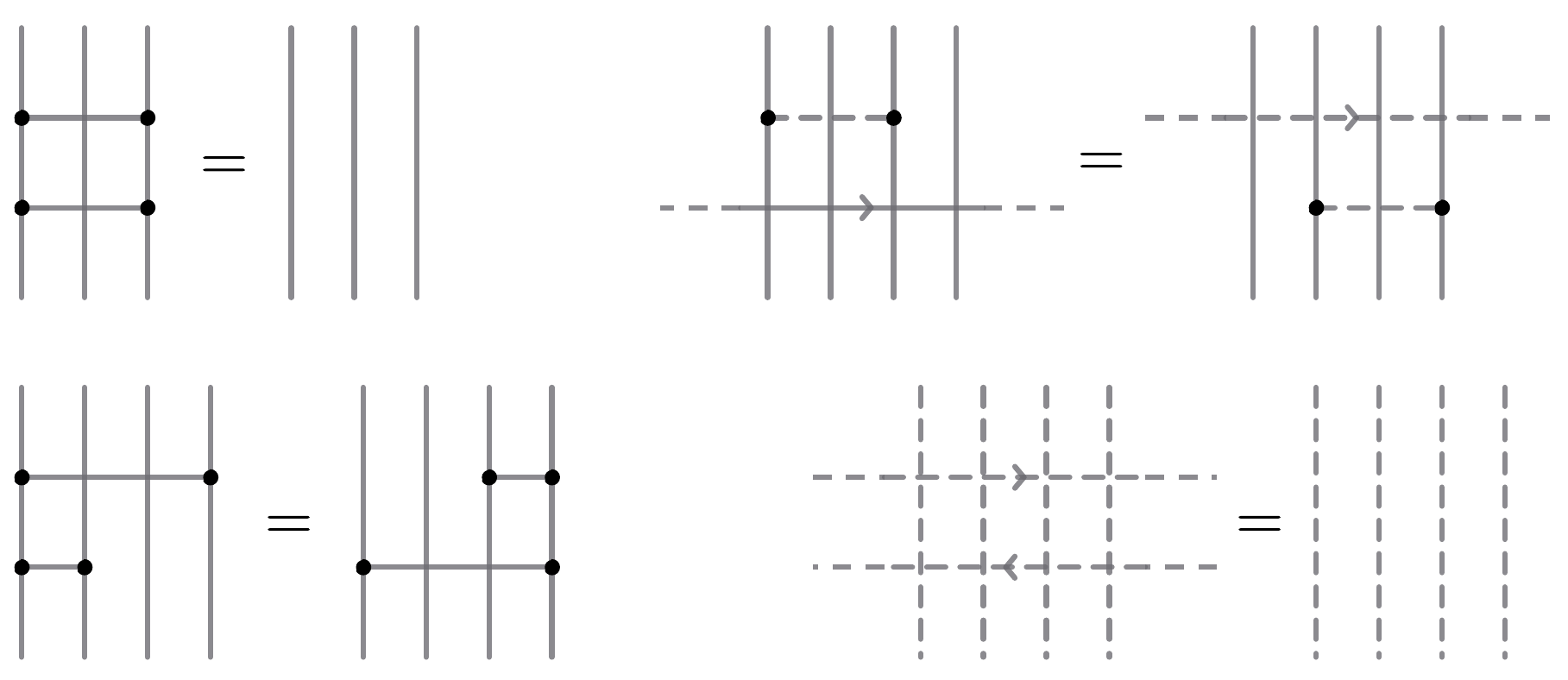}
\end{center}
This leads to a \emph{lamp-cactus group}, which admits
$$\left\langle t, s_I \ (I \subset \mathbb{Z} \text{ interval}) \mid \begin{array}{l} s_I^2 = 1, s_It = ts_{I+1} \\ \left[ s_I,s_J \right]=1 \text{ if } I \cap J=\emptyset \\ s_Is_J = s_J s_{J \ast I} \text{ if } I \subset J \end{array} \right\rangle$$
as a presentation and which is a finitely generated group that splits as a semidirect product $J_\infty \rtimes \mathbb{Z}$ (where $J_\infty$ denotes the cactus groups on infinitely many strands but whose elements are finitely supported).

\paragraph{Step 2.} Now, the problem is to find a simple construction that allows us to include all the examples previously mentioned. In this perspective, we introduce \emph{quandle products} of groups.

\medskip \noindent
Start with a set $\mathbb{S}$ and a collection of subsets $\mathrm{P}(\mathbb{S}) \subset \mathfrak{P}(\mathbb{S})$. Each subset $S \in \mathrm{P}(\mathbb{S})$ is decorated with a group $G_S$. Each factor $G_S$ acts on the disjoint union of groups $\bigsqcup_{R \subset S} G_R$, thought of as a groupoid. We require such a data to satisfy the following two conditions:
\begin{itemize}
	\item[(1)] the action $G_S \curvearrowright \{ R \in \mathrm{P}(\mathbb{S}) \mid R \subset S\}$ induced by $G_S \curvearrowright \bigsqcup_{R \subset S} G_R$ preserves inclusion and disjointness;
	\item[(2)] for all $S,T \in \mathrm{P}(\mathbb{S})$ satisfying $S \subset T$, the equality $$k \ast ( h \ast g) = (k\ast h ) \ast (k \ast g)$$ holds for all $k \in G_T$, $h \in G_S$, and $g \in \bigsqcup_{R \subset S} G_R$. 
\end{itemize}
Then, we define the quandle product $\mathfrak{Q}$ as
$$\left( \underset{S \in P (\mathbb{S})}{\ast} G_S \right) / \left\langle\left\langle \begin{array}{l} gh=hg \text{ for all } g \in G_R, h \in G_S \text{ with } R \cap S= \emptyset \\ gh=h(h\ast g) \text{ for all } g \in G_R, h \in G_S \text{ with } R \subset S \end{array} \right\rangle \right\rangle.$$
Condition (2), referred to as the \emph{quandle relation} and which echoes the definition of \emph{quandles}, gives its name to quandle product. We refer the reader to Section~\ref{section:Quandle} for a slightly more conceptual definition of quandle products, but which turns out to be equivalent to the definition above according to Lemma~\ref{lem:Oposet}. 

\medskip \noindent
Examples of quandle products include generalised cactus groups, graph products of groups, and permutational wreath products, as well as all the groups previously mentioned. We refer the reader to Sections~\ref{section:Quandle} and~\ref{section:Examples} for more examples.

\paragraph{Step 3.} Our definition of quandle products is sufficiently general to encompass all the examples of groups we are interested in, but we need to justify that it is not too general, that it is sufficiently restrictive in order to allow us to extract valuable information about the groups under consideration. 

\medskip \noindent
Our study of quandle products is mainly geometric. After solving the word problem in quandle products (relatively to their factors), we prove that:

\begin{thm}[{Theorem~\ref{thm:QuandleQM}}]
Cayley graphs of quandle products with respect to the union of their factors is a quasi-median graph. 
\end{thm}

\noindent
Median graphs are among the most studied graphs in metric graph theory, and are every well-known in geometric group theory through CAT(0) cube complexes (as a graph is median if and only if it is the one-skeleton of a CAT(0) cube complex). Quasi-median graphs provide a generalisation that allows our graph to contain complete subgraphs but that keep all the good properties of median graphs, including a rich combinatorics of hyperplanes. See Section~\ref{section:QM} for more details. 

\medskip \noindent
As our main contribution to the study of the algebraic structure of quandle products, we deduce from the theory of group actions on quasi-median graphs that:

\begin{thm}[{Corollary~\ref{cor:DecompositionQuandle}}]
A quandle product $\mathfrak{Q}$ of finitely many groups decomposes as
$$G_1 \rtimes (G_2 \rtimes ( \cdots \rtimes G_n))$$
where each $G_i$ is a graph product whose vertex-groups are factors of $\mathfrak{Q}$.
\end{thm}

\noindent
The theorem is a consequence of a more precise statement (Theorem~\ref{thm:SemiDirect}) that allows us to prove various properties for quandle products by induction on the number of factors. Thanks to this principle, and thanks to the theory developed in \cite{QM} for groups acting on quasi-median graphs, we prove that:

\begin{thm}\label{thm:IntroList}
Let $\mathfrak{Q}$ be a quandle product of finitely many groups.
\begin{itemize}
	\item If every factor has solvable word problem, then $\mathfrak{Q}$ has solvable word problem. More generally, parabolic subgroups have solvable membership problem.
	\item If every factor is torsion-free, then $\mathfrak{Q}$ is also torsion-free.
	\item If every factor is orderable, then $\mathfrak{Q}$ is also orderable.
	\item If every factor satisfies the Tits alternative, then $\mathfrak{Q}$ satisfies the Tits alternative.
	\item If every factor has finite asymptotic dimension, then $\mathfrak{Q}$ has finite asymptotic dimension. 
\end{itemize}
Moreover, assuming in addition that $\mathfrak{Q}$ has trivial holonomy:
\begin{itemize}
	\item If every factor acts (metrically) properly on a median graph, then $\mathfrak{Q}$ acts (metrically) properly on a median graph.
	\item If every factor has a locally finite (quasi-)median Cayley graph, then $\mathfrak{Q}$ has a locally finite (quasi-)median Cayley graph. 
	\item If every factor is a-T-menable (resp.\ a-$L^p$-menable for some odd $p \geq 1$), then $\mathfrak{Q}$ is a-T-menable (resp.\ a-$L^p$-menable for some odd $p \geq 1$).
	\item If every factor is finitely generated, then the inequality between $L^p$-compressions $$\alpha_p(\mathfrak{Q}) \geq \min \left( \frac{1}{p}, \min\limits_{G \text{ factor}} \alpha_p(G) \right)$$ holds for every $p \geq 1$. 
\end{itemize}
\end{thm}

\noindent
The statement summarises Corollaries~\ref{cor:WP}, \ref{cor:FactorsEmb}, \ref{cor:CombApplications}, Theorems~\ref{thm:ShortestWords} and~\ref{thm:NPC}, Theorem~\ref{thm:MedianCayl} and Lemma~\ref{lem:GeneratingNoHolo}, Theorem~\ref{thm:Compression}. We refer the reader to these statements for more details.

\medskip \noindent
Some words about the holonomy mentioned in Theorem~\ref{thm:IntroList} are in order. Given a quandle product $\mathfrak{Q}$, associated to a set $\mathbb{S}$ and a collection of subsets $\mathrm{P}(\mathbb{S})$, the actions $G_S \curvearrowright \bigsqcup_{R \subset S} G_R$ provide a system of isomorphisms between the factors of $\mathfrak{Q}$. For each factor $G$, we get a group of automorphisms $\mathrm{Hol}(G) \leq \mathrm{Aut}(G)$ by composing these isomorphisms, which we refer to as the \emph{holonomy} at $G$. Most of the examples mentioned above have trivial holonomy, but not all. For instance, the oriented cactus group is a quandle product of infinite cyclic groups with non-trivial holonomy.

\medskip \noindent
The idea to keep in mind is that quandle products of finitely many groups with trivial holonomy have a very nice behaviour. An idea that we often meet is:

\medskip \noindent
\textbf{Leitmotiv.} \emph{A group property stable under direct and free products is stable under quandle products of finitely many groups with trivial holonomy. }

\medskip \noindent
Allowing infinitely many factors may produce more exotic groups, such as permutational wreath products and Thompson's group $F$. It is not clear to us whether the family of quandle products of possibly infinitely many factors define a meaningful family of groups, algebraically speaking. Contrary to quandle products of finitely many groups, it may be possible that this includes examples of groups that are too different from each other. 

\medskip \noindent
Allowing holonomy may also create more exotic groups. For instance, the Baumslag-Solitar group $\mathrm{BS}(1,2)$ is a quandle product of $\mathbb{Z}[1/2]$ and $\mathbb{Z}$. But quandle products with holonomy are more complicated than quandle products without holonomy in the same way that semidirect products are more complicated than direct products. In full generality, we cannot say anything really interesting about semidirect products, we need to impose restrictions on the corresponding action. Similarly, it should be possible to say something interesting about quandle products with non-trivial but controlled holonomy.

\paragraph{A word about trickle groups.} Motivated by similarities in solutions to the word problem in cactus groups and graph products of cyclic groups, \cite{Trickle} investigates the following question: which group presentations of the form
$$\langle x_i \mid x_i^{p_i}=1, \ x_ix_j= x_jx_i' \rangle$$
 admit a solution to the word problem similar to cactus groups and graph products of cyclic groups? In \cite{Trickle}, the authors propose a list of axioms that group presentations should satisfy in order to have a positive answer to this question (see Section~\ref{section:Trickle}). Groups given by such presentations are called \emph{trickle groups}. 

\medskip \noindent
It turns out that trickle groups are quandle products. In fact, it is possible to characterise exactly which quandle products are trickle groups.

\begin{thm}[{Proposition~\ref{prop:QuandleVsTrickle}}]\label{thm:IntroTrickle}
A group is a trickle group if and only if it is a quandle product of cyclic groups with trivial holonomy.
\end{thm}

\noindent
Thus, quandle products with trivial holonomy generalise trickle groups in the same way that graph products generalise graph products of cyclic groups. As a consequence of Theorem~\ref{thm:IntroTrickle}, one gets an alternative perspective on trickle groups through the prism of quandle products.

\medskip \noindent
The next statement summarises what can be deduced about trickle groups from our results about quandle products.

\begin{thm}[{Corollary~\ref{cor:Trickle}}]
Let $\mathfrak{T}$ be a trickle group with finitely many generators.
\begin{itemize}
	\item[(i)] $\mathfrak{T}$ has a locally finite quasi-median Cayley graph. Consequently, $\mathfrak{T}$ is cocompactly cubulable. 
	\item[(ii)] For every prime $p \geq 2$, $\mathfrak{T}$ contains an element of order $p$ if and only if it has a generator whose order is divisible by $p$. Consequently, $\mathfrak{T}$ is torsion-free if and only if its generators all have infinite order. 
	\item [(iii)] $\mathfrak{T}$ decomposes as $$G_1 \rtimes ( G_2 \rtimes (\cdots \rtimes G_n))$$ where each $G_i$ is a graph product of finite groups.
	\item[(iv)] If all the generators of $\mathfrak{T}$ have infinite order, then $\mathfrak{T}$ is orderable.
	\item[(v)] For all conjugates of infinite-order generators $g_1, \ldots, g_m$, there exists $N \geq 1$ such that $\{g_1^N, \ldots, g_m^N\}$ is the basis of a right-angled Artin group.
\end{itemize}
\end{thm}

\noindent
Items (i) and (iv) answer (partially) \cite[Questions~(1), (2), and~(4) p.\ 6]{Trickle} and (ii) strengthens \cite[Theorem~2.16]{Trickle}.

\section{Basics on quandle products}

\subsection{Oposets}

\noindent
In order to define quandle products, we need to a set endowed with a partial order and with a relation of \emph{orthogonality} (or \emph{disjointness}). The next definition captures the notation that will be relevant for us.

\begin{definition}
An \emph{oposet} $(I,\leq, \perp)$ is a poset $(I,\leq)$ endowed with a symmetric irreflexive relation $\perp$ (\emph{orthogonality}) satisfying:
\begin{itemize}
	\item for all $x,y \in I$, if $x \perp y$ then $x$ and $y$ are not $\leq$-comparable;
	\item for all $x,y,z \in I$ if $x \leq y$ and $y \perp z$ then $x \perp z$. 
\end{itemize}
\end{definition}

\noindent
The typical example to keep in mind is, given a set $S$, the oposet $(\mathfrak{P}(S) \backslash \{ \emptyset \}, \subseteq, \perp \text{disjointness})$. Notice that, for every $I \subset \mathfrak{P}(S) \backslash \{ \emptyset\}$, one gets an oposet by restricting $\subseteq$ and $\perp$ to $I$. In fact, as shown by our next lemma, every oposet can be realised in this way. However, the more abstract point of view given by oposets may be convenient in practice. See Example~\ref{ex:Oposet} below, as well as Section~\ref{section:Examples} for other examples. 

\begin{lemma}\label{lem:Oposet}
Let $(I, \leq , \perp)$ be an oposet. There exists an injective map $\iota : I \hookrightarrow \mathfrak{P}(I) \backslash \{ \emptyset \}$ such that 
\begin{itemize}
	\item for all $x,y \in I$, $x \leq y$ if and only if $\iota(x) \subseteq \iota(y)$;
	\item for all $x,y \in X$, $x \perp y$ if and only if $\iota(x) \cap \iota(y) = \emptyset$.
\end{itemize}
\end{lemma}

\begin{proof}
Define an \emph{atom} as a subset of $I$ whose elements are pairwise non-orthogonal. For all atom $A$ and element $i \in I$, $A$ is an \emph{atom of $i$} if there exists $j \in A$ satisfying $j \leq i$. For instance, for every $i \in I$, $\{i\}$ is an atom of $i$. So every element of $I$ has at least one atom. Consider the map
$$\iota : \left\{ \begin{array}{ccc} I & \to & \mathfrak{P}(I) \backslash \{ \emptyset \} \\ i & \mapsto & \{ \text{atoms of } i\} \end{array} \right..$$
Let $x,y \in I$ be two elements. First, assume that $x \leq y$. If $A$ is an atom of $x$, then there exists $i \in A$ such that $i \leq x$. Necessarily, $i \leq y$, so $A$ is also an atom of $y$. Hence $\iota(x) \subset \iota(y)$. Next, if $x \perp y$ and  if $A$ is atom of both $x$ and $y$, then there exist $i,j \in A$ satisfying $i \leq x$ and $j \leq y$. But, since $x \perp y$, we must have $i \perp j$, contracting the fact that $A$ is an atom. Hence $\iota(x) \cap \iota(y)= \emptyset$.

\medskip \noindent
Conversely, if $x$ and $y$ are not $\leq$-comparable, then $\{x\}$ is an atom of $x$ but not an atom of $y$ and $\{y\}$ is an atom of $y$ but not an atom of $x$, so $\iota(x)$ and $\iota(y)$ are not $\subseteq$-comparable. Finally, if $x$ and $y$ are not $\perp$-comparable, then $\{x,y\}$ is an atom of both $x$ and $y$, hence $\iota(x) \cap \iota(y) \neq \emptyset$. 

\medskip \noindent
Thus, we have proved that $\iota$ satisfies the two items of our lemma. Notice that the injectivity of $\iota$ follows. Indeed, for all $x,y \in I$, if $\iota(x)= \iota(y)$, then we must have $x \leq y$ and $y \leq x$, hence $x=y$. 
\end{proof}

\begin{ex}\label{ex:Oposet}
Let $G$ be a group. For all subgroups $H$ and $K$, we note $H \leq_f K$ if $H$ is a finite-index subgroup of $K$; and $H \perp K$ if $\langle H, K \rangle$ decomposes as the free product $H \ast K$. Then $(\text{non-trivial subgroups of } G, \leq_f, \perp)$ is an oposet. 
\end{ex}

\subsection{Quandle systems and groups}\label{section:Quandle}

\noindent
Recall that a groupoid is a small category all of whose morphisms are isomorphisms. In practice, a groupoid can be thought of as a set endowed with a product that satisfies all the properties of a group law but that is not defined for all pairs of elements. For instance, a disjoint union of groups is a groupoid in which the product is well-defined only between two elements of the same group. Disjoint unions of groups are the only groupoids that we will consider. We also recall that the notion of (iso)morphisms naturally extends to groupoids. Then, an action $G \curvearrowright \mathrm{Gr}$ of a group $G$ on a groupoid $\mathrm{Gr}$ refers to a morphism $G \to \mathrm{Aut}(\mathrm{Gr})$.

\medskip \noindent
We are now ready to define \emph{quandle products of groups}, to which this article is dedicated. 

\begin{definition}
A \emph{quandle system} $(\mathcal{I}, \mathcal{G}, \mathcal{A})$ is the data of 
\begin{itemize}
	\item[(a)] an oposet $\mathcal{I}:= (I,\leq, \perp)$, 
	\item[(b)] a collection of groups $\mathcal{G}:= \{ G_i \mid i \in I \}$,
	\item[(c)] and a collection of actions $\mathcal{A}:= \left\{ G_i \curvearrowright \text{groupoid } \bigsqcup_{j< i} G_j \mid i \in I \right\}$
\end{itemize}
satisfying the two following conditions:
\begin{itemize}
	\item[(1)] for every $i \in I$, the action $G_i \curvearrowright \mathrm{Obj} \left( \bigsqcup_{j <i} G_j \right)$ preserves $\leq$ and $\perp$;
	\item[(2)] The equality $c \ast (b \ast a) = (c \ast b) \ast (c \ast a)$ holds for all $ j < k$, $b \in G_j$, $c \in G_k$, and $a \in \bigsqcup_{i< j } G_i$.
\end{itemize}
The \emph{quandle product} $\mathfrak{Q}(\mathcal{I},\mathcal{G},\mathcal{A})$ is
$$\left( \underset{i \in I}{\ast} G_i \right) / \left\langle\left\langle \begin{array}{l} ab=ba \text{ for all } i,j \in I \text{ satisfying } i\perp j \text{ and } a \in G_i, b\in G_j \\ ab=b( b\ast a) \text{ for all } i,j \in I \text{ satisfying } i<j \text{ and } a \in G_i, b \in G_j \end{array} \right\rangle\right\rangle.$$
\end{definition}

\noindent
We refer to the condition (2) as the \emph{quandle relation}, as motivated by the algebraic structure named ``quandle''. (See for instance \cite{MR2885229, MR2002606} for more information about quandles and their use in knot theory.) The central role played by this relation in our formalism justifies the choice to name our products ``quandle products''. 

\medskip \noindent
\textbf{Convention.} In the rest of the article, we will always assume that the factors of our quandle products are non-trivial.

\medskip \noindent
Notice that there is no loss of generality with this convention. Indeed, given an arbitrary quandle system $(\mathcal{I}, \mathcal{G}, \mathcal{A})$, we can consider the suboposet $\mathcal{I}_0$ of $\mathcal{I}$ given by $I_0:= \{ i \in I \mid G_i \neq \{1\}\}$, and we can restrict $\mathcal{G}$ and $\mathcal{A}$ to the subcollections $\mathcal{G}_0$ and $\mathcal{A}_0$ induced by $I_0$. Then, $(\mathcal{I}_0, \mathcal{G}_0, \mathcal{G}_0)$ is a quandle system that defines a quandle product defined by the same presentation as $\mathfrak{Q}(\mathcal{I}, \mathcal{G}, \mathcal{A})$. 

\medskip \noindent
Various examples of quandle products will be given in Section~\ref{section:Examples}. But let us mention at least few easy examples to get more familiar with the notion. 

\begin{ex}\label{ex:GP}
Consider an oposet $\mathcal{I}:=(I, = , \perp)$ and a collection of groups $\mathcal{G}:= \{ G_i \mid i \in I\}$. Then $(\mathcal{I}, \mathcal{G}, \emptyset)$ is a quandle system, and the corresponding quandle product $\mathfrak{Q}$ is
$$\left( \underset{i \in I}{\ast} G_i \right) / \left\langle\left\langle ab=ba \text{ for all } i,j \in I \text{ satisfying } i\perp j \text{ and } a \in G_i, b\in G_j \right\rangle\right\rangle.$$
Thus, $\mathfrak{Q}$ coincides with the graph product $\Gamma \mathcal{G}$ where $\Gamma$ denotes the graph whose vertex-set is $I$ and whose edges connect two $i,j \in I$ whenever $i \perp j$. 
\end{ex}

\begin{ex}\label{ex:SemiDirect}
Consider the oposet $\mathcal{I}:= ( \{1,2 \}, \leq , \emptyset )$, a collection of groups $\mathcal{G}:= \{ G_1 , G_2 \}$, and an action $\mathcal{A}:= \{ G_2 \curvearrowright G_1\}$. Then $(\mathcal{I}, \mathcal{G}, \emptyset)$ is a quandle system, and the corresponding quandle product $\mathfrak{Q}$ is
$$\left( G_1 \ast G_2 \right) / \left\langle\left\langle b^{-1}ab=b \ast a \text{ for all } a \in G_1, b \in G_2 \right\rangle\right\rangle.$$
In other words, $\mathfrak{Q}$ coincides with the semidirect product $G_1 \rtimes G_2$ given by our action $G_2 \curvearrowright G_1$. 
\end{ex}

\begin{ex}\label{ex:Permutations}
Let $S$ be a set and $\mathfrak{S} \subset \mathrm{Sym}(S)$ a collection of permutations. Consider the oposet $\mathcal{I}:= ( S \sqcup \mathfrak{S}, \leq, \emptyset)$ where $\leq$ denotes the smallest partial order for which every element of $S$ is smaller than every element of $\mathfrak{S}$. Let $\mathcal{G}$ be a collection of infinite cyclic groups indexed by $S \sqcup \mathfrak{S}$ and let $\mathcal{A}:= \{ G_\sigma \curvearrowright \bigsqcup_{s \in S} G_s \mid \sigma \in \mathfrak{S}\}$ be the collection of actions where each $G_\sigma$ permutes the $G_s$ according to the action of $\sigma$ on $S$. Then $(\mathcal{I}, \mathcal{G}, \mathcal{A})$ is a quandle system whose corresponding quandle product is given by the presentation
$$\langle S \cup \mathfrak{S} \mid s \cdot \sigma = \sigma \cdot \sigma(s) \text{ for all } s \in S, \sigma \in \mathfrak{S} \rangle.$$ 
\end{ex}

\begin{ex}\label{ex:Quandle}
Let $G$ be a group and $N_1 \leq \cdots \leq N_{n-1} \leq N_n= G$ a chain of normal subgroups. For all $g,h \in G$, write $g \leq h$ when there exist $1 \leq i \leq j \leq n$ such that $g \in N_i$ and $h \in N_j$. Consider the oposet $\mathcal{I}:= (G, \leq , \emptyset)$, a collection $\mathcal{G}$ of infinite cyclic groups $\langle z_g \rangle$ indexed by $g \in G$, and the collection of actions $\mathcal{A}:= \{ G_g \curvearrowright \bigsqcup_{h<g} G_h \mid g \in G \}$ such that each $z_g$ permutes the $z_h$ as follows: if $g \in N_i$ and $h \in N_j$ for some $j<i$, then $z_g \ast z_h : =s_{ghg^{-1}}$. Notice that, for all $1 \leq i <j<k \leq n$ and $a \in N_i$, $b \in N_j$, $c \in N_k$, we have
$$\begin{array}{lcl} (z_c \ast z_b) \ast (z_c \ast z_a) & = & z_{cbc^{-1}} \ast z_{cac^{-1}} = z_{cbc^{-1} \cdot cac^{-1} \cdot cb^{-1}c^{-1}} \\ \\ & = & z_{cbab^{-1}c^{-1}} = z_c \ast z_{bab^{-1}}= z_c \ast (z_b \ast z_a). \end{array}$$
Thus, the quandle relation is satisfied, proving that $(\mathcal{I}, \mathcal{G}, \mathcal{A})$ is a quandle system. The corresponding quandle product admits
$$\langle z_g, \ g \in G \mid z_gz_h= z_h z_{hgh^{-1}} \text{ for all } i<j \text{ and } g \in N_i, h \in N_j \rangle$$
as a presentation. 
\end{ex}

\begin{ex}\label{ex:Push}
Consider the poset $\mathcal{I}:= (\mathbb{R}, \leq, \emptyset)$ and fix a multiplicative subgroup $\Lambda \leq (\mathbb{R}, \times)$. Our collection of groups will be $\mathcal{G}:= \{ G_x = \Lambda \mid x \in \mathbb{R} \}$. For all $x \in \mathbb{R}$ and $\lambda \in \Lambda$, define the homeomorphism
$$D_x(\lambda) : t \mapsto \left\{ \begin{array}{cl} t & \text{if } t\geq x \\ x + \lambda (t-x) & \text{if } t<x \end{array} \right..$$
In other words, $D_x(\lambda)$ fixes $[x,+ \infty)$ pointwise and dilates $(- \infty,x]$ by a factor $\lambda$. Our collection of actions $\mathcal{A}:= \{G_x \curvearrowright \bigsqcup_{y<x} G_y, x \in \mathbb{R}\}$ is defined by making $G_x$ permute the $G_y$, $y < x$, according to the action $D_x : \Lambda \to \mathrm{Homeo}(\mathbb{R})$ on $(- \infty,x)$. More formally, 
$$\left\{ \begin{array}{l} \mu_y \ast x = D_y(\mu)(x) \\ \mu_y \ast \lambda_x = \lambda_ {D_y(\mu)}(x) \end{array} \right. \text{ for all } y<x \text{ and } \lambda, \mu \in \Lambda,$$
where, for all $z \in \mathbb{R}$ and $\xi \in \Lambda$, we denote by $\xi_z$ the element $\xi$ thought of as an element of $G_z$. For all $x<y<z$ and $\lambda, \mu \in \Lambda$, we find that 
$$\lambda_z \ast (\mu_y \ast x) = D_z(\lambda) \circ D_y(\mu)(x) = D_z(\lambda)( y+\mu(x-y)) = z + \lambda ( y+ \mu (x-y) -z)$$
coincides with
$$\begin{array}{lcl} (\lambda_z \ast \mu_y) \ast (\lambda_z \ast x) & = & D_{D_z(\lambda)(y)}(\mu) \circ D_z(\lambda)(x) = D_{z+\lambda(y-z)}(\mu) (z+ \lambda(x-z)) \\ \\ & = & z + \lambda(y-z) + \mu \left( z+ \lambda(x-z) -z-\lambda(y-z) \right). \end{array}$$
Thus, the quandle relation is satisfied, proving that $(\mathcal{I}, \mathcal{G}, \mathcal{A})$ defines a quandle system. When $\Lambda = \langle 2 \rangle$, the corresponding quandle product admits
$$\langle x \in \mathbb{R} \mid xy = y ( 2x -y) \text{ for all } x<y \rangle$$
as a presentation.
\end{ex}

\begin{ex}
Let $(\mathcal{I}, \mathcal{G}, \mathcal{A})$ and $(\mathcal{J}, \mathcal{H}, \mathcal{B})$ be two quandle systems. Defining the \emph{oproduct} $\mathcal{I} \otimes \mathcal{J}$ of our two oposets as $(I \sqcup J, \leq_I \sqcup \leq_J, \perp_I \sqcup \perp_J)$, one can define a quandle system as $( \mathcal{I} \otimes \mathcal{J}, \mathcal{G} \sqcup \mathcal{H}, \mathcal{A} \sqcup \mathcal{B})$. The corresponding quandle product naturally decomposes as a free product:
$$\mathfrak{Q}( \mathcal{I} \otimes \mathcal{J}, \mathcal{G} \sqcup \mathcal{H}, \mathcal{A} \sqcup \mathcal{B}) \simeq \mathfrak{Q}(\mathcal{I}, \mathcal{G}, \mathcal{A}) \ast \mathfrak{Q}(\mathcal{J}, \mathcal{H}, \mathcal{B}).$$
Similarly, define the \emph{oposum} $\mathcal{I} \oplus \mathcal{J}$ as $(I \sqcup J, \leq_I \sqcup \leq_J, \perp_I \ast \perp_J)$ where $\perp_I \ast \perp_J$ denotes the relation $\perp$ defined as follows: for all $a,b \in I \sqcup J$, $a \perp b$ whenever $a,b \in I$ and $a \perp_I$, or $a,b \in J$ and $a \perp_J b$, or $a \in I$ and $b \in J$, or $a \in J$ and $b \in I$. One can define a quandle system as $( \mathcal{I} \oplus \mathcal{J}, \mathcal{G} \sqcup \mathcal{H}, \mathcal{A} \sqcup \mathcal{B})$, and the corresponding quandle product naturally decomposes as a direct sum:
$$\mathfrak{Q}( \mathcal{I} \oplus \mathcal{J}, \mathcal{G} \sqcup \mathcal{H}, \mathcal{A} \sqcup \mathcal{B}) \simeq \mathfrak{Q}(\mathcal{I}, \mathcal{G}, \mathcal{A}) \times \mathfrak{Q}(\mathcal{J}, \mathcal{H}, \mathcal{B}).$$
\end{ex}

\subsection{Holonomy}\label{section:Holonomy}

\noindent
Typically, in a quandle product, two distinct factors can be conjugate by an element of a third factor. By composing such elementary conjugations, we may find an element of our quandle product that conjugates a factor to itself in a non-trivial way. Algebraically speaking, the normaliser of a factor $G$ may define a non-trivial subgroup of $\mathrm{Aut}(G)$. As motivated in Section~\ref{section:Combination}, quandle products for which such a phenomenon does not occur have a much better behaviour. 

\medskip \noindent
Below, we record this phenomenon a more suitable way for quandle products. Recall that a \emph{system of groups} is a directed graph whose vertices are labelled by groups and whose directed edges are labelled by morphisms between the groups labelling its endpoints. 

\begin{definition}
Let $(\mathcal{I},\mathcal{G},\mathcal{A})$ be a quandle system. Define the system of groups $\mathscr{S}(\mathcal{I}, \mathcal{G}, \mathcal{A})$ by considering all the groups in $\mathcal{G}$ and all the isomorphisms $G_i \to G_{j \ast i}$ given by $h \mapsto g \ast h$ for all $h \in G_i$, $g \in G_j$, $i,j \in I$ satisfying $i<j$. Given an $i \in I$, the \emph{holonomy of $(\mathcal{I},\mathcal{G},\mathcal{A})$ at $i$} is the group $\mathrm{Hol}(i) \leq \mathrm{Aut}(G_i)$ of all the automorphisms of $G_i$ that can be obtained by composing the isomorphisms of $\mathscr{S}(\mathcal{I}, \mathcal{G}, \mathcal{A})$. The system $(\mathcal{I},\mathcal{G},\mathcal{A})$ has \emph{trivial holonomy} if $\mathrm{Hol}(i)= \{ \mathrm{id}\}$ for every $i \in I$. 
\end{definition}

\noindent
In other words, an element of $\mathrm{Hol}(i)$ can be described as a map of the form
$$g \mapsto a_1 \ast (a_2 \ast ( \cdots \ast ( a_n \ast g)))$$
for some $a_1 \in G_{r(1)}, \ldots, a_n \in G_{r(n)}$. Notice that, for such a map, in order to be well-defined on $G_i$, we need the inequalities
$$\left\{ \begin{array}{l} r(n)>i \\ r(n-1) > a_n \ast i \\ r(n-2)> a_{n-1} \ast (a_n \ast i) \\ \vdots \\ r(1) > a_2 \ast ( \cdots \ast (a_{n-1} \ast (a_n \ast i))) \end{array} \right.$$
to be satisfied. And, in order to take values in $G_i$, we also need the equality 
$$a_1 \ast (a_2 \ast ( \cdots \ast (a_n \ast i)))=i$$ 
to hold. 

\medskip \noindent
In Example~\ref{ex:SemiDirect}, $\mathrm{Hol}(2)$ is trivial but $\mathrm{Hol}(1)$ coincides with the image of $G_2$ in $\mathrm{Aut}(G_1)$ given by the action $G_2 \curvearrowright G_1$. The quandle systems given in Examples~\ref{ex:GP}, ~\ref{ex:Permutations}, and~\ref{ex:Quandle} have trivial holonomy. See Section~\ref{section:Examples} for more examples.

\begin{remark}\label{remark:NoHolonomy}
When a quandle system $(\mathcal{I},\mathcal{G},\mathcal{A})$ has trivial holonomy, we can naturally identify two factors $G_i$ and $G_j$ whenever there exists an isomorphism $G_i \to G_j$ obtained by composition of isomorphisms given by the system $\mathscr{S}(\mathcal{I}, \mathcal{G}, \mathcal{A})$, since all such compositions will automatically coincide. Then, the relations
$$ab=b (b\ast a) \text{ for all } i,j \in I \text{ satisfying } i<j \text{ and } a \in G_i, b \in G_j$$
defining the quandle product $\mathfrak{Q}(\mathcal{I},\mathcal{G},\mathcal{A})$ can be rewritten as $ab=ba$ where the first $a$ is thought of as an element of $G_i$ and where the second $a$ is thought of as the same element but in $G_{j \ast i}$. 
\end{remark}

\subsection{Parabolic subgroups}

\noindent
Given a product of groups, it is natural to consider subgroups generated by only some of the factors. 

\begin{definition}
Let $\mathfrak{Q}:= \mathfrak{Q}(\mathcal{I},\mathcal{G},\mathcal{A})$ be a quandle product. A \emph{standard parabolic subgroup} is a subgroup $\langle G_j \mid j \in J\rangle$ for some $J \subset I$. A \emph{parabolic subgroup} is a conjugate in $\mathfrak{Q}$ of a standard parabolic subgroup. 
\end{definition}

\noindent
In the rest of the article, we will use the convenient notation $\langle J \rangle$ for the subgroup $\langle G_j \mid j \in J \rangle$ given by $J \subset I$. Our definition extends parabolic subgroups in graph products \cite{MR3365774}, some subgroups of cactus groups studied in \cite[Section~5]{MR4874027}, and parabolic subgroups in trickle groups \cite{Trickle} (see Section~\ref{section:Trickle}).

\medskip \noindent
First of all, notice that, given a quandle product $\mathfrak{Q}:= \mathfrak{Q}(\mathcal{I}, \mathcal{G}, \mathcal{A})$, we can have $\langle R \rangle = \langle S \rangle$ but $R \neq S$ for some $R,S \subset I$. For instance, in Example~\ref{ex:Permutations}, assuming that $\mathfrak{S}$ acts on $S$ transitively, all the generators $s \in S$ are conjugate under $\mathfrak{S}$, so $\langle \{s\} \cup \mathfrak{S} \rangle = \langle S \cup \mathfrak{S} \rangle$ for every $s \in S$. This motivates the following definition:

\begin{definition}
Let $(\mathcal{I},\mathcal{G},\mathcal{A})$ be a quandle system. A subset $R \subset I$ is \emph{stable} if, for all $i,j \in R$ satisfying $i<j$ and every $g \in G_j$, $g \ast i \in R$.
\end{definition}

\noindent
Clearly, an intersection of stable subsets is stable, which allows to define the \emph{stable closure} of a subset as the intersection of all the stable subsets that contain it. Notice that:

\begin{lemma}\label{lem:StableClosure}
Let $\mathfrak{Q}:= \mathfrak{Q}(\mathcal{I},\mathcal{G},\mathcal{A})$ be a quandle product. For all $S \subset I$, the equality $\langle S \rangle = \langle \mathrm{sc}(S) \rangle$ holds, where $\mathrm{sc}(S)$ denotes the stable closure of $S$. 
\end{lemma}

\begin{proof}
Our goal is to define by transfinite induction an increasing (generalised) sequence of subsets $S \subset S_\alpha \subset \mathrm{sc}(S)$ such that $\langle S_\alpha \rangle = \langle S \rangle$ for every $\alpha$ and such that $S_\alpha  = \mathrm{sc}(S)$ for some sufficiently large ordinal $\alpha$. 

\medskip \noindent
Start by setting $S_0:=S$. Now, assume that $S_\alpha$ has been already defined for some ordinal $\alpha$ and that $\langle S_\alpha \rangle = \langle S \rangle$. If $S_\alpha$ is stable, we stop our sequence. Otherwise, we can find $i,j \in S_\alpha$ and $g \in G_j$ such that $i< j$ and $g \ast i \notin S_\alpha$. Then, set $S_{\alpha+1}:= S_\alpha \cup \{ g \ast i \}$. Notice that, since $G_{g \ast i} = g G_i g^{-1}$, we still have $\langle S_{\alpha+1} \rangle = \langle S \rangle$. Next, assume that $S_\alpha$ has been already defined for every $\alpha < \lambda$ for some limit ordinal $\lambda$, and that $\langle S_\alpha \rangle = \langle S \rangle$ for every $\alpha < \lambda$. Then, we set $S_\lambda:= \bigcup_{\alpha< \lambda} S_\alpha$. It is clear that $\langle S_\lambda \rangle = \langle S \rangle$. 

\medskip \noindent
Since $(S_\alpha)_\alpha$ is increasing, it has to stop for some sufficiently large ordinal $\kappa$, i.e.\ $S_\kappa= \mathrm{sc}(S)$. Then, $\langle S \rangle = \langle S_\kappa \rangle = \langle \mathrm{sc}(S) \rangle$.  
\end{proof}

\noindent
Thus, there is no loss of generality in considering only parabolic subgroups given by stable subsets. We will see in Section~\ref{section:WP} that, given two stable subsets $R$ and $S$, the equality $\langle R \rangle = \langle S \rangle$ holds if and only if $R=S$ (Corollary~\ref{cor:FactorsEmb}). 

\medskip \noindent
For more information on parabolic subgroups, we refer the reader to Corollary~\ref{cor:FactorsEmb} and Proposition~\ref{prop:QuasiRetract} below. In particular, we show that parabolic subgroups are quandle products themselves.

\section{Word problem}\label{section:WP}

\noindent
In this section, our goal is to solve the word problem in quandle products. More precisely, we want to determine when a product of elements coming from factors is trivial. Theorem~\ref{thm:NormalForm} is the main result of this section. We also refer to Section~\ref{section:FirstApplications} for more information on the word problem. As an application, we will be able to deduce that the word problem is algorithmically solvable in quandle products of finitely many groups with solvable word problem (Corollary~\ref{cor:WP}); and that parabolic subgroups in quandle products are themselves quandle products (Corollary~\ref{cor:FactorsEmb}). 

\medskip \noindent
From now on, we fix a quandle system $(\mathcal{I}, \mathcal{G}, \mathcal{A})$. By a \emph{word}, we will refer to a word written over $\bigsqcup_{i \in G} G_i \backslash \{1\}$. Clearly, every element of our quandle product $\mathfrak{Q}:=\mathfrak{Q}(\mathcal{I}, \mathcal{G}, \mathcal{A})$ can be represented by such a word. 

\begin{definition}
Let $w:=s_1 \cdots s_n$ be a word where each letter $s_i$ belongs to some $G_{r(i)}$. A word $w'$ is obtained from $w$ by
\begin{itemize}
	\item a \emph{$\perp$-commutation}, if there exist $1 \leq i < n$ such that $r(i) \perp r(i+1)$ and $w'=s_1 \cdots s_{i-1}s_{i+1}s_i s_{i+2} \cdots s_n$;
	\item a \emph{fusion}, if there exist $1 \leq i <n$ such that $r(i)=r(i+1)$ and $w'=s_1 \cdots s_{i-1}s_{i+2} \cdots s_n$ if $s_is_{i+1} = 1$ in $G_{r(i)}$ and $s_1 \cdots s_{i-1} (s_is_{i+1}) s_{i+2} \cdots s_n$ otherwise;
	\item a \emph{twisted left-commutation}, if there exist $1 \leq i < n$ such that $r(i)<r(i+1)$ and $w'=s_1 \cdots s_{i-1}s_{i+1}(s_{i+1} \ast s_i) s_{i+2} \cdots s_n$. 
\end{itemize}
\end{definition}

\noindent
Clearly, applying a $\perp$-commutation, a fusion, or a twisted left-commutation to a word does not modify the element of $\mathfrak{Q}$ it represents. 

\begin{definition}
A \emph{braid} is an equivalence class of words written over $\bigsqcup_{i \in I} G_i \backslash \{1\}$ with respect to the reflexive-transitive closure of the relation that identifies two words whenever one can be obtained from the other by a $\perp$-commutation. A braid $\beta$ is \emph{rankable} if it has a representative of the form $w_1abw_2$ where $a \in G_i$ and $b \in G_j$ for some $i \leq j$. A braid represented by $w_1(ab)w_2$ if $i=j$ or by $w_1b(b\ast a)w_2$ if $i<j$ is obtained from $\beta$ by applying a \emph{ranking}. A braid that is not rankable is \emph{ranked}. 
\end{definition}

\noindent
The motivation behind the definition of a braid comes from cactus groups. As explained in the introduction, every element of a cactus group can be represented by a braid-like picture. Starting from a braid, the order in which we apply two twists with disjoint supports to the bottom strands does not affect the braid we obtain. This observation amounts to saying that two products of generators yield to the same braid whenever one can be obtained from the other by a $\perp$-commutation. Still in cactus groups, in \cite{MR4874027} we put braids in normal form by pushing to the top each twist as far as possible. In other words, we have proved that every element of a cactus group is represented by a unique ranked braid. As we show now, this result extends to arbitrary quandle products. 

\begin{thm}\label{thm:NormalForm}
An element of $\mathfrak{Q}(\mathcal{I}, \mathcal{G}, \mathcal{A})$ is represented by a unique ranked braid. 
\end{thm}

\begin{proof}
Fix an element $g$ of our quandle product and let $\Gamma(g)$ denote the oriented graph whose vertices are the braids representing $g$ and whose oriented edges connect a braid $\beta_1$ to a braid $\beta_2$ whenever $\beta_2$ can be obtained from $\beta_1$ by applying a ranking. Given two vertices $\beta_1,\beta_2 \in V(\Gamma(g))$, we write $\beta_1 \to \beta_2$ (resp.\ $\beta_1 \overset{\ast}{\to} \beta_2$) if there exists an oriented edge (resp.\ an oriented path) in $\Gamma(g)$ connecting $\beta_1$ to $\beta_2$. Formally, $\Gamma(g)$ defines a rewriting system and our goal is to prove that it is \emph{locally confluent} and \emph{terminating}. 

\begin{claim}\label{claim:Terminating}
$\Gamma(g)$ is terminating, i.e.\ every sequence $\beta_1 \to \beta_2 \to \cdots$ terminates.
\end{claim}

\noindent
It suffices to introduce a complexity $\chi(\cdot) \in \mathbb{N}$ satisfying $\chi(\mu)< \chi (\nu)$ for all braids $\mu \to \nu$. Let $\beta \in \Gamma(g)$ be a braid. First, given a word $w=s_1 \cdots s_n$ where each letter $s_i$ belongs to some factor $G_{r(i)}$, define the complexity
$$\chi(w):= \text{length of } w + \# \{ \text{pairs } a<b \mid r(a)<r(b) \}.$$
Clearly, applying a $\perp$-commutation to a word does not modify its complexity. Consequently, it makes sense to define the complexity $\chi(\beta)$ of a braid as the common complexity of its representatives. One easily sees that, if $\mu,\nu \in V(\Gamma(g))$ are two braids satisfying $\mu \to \nu$, then $\chi(\mu)< \chi(\nu)$ as desired. This proves Claim~\ref{claim:Terminating}.

\begin{claim}\label{claim:LocallyConfluent}
$\Gamma(g)$ is locally confluent, i.e.\ for all $\beta_0,\beta_1,\beta_2 \in V(\Gamma(g))$ satisfying $\beta_0 \to \beta_1, \beta_2$, there exists $\beta \in V(\Gamma(g))$ such that $\beta_1, \beta_2 \overset{\ast}{\to} \beta$.
\end{claim}

\noindent
Let $\beta_0,\beta_1,\beta_2 \in V(\Gamma(g))$ be three braids satisfying $\beta_0 \to \beta_1, \beta_2$ but $\beta_1 \neq \beta_2$. Because there are two types of oriented edges in $\Gamma(g)$, corresponding to fusions and twisted left-commutations, we need to distinguish four cases.

\medskip \noindent
\emph{Case 1:} Assume that $\beta_1$ and $\beta_2$ are both obtained from $\beta_0$ by a fusion. If the two pairs of letters that are merged in order to get $\beta_1$ and $\beta_2$ from $\beta_0$ do not intersect, then $\beta_0$ has a representative $w_1abw_2pqw_3$, where $a,b,p,q$ are letters, such that, up to switching $\beta_1$ and $\beta_2$, $w_1(ab)w_2pqw_3$ is a representative of $\beta_1$ and $w_1abw_2(pq)w_3$ a representative of $\beta_2$. Then, $\beta_1,\beta_2 \to \beta:= [w_1(ab)w_2(pq)w_3]$. Otherwise, if our two pairs of letters intersect, then $\beta_0$ admits a representative $w_1abcw_2$ such that, again up to switching $\beta_1$ and $\beta_2$, $w_1(ab)cw_2$ is a representative of $\beta_1$ and $w_1a(bc)w_2$ a representative of $\beta_2$. Then, $\beta_1, \beta_2 \to \beta:=[w_1(abc)w_2]$.

\medskip \noindent
\emph{Case 2:} Assume that $\beta_1$ is obtained from $\beta_0$ by a fusion but that $\beta_2$ is obtained from $\beta_0$ by a twisted left-commutation. If the two pairs of letters that are merged or twisted-commuted in order to get $\beta_1$ and $\beta_2$ from $\beta_0$ do not intersect, then $\beta_0$, up to a mirror symmetry, has a representative $w_1abw_2pqw_3$ such that $a,b \in G_i$, $p \in G_r$, $q \in G_s$ for some $r<s$ and such that $w_1(ab)w_2pqw_3$ (resp.\ $w_1abw_2q (q \ast p) w_3$) is a representative of $\beta_1$ (resp.\ $\beta_2$). Then, $\beta_1,\beta_2 \to \beta:= [w_1(ab) w_2 q (q \ast p) w_3]$. 

\medskip \noindent
\begin{minipage}{0.45\linewidth}
\includegraphics[width=0.95\linewidth]{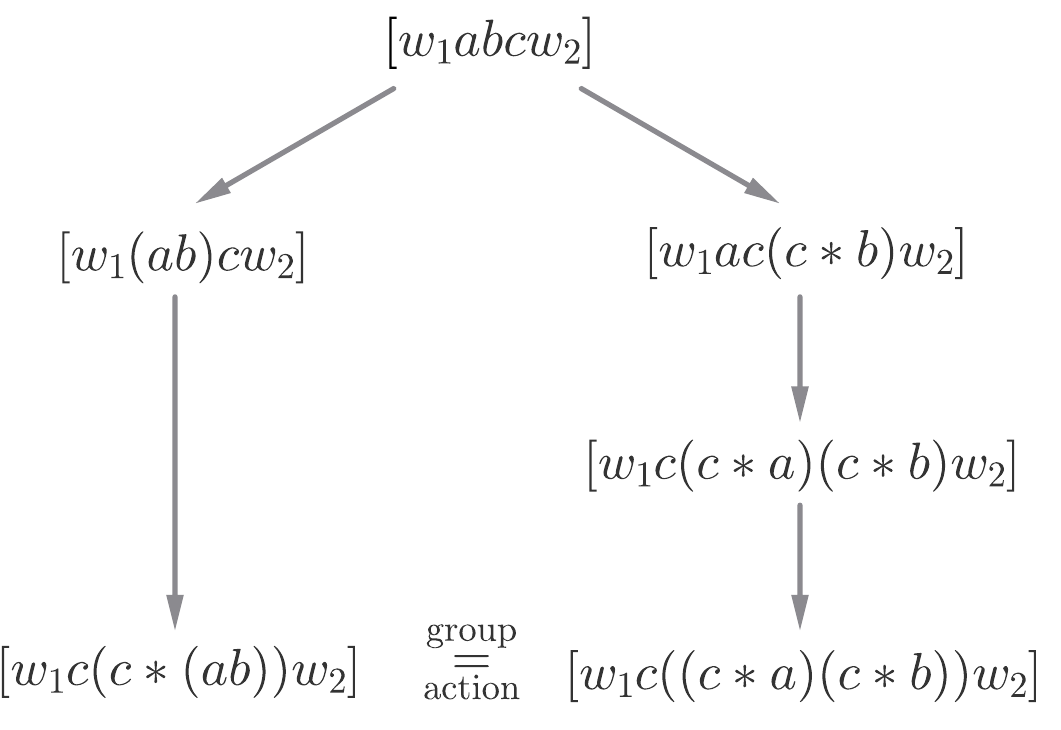}
\end{minipage}
\begin{minipage}{0.55\linewidth}
Otherwise, if our two pairs of letters intersect, then $\beta_0$, up to a mirror symmetry, has a representative $w_1abcw_2$ such that $a,b \in G_i$ and $c \in G_j$ for some $i<j$ and such that $w_1(ab)cw_2$ (resp.\ $w_1ac(c\ast b)w_2$) is a representative of $\beta_1$ (resp.\ $\beta_2$). Then, as justified by the figure on the left, $\beta_1,\beta_2 \overset{\ast}{\to} \beta:= [w_1c (c\ast (ab)) w_2]$. 
\end{minipage}

\medskip \noindent
\emph{Case 3:} Assume that $\beta_2$ is obtained from $\beta_0$ by a fusion but that $\beta_1$ is obtained from $\beta_0$ by a twisted left-commutation. The case is symmetric to the previous one.

\medskip \noindent
\emph{Case 4:} Assume that $\beta_1$ and $\beta_2$ are both obtained from $\beta_0$ by a twisted left-commutation. If the two pairs of letters that are twisted-commuted in order to get $\beta_1$ and $\beta_2$ from $\beta_0$ do not intersect, then $\beta_0$ has a representative $w_1abw_2pqw_3$ such that $a \in G_i$, $b \in G_j$, $p \in G_r$, $q \in G_s$ for some $i<j$, $r<s$ and such that, up to switching $\beta_1$ and $\beta_2$, $w_1b(b\ast a)w_2pqw_3$ (resp.\ $w_1abw_2q (q \ast p) w_3$) is a representative of $\beta_1$ (resp.\ $\beta_2$). Then, $\beta_1,\beta_2 \to \beta:=[w_1b(b\ast a )w_2 q (q\ast p) w_3]$. 

\medskip \noindent
\begin{minipage}{0.55\linewidth}
\includegraphics[width=0.95\linewidth]{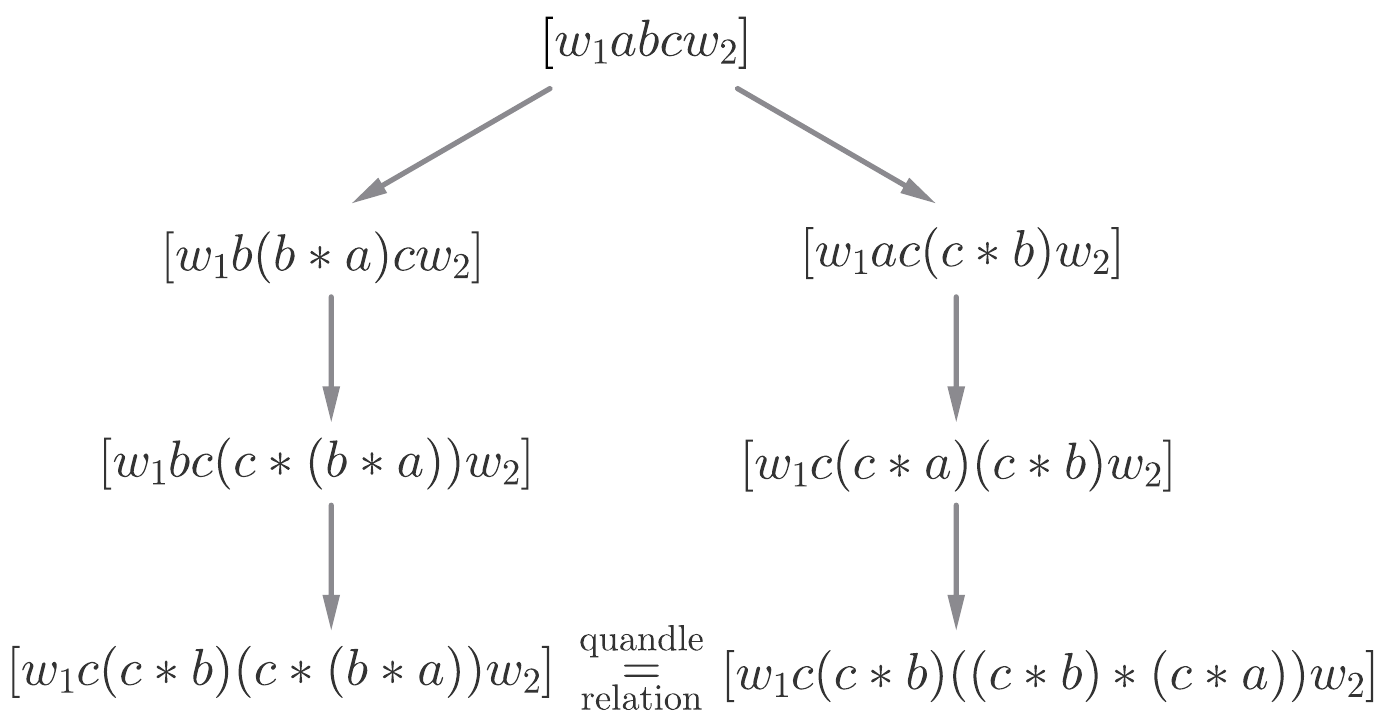}
\end{minipage}
\begin{minipage}{0.45\linewidth}
Otherwise, if our two pairs of letters intersect, then $\beta_0$ admits a representative $w_1abcw_2$ such that $a \in G_i$, $b \in G_j$, $c \in G_k$ for some $i<j<k$ and such that, up to switching $\beta_1$ and $\beta_2$, $w_1b(b\ast a)cw_2$ (resp.\ $w_1ac(c\ast b)w_2$) is a representative of $\beta_1$ (resp.\ $\beta_2$). Then, as justified by the figure on the left, $\beta_1,\beta_2 \overset{\ast}{\to} \beta$ where $\beta:=[w_1 c (c \ast b) (c \ast (b\ast a)) w_2]$.
\end{minipage}

\medskip \noindent
This concludes the proof of Claim~\ref{claim:LocallyConfluent}.

\medskip \noindent
One easily shows that $\Gamma(g)$ locally confluent implies that $\Gamma(g)$ is \emph{confluent}, i.e.\ for all distinct braids $\alpha,\beta,\gamma \in V(\Gamma(g))$ satisfying $\alpha \overset{\ast}{\to} \beta,\gamma$, there exists a braid $\zeta \in V(\Gamma(g))$ such that $\beta,\gamma \overset{\ast}{\to} \zeta$; see for instance \cite[Theorem~1]{MR7372} or \cite[Section~4.1]{MR1127191}. Our theorem now follows by noticing that a vertex in $\Gamma(g)$ corresponds to a ranked braid precisely when there is no oriented edge starting from it. Indeed, the fact that $\Gamma(g)$ is terminating implies that every braid can be turned into a ranked braid by applying rankings; and because $\Gamma(g)$ is confluent, two ranked braids obtained from two braids representing $g$ must coincide.
\end{proof}

\noindent
As an easy consequence of Theorem~\ref{thm:NormalForm}, we can solve the word problem in quandle products:

\begin{cor}\label{cor:WP}
A quandle product of finitely many groups with solvable word problem has solvable word problem. 
\end{cor}

\begin{proof}
Let $\mathfrak{Q}:= \mathfrak{Q}(\mathcal{I}, \mathcal{G}, \mathcal{A})$ be a quandle product with $\mathcal{I}$ finite and such that every group in $\mathcal{G}$ has a solvable word problem. Fix a finite generating set $S_i$ of $G_i$ for every $i \in I$. Clearly, $S:= \bigcup_{i \in I} S_i$ is a generating set of $\mathfrak{Q}$. Given a product of generators, we can pack letters together into syllables whenever they are generators of a common factor, yielding a word in $\mathfrak{Q}$, and a fortiori a braid. Because the word problem is solvable in each factor, our braid can be ranked algorithmically. (A $\perp$-commutation can always be applied algorithmically, once we assume that $\perp$ is part of the (finite) data of our algorithm. But, in order to apply a fusion, we need to be able to recognise when a product of generators in a factor is trivial. A twisted left-commutation can always be applied algorithmically, once we assume that the actions $G_i \curvearrowright \bigcup_{j<i} G_j$ are part of the data of our algorithm, which can be done by recording the $s_i \ast j$ and the $s_i \ast s_j$ (as products of generators of $S_{s_i \ast j}$) for all $i<j$, $s_i \in S_i$, and $s_j \in S_j$.) Then, according to Theorem~\ref{thm:NormalForm}, our initial element of $\mathfrak{Q}$ is trivial if and only if the ranked braid we obtain is empty.
\end{proof}

\noindent
We conclude this section by recording some information about parabolic subgroups that can be extracted from Theorem~\ref{thm:NormalForm}. 

\begin{cor}\label{cor:FactorsEmb}
Let $\mathfrak{Q}:= \mathfrak{Q}(\mathcal{I},\mathcal{G},\mathcal{A})$ be a quandle product and $S \subset I$ a stable subset. 
\begin{itemize}
	\item The subset $S$ induces a suboposet $\mathcal{S}:= (S, \leq , \perp)$ of $\mathcal{I}$ and a quandle system $(\mathcal{S}, \mathcal{G}_{|S}, \mathcal{A}_{|S})$. The identity maps $G_i \to G_i$ ($i \in S$) induce an isomorphism from the quandle product $\mathfrak{Q}_S:= \mathfrak{Q}(\mathcal{S}, \mathcal{G}_{|S}, \mathcal{A}_{|S})$ to the parabolic subgroup $\langle S \rangle$. 
	\item For all stable $R,S \subset I$, $\langle R \rangle \cap \langle S \rangle = \langle R \cap S \rangle$.
	\item For all stable $R,S \subset I$, $\langle R \rangle \leq \langle S \rangle$ if and only if $R \subset S$. Consequently, $\langle R \rangle = \langle S \rangle$ if and only if $R=S$.
	\item If $\mathcal{I}$ is finite and if every group in $\mathcal{G}$ has a solvable word problem, then the membership problem is solvable for $\langle S \rangle$.
\end{itemize}
\end{cor}

\begin{proof}
It is clear that the identity maps $G_i \to G_i$ ($i \in S$) induce a morphism $\varphi : \mathfrak{Q}_S \to \langle S \rangle$, since relations of $\mathfrak{Q}_S$ are sent to relations in $\mathfrak{Q}$. It is also clear that $\varphi$ is surjective, since $\langle S \rangle$ is generated by the $G_i$, $i \in S$. What we have to prove is that $\varphi$ is injective.

\medskip \noindent
So let $g \in \mathfrak{Q}_S$ be a non-trivial element. We can represent $g$ as a ranked braid $\beta$. Since $\varphi$ sends to $\perp$-commuting generators to two $\perp$-commuting generators, it makes sense to refer to the braid $\varphi(\beta)$ in $\mathfrak{Q}$. The key observation is that every fusion or twisted left-commutation that can be applied to $\beta$ corresponds to a fusion or twisted left-commutation that can be applied to $\varphi(\beta)$, and vice versa. Consequently:

\begin{fact}\label{fact:RankedSub}
The image under $\varphi$ of a braid $\beta$ is ranked in $\mathfrak{Q}$ if and only if $\beta$ is ranked in~$\mathfrak{Q}_S$. 
\end{fact}

\noindent
Thus, because $\beta$ is ranked in $\mathfrak{Q}_S$, $\varphi(\beta)$ must be ranked in $\mathfrak{Q}$. A fortiori, $\varphi(\beta)$ cannot be trivial, proving the injectivity of $\varphi$, as desired.

\medskip \noindent
Now, let $R,S \subset I$ be two stable subsets. If $R \subset S$, then it is clear that $\langle R \rangle \leq \langle S \rangle$. Conversely, assume that $\langle R \langle \leq \langle S \rangle$. Fix an $r \in R$ and a non-trivial element $g \in G_r$. Because $g \in \langle S \rangle$, we can represent $g$ as a ranked braid $\beta$ in $\mathfrak{Q}_S$. According to Fact~\ref{fact:RankedSub}, $\beta$ is also a ranked braid in $\mathfrak{Q}$. But $[g]$ is already a ranked braid representing $g$ in $\mathfrak{Q}$. We deduce from Theorem~\ref{thm:NormalForm} that $\beta = [g]$, which implies that $r \in S$. Thus, we must have $R \subset S$, as desired.

\medskip \noindent
Next, it is clear that $\langle R \cap S \rangle \leq \langle R \rangle \cap \langle S \rangle$. Conversely, consider a non-trivial element $g \in \langle R \rangle \cap \langle S \rangle$. We can represent $g$ as a ranked braid $\alpha$ (resp.\ $\beta$) in $\mathfrak{Q}_R$ (resp.\ $\mathfrak{Q}_S$). According to Fact~\ref{fact:RankedSub}, $\alpha$ (resp.\ $\beta$) is also ranked in $\mathfrak{Q}$. It follows from Theorem~\ref{thm:NormalForm} that $\alpha= \beta$. Since all the letters of $\alpha$ (resp.\ $\beta$) must be generators coming from factors labelled by $R$ (resp.\ $S$), it follows that the letters of the ranked braid representing $g$ are generators all coming from factors labelled by $R \cap S$, hence $g \in \langle R \cap S \rangle$. Thus, we have proved that $\langle R \rangle \cap \langle S \rangle \leq \langle R \cap S \rangle$, as desired. 

\medskip \noindent
Finally, assume that $\mathcal{I}$ is finite and that every group in $\mathcal{G}$ has a solvable word problem. For every $i \in I$, fix a finite generating set $S_i$ of $G_i$. Clearly, $S:= \bigcup_{i \in I} S_i$ is a generating set of $\mathfrak{Q}$. Given a product of generators coming from $S$, we can pack letters together into syllables whenever they belong to a common factor, yielding a word in $\mathfrak{Q}$, and a fortiori a braid. As in Corollary~\ref{cor:WP}, because each group in $\mathfrak{G}$ has a solvable word problem, our braid can be ranked. As a consequence of Fact~\ref{fact:RankedSub}, our initial element belongs to $\langle S \rangle$ if and only if the ranked braid we obtain has all its letters coming from factors indexed by $S$. Thus, we can algorithmically decide when an element of $\mathfrak{Q}$ belongs or not to $\langle S \rangle$. 
\end{proof}

\section{Quasi-median geometry}\label{section:QMquandle}

\noindent
Our key tool in the study of quandle products comes from the observation that quandle products admit quasi-median Cayley graph. In Section~\ref{section:QM}, we first recall basic definitions and properties related to quasi-median geometry, and then we prove in Section~\ref{section:Cayley} that Cayley graphs of quandle products with respect to factors are quasi-median. As a first application, we record in Section~\ref{section:FirstApplications} some useful information about the word problem in quandle products that can be deduced.

\subsection{Preliminary}\label{section:QM}

\noindent
Recall that a connected graph $X$ is \emph{median} if, for all vertices $x_1,x_2,x_3 \in X$, there exists a unique vertex $m \in X$, called the \emph{median point}, satisfying
$$d(x_i,x_j)= d(x_i,m)+d(m,x_j) \text{ for all } i \neq j.$$
Typical examples of median graphs are trees, and, more generally, products of trees, including hypercubes. Non-examples of median graphs are bipartite complete graphs $K_{n,m}$ with $n \geq 2$ and $m \geq 3$, and complete graphs $K_n$ with $n \geq 3$. 

\medskip \noindent
Loosely speaking, quasi-median graphs define the smallest reasonable family of graphs that allows (products of) complete graphs but keep most of the good properties of median graphs. A possible definition of quasi-median graphs is the following.

\medskip \noindent
First, given a connected graph $X$ and three vertices $x_1,x_2,x_3 \in X$, say that $y_1,y_2,y_3 \in X$ is a \emph{median triple} if 
$$d(x_i,x_j)= d(x_i,y_i)+d(y_i,y_j)+d(y_j,x_j) \text{ for all } i \neq j.$$
Notice that $x_1,x_2,x_3$ is itself a median triple. The goal is to find a median triple as small as possible. In median graphs, it may be reduced to a single vertex. 

\begin{definition}
A connected graph $X$ is \emph{quasi-median} if any three vertices admit a unique median triple whose convex hull is a product of complete graphs.
\end{definition}

\noindent
Examples of quasi-median graphs include median graphs, of course. In fact, it is possible to characterise median graphs as specific quasi-median graphs:

\begin{thm}[{\cite[(25) p.\ 149]{MR605838}}]\label{thm:MedianQM}
A graph is median if and only if it is quasi-median and $\triangle$-free (i.e.\ without $3$-cycle). 
\end{thm}

\noindent
Other examples of quasi-median graphs include products of complete graphs (also known as Hamming graphs) and block graphs. Bipartite complete graphs $K_{n,m}$ with $n \geq 2$ and $m \geq 3$ are still not quasi-median.

\paragraph{Alternative characterisations.} There exist many equivalent definitions of quasi-median graphs. We refer the reader to \cite{MR1297190} for more information. We mention the following one for future use:

\begin{thm}\label{thm:QMmodular}
A graph is quasi-median if and only if it is weakly modular and does not contain any induced subgraph isomorphic to $K_4^-$ or $K_{3,2}$. 
\end{thm}

\noindent
Here, $K_4^-$ denotes the graph obtained from the complete graph $K_4$ by removing an edge. (So it looks like two triangles glued along a single edge.) Recall that a connected graph is \emph{weakly modular} when it satisfies the following two conditions:
\begin{description}
	\item[(Triangle Condition)] For all vertices $o,x,y \in X$ satisfying $d(o,x)=d(o,y)$ and $d(x,y)=1$, there exists a common neighbour $z$ of $x$ and $y$ such that $d(o,z)=d(o,x)-1$.
\end{description}
\begin{center}
\includegraphics[width=0.5\linewidth]{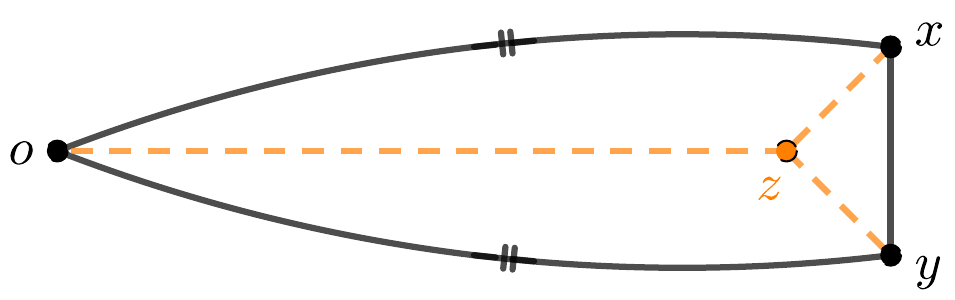}
\end{center}
\begin{description}
	\item[(Quadrangle Condition)] For all vertices $o,x,y,z \in X$ satisfying $d(o,x)=d(o,z)=d(o,y)-1$, $d(y,x)=d(y,z)=1$, and $d(x,z)=2$, there exists a common neighbour $w$ of $x$ and $z$ such that $d(o,w)=d(o,y)-2$.
\end{description}
\begin{center}
\includegraphics[width=0.5\linewidth]{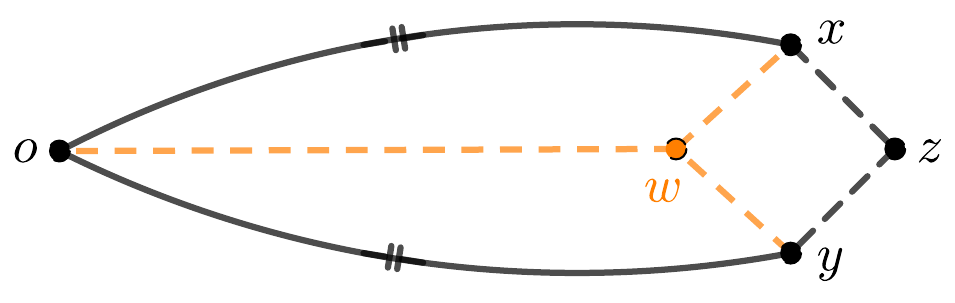}
\end{center}

\noindent
Let us also mention a local-to-global characterisation of quasi-median graphs that will be useful later. 

\begin{thm}
A connected graph $X$ is quasi-median if and only if it satisfies the following conditions:
\begin{itemize}
	\item its square-triangle completion $X^{\square\triangle}$ is simply connected;
	\item $X$ does not contain an induced copy of $K_4^-$;
	\item for every vertex $x \in X$ and all neighbours $y_1,y_2,y_3 \in X$, if the edges $[x,y_1],[x,y_2]$ span a $4$-cycle and $[x,y_2],[x,y_3]$ a $4$-cycle, then they globally span a prism $K_3 \times K_2$;
	\item for every vertex $x \in X$ and all neighbours $y_1,y_2,y_3 \in X$, if the edges $[x,y_1]$, $[x,y_2]$, $[x,y_3]$ pairwise span a $4$-cycle, then they globally span a $3$-cube $Q_3$.
\end{itemize}
\end{thm}

\noindent
Here, the \emph{square-triangle completion} $X^{\square\triangle}$ of a graph $X$ refers to the $2$-complex obtained from $X$ by filling in all the induced $3$- and $4$-cycles with triangles and squares respectively.
\begin{figure}[h!]
\begin{center}
\includegraphics[width=0.6\linewidth]{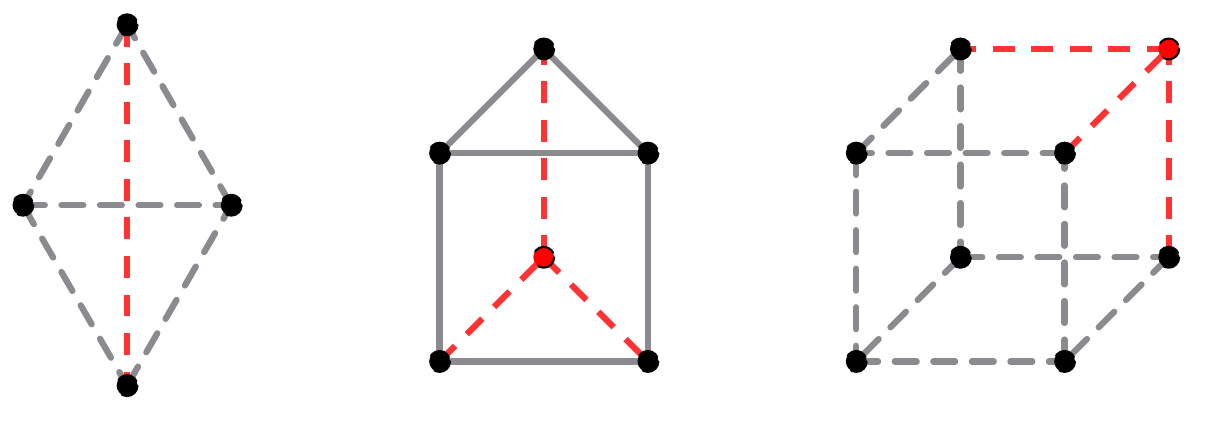}
\caption{From left to right: the no-$K_4^-$ condition, the house condition, and the $3$-cube condition.}
\label{Conditions}
\end{center}
\end{figure}

\begin{proof}
Assume that $X$ is quasi-median. Then the four items from our statement are satisfied as a consequence of \cite[Theorem~2.127]{QM}.

\medskip \noindent
Conversely, assume that the four items from our statement are satisfied by $X$. It follows from \cite[Theorem~1.1]{MR3062742} that $X$ is weakly modular. Therefore, according to Theorem~\ref{thm:QMmodular}, it suffices to verify that $X$ does not contain induced copies of $K_4^-$ and $K_{3,2}$ in order to conclude that $X$ is quasi-median. There is no $K_4^-$ by assumption, and the fourth item for our statement implies that there is no $K_{3,2}$. 
\end{proof}

\paragraph{Paths and geodesics.} Admitting a quasi-median Cayley graph has nice applications to the word problem of the group under consideration because paths and geodesics are connected through elementary transformations. This observation will be useful during our study of quandle products. 

\begin{definition}
Let $X$ be a graph and $\gamma=(x_1, \ldots, x_n)$ a path. A path $\gamma'$ is obtained from $\gamma$ by
\begin{itemize}
	\item \emph{removing a backtrack} if there exists $1 \leq i \leq n-2$ such that $x_i=x_{i+2}$ and $\gamma'= (x_1, \ldots, x_{i-1}, x_{i+2}, x_{i+3} ,\ldots, x_n)$;
	\item \emph{shortening a $3$-cycle} if there exists $1 \leq i \leq n-2$ such that $x_i,x_{i+2}$ are adjacent and $\gamma'= (x_1, \ldots, x_i,x_{i+2},x_{i+3}, \ldots,x_n)$;
	\item \emph{bypassing a $4$-cycle} if there exist $2 \leq i \leq n-1$ and $x_i' \in X$ such that $x_{i-1},x_i,x_{i+1},x_i'$ define an induced $4$-cycle and $\gamma'= (x_1, \ldots, x_{i-1},x_i',x_{i+1}, \ldots, x_n)$. 
\end{itemize}
\end{definition}

\noindent
The following statement is a particular case of \cite[Lemmas~3.3 and~3.4]{Mediangle}.

\begin{prop}\label{prop:PathAndGeodQM}
Let $X$ be a quasi-median graph and $x,y \in V(X)$ two vertices. A path connecting $x$ and $y$ is a geodesic if and only if it cannot be shortened by removing backtracks, shortening $3$-cycles, and bypassing $4$-cycles. Moreover, any two geodesic connecting $x$ and $y$ only differ by bypassing $4$-cycles. 
\end{prop}

\paragraph{Cliques and prisms.} In a quasi-graph, a \emph{clique} refers to a maximal complete subgraph. A \emph{prism} is a subgraph that decomposes as a product of cliques. The number clique-factors is referred to as the \emph{cubical dimension of the prism}. By extension, the \emph{cubical dimension} of a quasi-median graph is the maximal (possibly infinite) cubical dimension of its prisms. 

\begin{lemma}[{\cite[Lemmas~2.16 and~2.80]{QM}}]
In quasi-median graphs, cliques and prisms are gated. 
\end{lemma}

\noindent
Recall that, given a graph $X$, an induced subgraph $Y \subset X$ is \emph{gated} if, for every $x \in X$, there exists a vertex $p \in Y$, refers to as the \emph{gate of $x$ in $Y$}, such that $x$ can be connected to every vertex of $Y$ by geodesics passing through $p$. Notice that, when it exists, the gate is unique. Gatedness can be thought as a strong convexity condition.

\paragraph{Hyperplanes.} Similarly to median graphs, a key tool in the study of quasi-median graphs is given by \emph{hyperplanes}. Roughly speaking, the geometry of a quasi-median graph reduces to the combinatorics of its hyperplanes. 

\begin{definition}
In a quasi-median graph, a \emph{hyperplane} is an equivalence class of edges with respect to the reflexive-transitive closure of the relation that identifies to edges when they belong to a common $3$-cycle or are \emph{parallel} (i.e.\ are opposite sides of a $4$-cycle).
\end{definition}
\begin{figure}[h!]
\begin{center}
\includegraphics[trim=0 16.5cm 10cm 0, clip, width=0.6\linewidth]{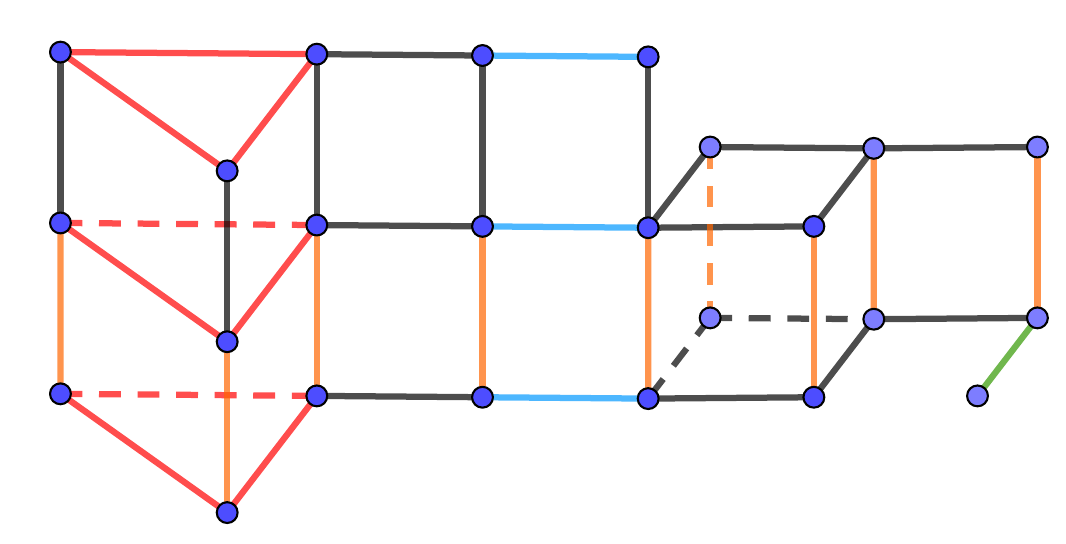}
\caption{Hyperplanes in a quasi-median graph.}
\label{HypEx}
\end{center}
\end{figure}

\noindent
Equivalently, one can think of a hyperplane as an equivalence class of cliques with respect to the reflexive-transitive closure of the relation that identifies two cliques whenever they are \emph{parallel}, i.e.\ are parallel cliques in a prism. 

\medskip \noindent
Let us record some vocabulary related to hyperplanes in quasi-median graphs that will be used later. 

\begin{definition}
Let $X$ be a quasi-median graph.
\begin{itemize}
	\item Let $J$ be a hyperplane. The \emph{carrier of $J$}, denoted by $N(J)$, is the smallest induced subgraph containing all the edges of $J$. A \emph{sector delimited by $J$} is a connected component of the graph $X \backslash\backslash J$ obtained from $X$ by removing the edges in $J$. The \emph{fibres of $J$} are the connected components of $N(J) \backslash \backslash J$.
	\item Two hyperplanes $J$ and $H$ are \emph{transverse} whenever there exist intersecting edges $e \in J$ and $f \in H$ that span a $4$-cycle.
	\item Given a group $G$ acting on $X$, the \emph{rotative-stabiliser} of a hyperplane $J$ is $$\mathrm{stab}_\circlearrowright(J):= \bigcap\limits_{C \subset J \text{ clique}} \mathrm{stab}(C).$$
\end{itemize}
\end{definition}

\noindent
The main properties satisfied by hyperplanes in quasi-median graphs are recorded by the next statement.

\begin{thm}[{\cite[Propositions~2.15 and~2.30]{QM}}]\label{thm:QMbig}
Let $X$ be a quasi-median graph. The following assertions hold.
\begin{itemize}
	\item Every hyperplane delimits at least two sectors.
	\item Carriers, sectors, and fibres are gated.
	\item For every hyperplane $J$, if $C \subset J$ is a clique and $F$ a fibre of $J$, then $N(J)= C \times F$.
	\item A path in $X$ is a geodesic if and only if it crosses each hyperplane at most once.
	\item The distance between two vertices coincides with the number of hyperplanes separating them. 
\end{itemize}
\end{thm}

\noindent
Let us conclude this section with a preliminary lemma that will be useful later.

\begin{lemma}\label{lem:Ladder}
Let $J$ be a hyperplane, $F$ a fibre. For all $x,y \in F$ and for every neighbour $z$ of $y$ such that $\{y,z\} \in J$, there exists a neighbour $w$ of $x$ such that $\{x,w\} \in J$ and such that some ladder connects $(x,w)$ to $(y, z)$. 
\end{lemma}

\noindent
Here, a \emph{ladder} connecting two oriented edges $(a,b)$ and $(p,q)$ refers to a sequence of oriented edges $(x_1,y_1)=(a,b), \ldots, (x_n,y_n)=(p,q)$ such that, for every $1 \leq i \leq n-1$, $(x_i,y_i)$ and $(x_n,y_n)$ are opposite sides of a $4$-cycle. The oriented edges $(x_i,y_i)$ are referred to as the \emph{(oriented) rungs} of the ladder. 
\begin{center}
\includegraphics[width=0.6\linewidth]{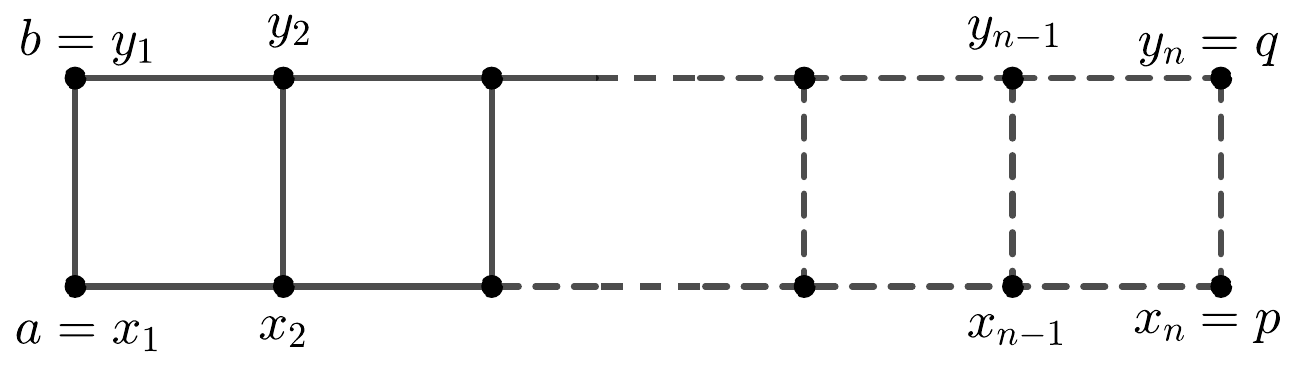}
\end{center}

\noindent
Lemma~\ref{lem:Ladder} is an immediate consequence of the product decomposition of carriers provided by Theorem~\ref{thm:QMbig}.

\subsection{Quasi-median Cayley graphs}\label{section:Cayley}

\noindent
This subsection is dedicated to the proof of the following statement, which will be central in our study of quandle products:

\begin{thm}\label{thm:QuandleQM}
Let $\mathfrak{Q}:= \mathfrak{Q}(\mathcal{I},\mathcal{G},\mathcal{A})$ be a quandle product. The Cayley graph 
$$\mathfrak{M}(\mathcal{I}, \mathcal{G}, \mathcal{A}):= \mathrm{Cayl}(\mathfrak{Q}, \bigcup \mathcal{G})$$ 
is quasi-median of cubical dimension at most 
$$\sup \left\{n \geq 0 \mid \exists i_1, \ldots, i_n \in I \text{ pairwise $<$- or $\perp$-comparable} \right\}.$$
\end{thm}

\noindent
The theorem generalises \cite[Proposition~8.2]{QM}, when applied to graph products, and \cite[Theorem~2.7]{MR4874027}, when applied to cactus groups. 

\medskip \noindent
We start by describing cliques in our Cayley graphs.

\begin{lemma}\label{lem:QuandleClique}
Let $\mathfrak{Q}:= \mathfrak{Q}(\mathcal{I},\mathcal{G},\mathcal{A})$ be a quandle product. Every complete subgraph in $\mathfrak{M}(\mathcal{I}, \mathcal{G}, \mathcal{A})$ is contained in a coset of a factor.
\end{lemma}

\begin{proof}
We start by considering the case of a $3$-cycle. So let $a,b,c$ be three pairwise adjacent vertices in $\mathfrak{M}:= \mathfrak{M}(\mathcal{I}, \mathcal{G}, \mathcal{A})$. Up to translating by an element of $\mathfrak{Q}$, we can assume for simplicity that $a=1$. Let $i,j \in I$ be such that $b \in G_i$ and $c \in G_j$. If $i \neq j$, then $b^{-1}c$ is represented by a ranked braid with two letters, so Theorem~\ref{thm:NormalForm} implies that $b^{-1}c$ is not a generator, contradicting the fact that $b$ and $c$ are adjacent in $\mathfrak{Q}$. Therefore, we must have $i=j$. 

\medskip \noindent
Thus, the conclusion of our lemma holds for $3$-cycles. Now, consider an arbitrary complete subgraph $K$ in $\mathfrak{M}$. If $K$ has $\leq 2$ vertices, there is nothing to prove; and, if $K$ has three vertices, then we already know that the desired conclusion holds. So assume that $K$ contains at least four vertices. Fix three distinct vertices $a,b,c \in K$. Again, we assume for simplicity that $a=1$. We already know that there exists some $i \in I$ such that $b,c \in G_i$. Let $d \in K$ be a vertex distinct from $a,b,c$. Since $b,c,d$ span a $3$-cycle, we know that they must all belong to the same factor. But we already know that $b$ and $c$ belongs to $G_i$ and two distinct factors intersect trivially according to Corollary~\ref{cor:FactorsEmb}. Thus, we also have $d \in G_i$. We conclude that $K \subset G_i$, as desired. 
\end{proof}

\begin{cor}\label{cor:QuandleClique}
Let $\mathfrak{Q}:= \mathfrak{Q}(\mathcal{I},\mathcal{G},\mathcal{A})$ be a quandle product. The cliques in $\mathfrak{M}(\mathcal{I}, \mathcal{G}, \mathcal{A})$ are the cosets of the factors.
\end{cor}

\begin{proof}
Since every complete subgraph is contained in a coset of a factor according to Lemma~\ref{lem:QuandleClique}, and since two distinct factors always intersect trivially according to Corollary~\ref{cor:FactorsEmb}, the desired conclusion follows. 
\end{proof}

\noindent
Then, we describe induced $4$-cycles in our Cayley graphs. The next lemma essentially shows that such cycles only come from the relations in the presentations defining our quandle products. 

\begin{lemma}\label{lem:QuandleCycle}
Let $\mathfrak{Q}:= \mathfrak{Q}(\mathcal{I},\mathcal{G},\mathcal{A})$ be a quandle product. For every induced $4$-cycle $C$ in $\mathfrak{M}(\mathcal{I}, \mathcal{G}, \mathcal{A})$, there exist $g \in \mathfrak{Q}$ and $a \in G_i$, $b \in G_j$ for some $i,j \in I$ satisfying either $i \perp j$ or $i<j$ such that the vertices of $C$ are $g,ga,gb,gab$. Conversely, every $4$-cycle of this form is induced. 
\end{lemma}

\begin{proof}
Our goal is to describe all the paths of length two connecting two vertices of $\mathfrak{M}:= \mathfrak{M}(\mathcal{I}, \mathcal{G}, \mathcal{A})$ at distance two. Up to translating by an element of $\mathfrak{Q}$, we can assume for simplicity that our two vertices are $1$ and $x$. Let $[ab]$ be a ranked braid representing $x$. We know that it has two letters because $x$ lies at distance two from $1$. Let $i,j \in I$ be such that $a \in G_i$ and $b \in G_j$. Notice that $i \neq j$, since otherwise $1$ and $x$ would be adjacent; and that $i<j$ cannot hold since $[ab]$ is ranked. 

\medskip \noindent
First, assume that $i \perp j$. A path of length two connecting $1$ to $x$ amounts to an equality $x=pq$ where $p \in G_r$ and $q \in G_s$ for some $r,s \in I$. Notice that $r \neq s$ since otherwise $1$ and $x$ would be adjacent. If $[pq]$ is rankable, then we must have $r<s$, in which case $[q(q \ast p)]$ is ranked. As a consequence of Theorem~\ref{thm:NormalForm}, we must have either $q=a$ and $q \ast p = b$ or $q=b$ and $q \ast p = a$. In the former case, $s=i$ and $q \ast r = j$, hence $j = q \ast r < s =i$, contradicting that $i \perp j$. In the latter case, similarly, $s=j$ and $q \ast r = i$, hence $i = q \ast r < s = j$, contradicting that $i \perp j$. So $[pq]$ must be ranked, which implies that $p=a$ and $q=b$ or $p=b$ and $q=a$. Thus, we have proved that $1, a, ab=x$ and $1,b,ba=x$ are the only two paths of length two connecting $1$ to $x$.

\medskip \noindent
Next, assume that $i>j$. Similarly, a path of length two connecting $1$ to $x$ amounts to an equality $x=pq$ where $p \in G_r$ and $q \in G_s$ for some $r,s \in I$ satisfying $r \neq s$. If $[pq]$ is rankable, then $r<s$, in which case $[q(q\ast p )]$ is ranked. As a consequence of Theorem~\ref{thm:NormalForm}, we must have $q=a$ and $q\ast p =b$. Otherwise, if $[pq]$ is ranked, it follows from Theorem~\ref{thm:NormalForm} that $p=a$ and $q=b$. Thus, we have proved that $1,a,ab=x$ and $1, a^{-1} \ast b, (a^{-1} \ast b)a=x$ are the only paths of length two connecting $1$ to $x$.

\medskip \noindent
Finally, assume that $i$ and $j$ are not comparable for either $\perp$ or $\leq$. Again, a path of length two connecting $1$ to $x$ amounts to an equality $x=pq$ where $p \in G_r$ and $q \in G_s$ for some $r,s \in I$ satisfying $r \neq s$. If $[pq]$ is rankable, then $r<s$, in which case $[q (q \ast p)]$ is ranked. As a consequence of Theorem~\ref{thm:NormalForm}, we must have $q=a$ and $q \ast p = b$. Hence $j=q \ast r < s=i$, contradicting that $i$ and $j$ are not $\leq$-comparable. So $[pq]$ must be ranked, which implies, according to Theorem~\ref{thm:NormalForm}, that $p=a$ and $q=b$. Thus, we have proved that $1,a,ab=x$ is the only path of length two connecting $1$ to $x$. 

\medskip \noindent
This concludes the proof of Lemma~\ref{lem:QuandleCycle}. Notice that, as a by-product of our argument, we get freely:

\begin{fact}
There is no induced copy of $K_{3,2}$ in $\mathfrak{M}(\mathcal{I}, \mathcal{G}, \mathcal{A})$. 
\end{fact}

\noindent
Indeed, we have shown that any two vertices at distance two are connected by at most two paths of length two. 
\end{proof}

\begin{proof}[Proof of Theorem~\ref{thm:QuandleQM}.]
Our goal is to show that Theorem~\ref{thm:QuandleQM} applies. First of all, notice that it follows from Lemmas~\ref{lem:QuandleClique} and~\ref{lem:QuandleCycle} that the square-triangle completion of $\mathfrak{M}:=\mathfrak{M}(\mathcal{I}, \mathcal{G}, \mathcal{A})$ coincides with a Cayley complex of our quandle product. Consequently, $\mathfrak{M}^{\square\triangle}$ is simply connected. 

\medskip \noindent
The fact that there is no induced copy of $K_4^-$ in $\mathfrak{M}$ is straightforward consequence of Lemma~\ref{lem:QuandleClique}. Indeed, if $x \in \mathfrak{M}$ is a vertex and $y_1,y_2,y_3 \in \mathfrak{M}$ three neighbours such that $y_2$ is adjacent to both $y_1$ and $y_3$, then it follows from Lemma~\ref{lem:QuandleClique} that there exist $i \in I$ and $a_1,a_2,a_3 \in G_i$ such that $y_1=xa_1$, $y_2= xa_2$, and $y_3=xa_3$. So $y_1$ must be adjacent to $y_3$ as $y_3= y_1(a_1^{-1}a_3)$ with $a_1^{-1}a_3 \in G_i$.

\medskip \noindent
Now, let $x \in \mathfrak{M}$ be a vertex and $y_1,y_2,y_3 \in \mathfrak{M}$ three neighbours such that $[x,y_1],[x,y_2]$ span a $4$-cycle and $[x,y_2],[x,y_3]$ a $4$-cycle. Up to translating by an element of $\mathfrak{Q}$, we can assume for simplicity that $x=1$. According to Lemma~\ref{lem:QuandleClique}, there exists $i \in I$ such that $y_2,y_3 \in G_i$. Let $j \in I$ be such that $y_1 \in G_j$. According to Lemma~\ref{lem:QuandleCycle}, three cases may happen, depending on whether $i \perp j$, $i<j$, or $i>j$. As shown below, we easily verify case by case that our three edges $[x,y_i]$ span a prism $K_3 \times K_2$. 
\begin{center}
\includegraphics[width=0.8\linewidth]{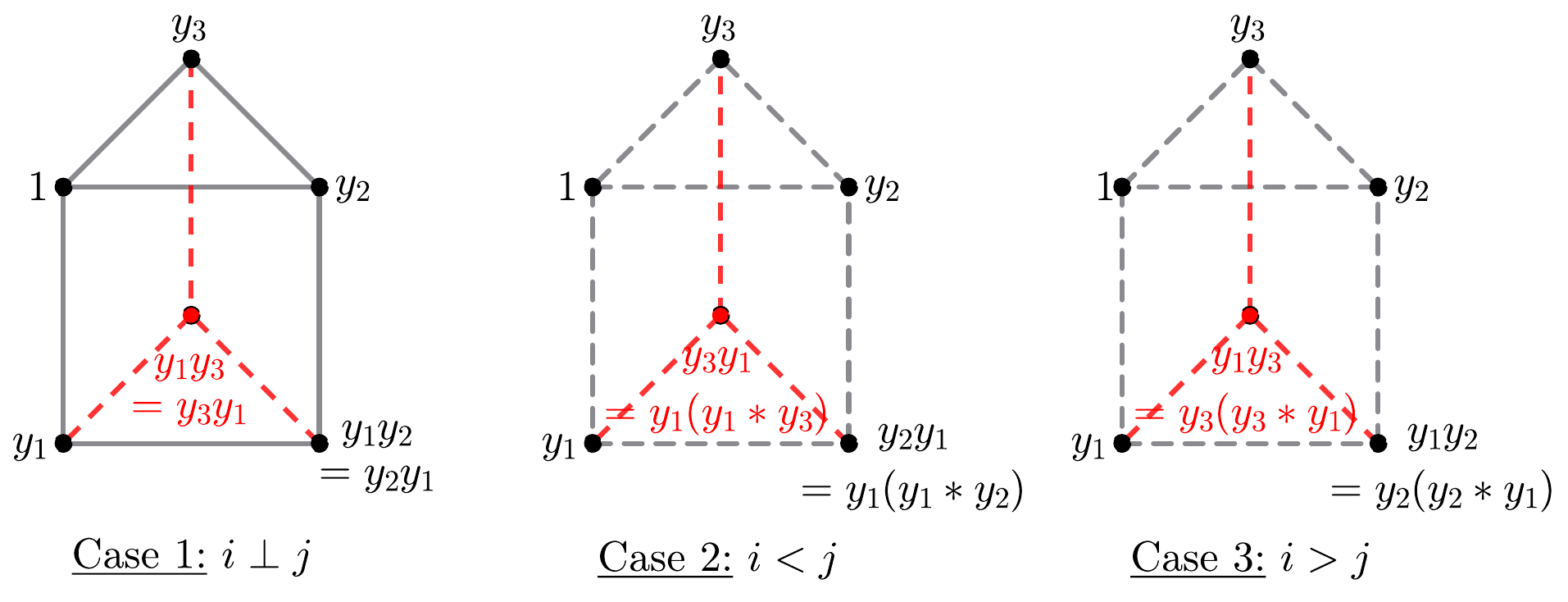}
\end{center}

\noindent
Finally, let $x \in \mathfrak{M}$ be a vertex and $a,b,c \in \mathfrak{M}$ three neighbours such that $[x,a]$, $[x,b]$, $[x,c]$ pairwise span a $4$-cycle. Up to translating by an element of $\mathfrak{Q}$, we can assume for simplicity that $x=1$. Let $i,j,k \in I$ be such that $a \in G_i$, $b \in G_j$, and $c \in G_k$. According to Lemma~\ref{lem:QuandleCycle}, the indices $i,j,k$ are pairwise comparable with respect to $\perp$ or $<$. Up to symmetry, there are four cases to consider. As shown below, we easily verify case by case that our three edges $[x,a]$, $[x,b]$, $[x,c]$ span a $3$-cube whose missing vertex is $abc$.
\begin{center}
\includegraphics[width=0.9\linewidth]{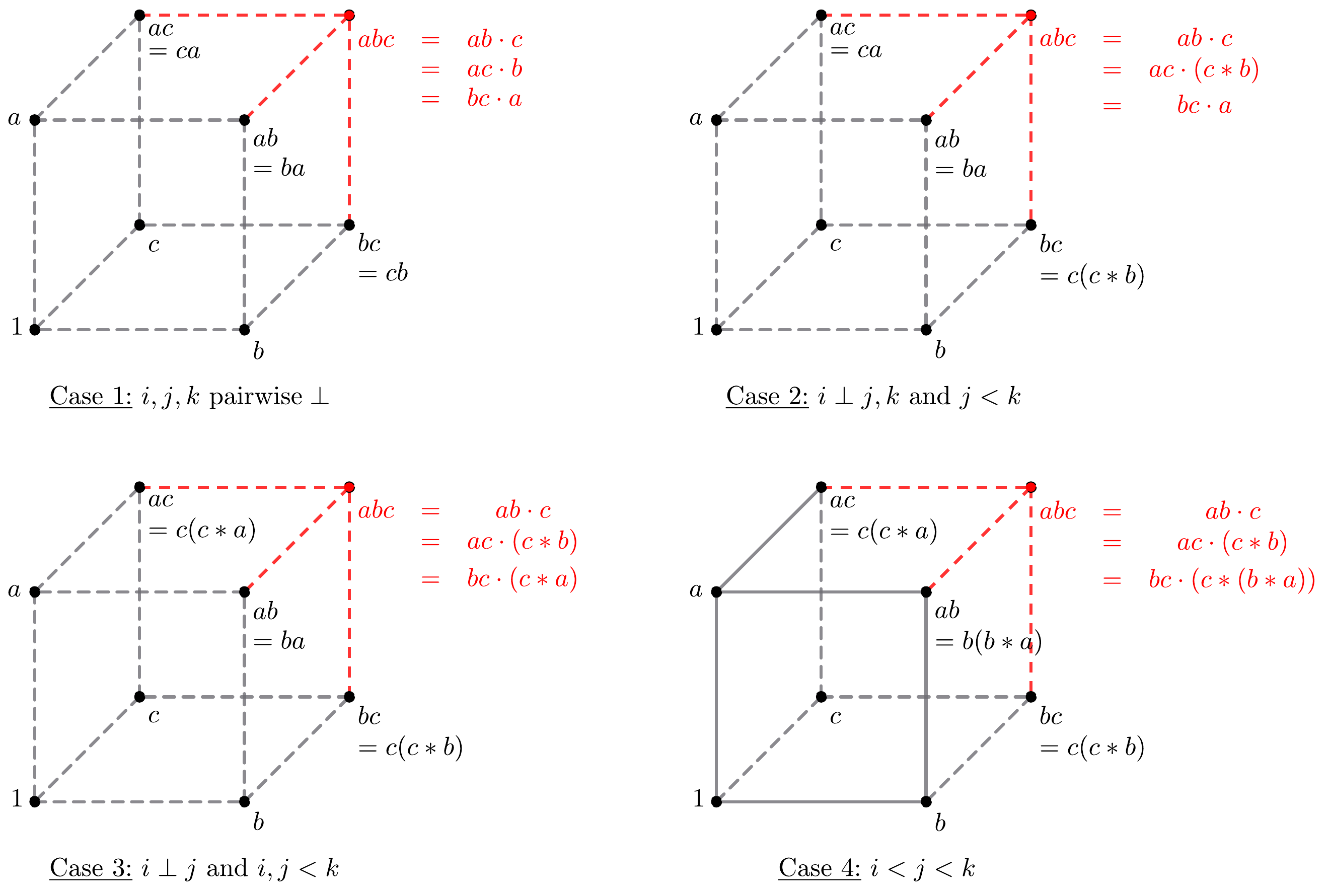}
\end{center}

\noindent
Thus, Theorem~\ref{thm:QuandleQM} applies and shows that our Cayley graph $\mathfrak{M}$ is indeed a quasi-median graph. It remains to verify the upper bound on its cubical dimension.

\medskip \noindent
Let $P$ be a prism of $\mathfrak{M}$. Up to translating by an element of $\mathfrak{Q}$, we can assume for simplicity that $1 \in P$. Let $a_1, \ldots, a_n \in P$ be neighbours of $1$ lying in distinct cliques. For every $1 \leq k \leq n$, let $i_k \in I$ be such that $a_k \in G_{i_k}$. For all distinct $1 \leq r,s \leq n$, we know that the edges $\{1,a_r\}$ and $\{1,a_s\}$ must span an induced $4$-cycle. It follows from Lemma~\ref{lem:QuandleCycle} that $i_r$ and $i_s$ are $<$- or $\perp$-comparable. Thus, we find the desired upper bound on the cubical dimension of $P$. This complete the proof of our theorem.  
\end{proof}

\subsection{The word problem again}\label{section:FirstApplications}

\noindent
We saw in Section~\ref{section:WP} how to represent efficiently elements in quandle products by some products of generators. However, our definition of ranked braids sounds articial: there is no reason to push our letters to the left as much as possible instead of pushing them to the right. In this section, we propose a different approach. 

\begin{definition}
Let $w:=s_1 \cdots s_n$ be a word where each letter $s_i$ belongs to some $G_{r(i)}$. A word $w'$ is obtained from $w$ by
\begin{itemize}
	\item a \emph{twisted right-commutation} if there exist $1 \leq i < n$ such that $r(i)>r(i+1)$ and $w'=s_1 \cdots s_{i-1} (s_i^{-1} \ast s_{i+1} ) s_i s_{i+2} \cdots s_n$;
	\item a \emph{twisted commutation} if $w'$ can be obtained from $w$ by a twisted left- or right-commutation. 
\end{itemize}
A word is \emph{$\mathfrak{Q}$-reduced} if it cannot be shortened by applying $\perp$-commutations, fusions, and twisted commutations.
\end{definition}

\noindent
It is worth mentioning that, if $[w]$ is a ranked braid, then $w$ is necessarily $\mathfrak{Q}$-reduced. However, if $w$ is $\mathfrak{Q}$-reduced, then the braid $[w]$ may not be ranked. 

\medskip \noindent
Our next statement, which will easily follow from the quasi-median geometry provided by Theorem~\ref{thm:QuandleQM}, shows how to solve the word problem in quandle products by using $\mathfrak{Q}$-reduced words instead of ranked braids. It extends what is known for graph products \cite{GreenGP} and what has been done for cactus groups \cite{MR4701892, MR4874027}. 

\begin{thm}\label{thm:ShortestWords}
Let $\mathfrak{Q}:= \mathfrak{Q}(\mathcal{I}, \mathcal{G}, \mathcal{A})$ be a quandle product. A word representing a given element of $\mathfrak{Q}$ has minimal length if and only if it is $\mathfrak{Q}$-reduced. Moreover, any two such words only differ by $\perp$- and twisted commutations. 
\end{thm}

\begin{proof}
In $\mathfrak{M}:= \mathfrak{M}(\mathcal{I}, \mathcal{G},\mathcal{A})$, as in any Cayley graph, a path is naturally labelled by a word of generators. As a consequence of the descriptions of $3$- and $4$-cycles provided by Lemmas~\ref{lem:QuandleClique} and~\ref{lem:QuandleCycle}, we know that removing a backtrack or shortening a $3$-cycle amounts to applying a fusion to the corresponding word; and that bypassing a $4$-cycle amounts to applying a $\perp$- or twisted commutation to the word. Therefore, since $\mathfrak{M}$ is quasi-median according to ~\ref{thm:QuandleQM}, Proposition~\ref{prop:PathAndGeodQM} applies and yields the desired conclusion. 
\end{proof}

\section{Combination results}\label{section:Combination}

\noindent
A natural question to ask is which group properties are preserved by quandle products. In \cite{QM}, we proved such combination results for groups acting on quasi-median graphs under some assumptions. The main assumption is that the action must be \emph{topical-transitive} (see Definition~\ref{def:Topic}). Proposition~\ref{prop:TopicHolonomy} below shows that this assumption holds exactly when the holonomy is trivial. This allows us to apply the theory developed in \cite{QM} to quandle products with trivial holonomy.

\subsection{Proper cubulation}\label{section:ProperCub}

\noindent
We start by proving that quandle products preserve proper actions on median graphs. (See Conjecture~\ref{conj:CAT} and the related discussion for cocompact actions.) 

\begin{thm}\label{thm:NPC}
Let $\mathfrak{Q}:= \mathfrak{Q}(\mathcal{I}, \mathcal{G}, \mathcal{A})$ be a quandle product with trivial holonomy. 
\begin{itemize}
	\item If every group in $\mathcal{G}$ acts properly on a median graph, then so does $\mathfrak{Q}$.
	\item If $\mathcal{I}$ is finite and if every group in $\mathcal{G}$ acts metrically properly on a median graph, then so does $\mathfrak{Q}$.
\end{itemize}
\end{thm}

\noindent
It is worth mentioning that, as a consequence of Example~\ref{ex:SemiDirect}, the Baumslag-Solitar group $\mathrm{BS}(1,2)= \mathbb{Z}[1/2] \rtimes \mathbb{Z}$ is a quandle product of $\mathbb{Z}[1/2]$ and $\mathbb{Z}$. Despite the fact that $\mathbb{Z}[1/2]$ and $\mathbb{Z}$ both admit a proper action on a median graph, $BS(1,2)$ cannot act properly on a median graph \cite{MR4645691}. This shows that assuming the triviality of the holonomy in Theorem~\ref{thm:NPC} is necessary.

\medskip \noindent
As already said, our goal is to apply the theory built in \cite{QM}, which allows us to prove combination results starting from specific actions on quasi-median graphs. First of all, let us recall the following central definition:

\begin{definition}\label{def:Topic}
Let $G$ be a group acting on a quasi-median graph $X$. The action is \emph{topical-transitive} if
\begin{itemize}
	\item for all hyperplane $J$, clique $C \subset J$, and $g \in \mathrm{stab}(J)$, there exists $h \in \mathrm{stab}(C)$ such that $g$ and $h$ induce the same permutation on the set of fibres delimited by $J$;
	\item for every clique $C$, either $C$ is finite and $\mathrm{stab}(C)= \mathrm{fix}(C)$ or $\mathrm{stab}(C) \curvearrowright C$ is free-transitive.
\end{itemize}
\end{definition}

\noindent
Notice that, when considering quandle products acting on their quasi-median Cayley graphs, clique-stabilisers are always free-transitive according to Corollary~\ref{cor:QuandleClique}, so the second item of the previous definition is always satisfied. Our next result characterises exactly when such actions are topical-transitive. 

\begin{prop}\label{prop:TopicHolonomy}
Let $\mathfrak{Q}:= \mathfrak{Q}(\mathcal{I}, \mathcal{G}, \mathcal{A})$ be a quandle product. The action $\mathfrak{Q} \curvearrowright \mathfrak{M}(\mathcal{I}, \mathcal{G}, \mathcal{A})$ is topical-transitive if and only if $(\mathcal{I}, \mathcal{G}, \mathcal{A})$ has trivial holonomy. 
\end{prop}

\noindent
In order to prove the proposition, we need to describe hyperplanes in quasi-median Cayley graphs of quandle products. For this purpose, we first have to introduce some terminology. Roughly speaking, given a quandle product $\mathfrak{Q}(\mathcal{I},\mathcal{G},\mathcal{A})$ and an $i \in I$, a word $w$ is \emph{$i$-diaphanous} if, for some (or equivalently, any) non-trivial element $s \in G_i$, the letter $s$ in $sw$ can be shuffled to the right by using the relations defining the quandle product. During the process, the letter $s$ becomes a possibly different letter $r \in G_j$. Then, $w$ is \emph{$(i,j)$-diaphanous}. More formally:

\begin{definition}
Let $(\mathcal{I}, \mathcal{G}, \mathcal{A})$ be a quandle system.
\begin{itemize}
	\item The empty word is $i$- and $(i,i)$-diaphanous for every $i \in I$.
	\item Let $w$ be a word of positive length. Write $w=w_0s$ where $s \in G_k$ is its last letter. If $w_0$ is $(i,j)$-diaphanous for some $i,j \in I$ and $k$ is $<-$ or $\perp$-comparable with $j$, then $w$ is $i$-diaphanous. Moreover, if $k \perp j$ or $k<j$, then $w$ is $(i,j)$-diaphanous; and, if $k>j$, then $w$ is $(i,s \ast j)$-diaphanous. 
\end{itemize}
\end{definition}

\noindent
Let us record a few straightforward consequences of this definition:

\begin{fact}\label{fact:Diaph}
Let $(\mathcal{I}, \mathcal{G}, \mathcal{A})$ be a quandle system. The following assertions hold:
\begin{itemize}
	\item For all $i,j \in I$, an $(i,j)$-diaphanous word is $i$-diaphanous. Conversely, for every $i \in I$, every $i$-diaphanous word is $(i,j)$-diaphanous for some $j \in I$. 
	\item For all $i,j,k \in I$, if $w_1$ is an $(i,j)$-diaphanous word and if $w_2$ is a $(j,k)$-diaphanous word, then $w_1w_2$ is $(i,k)$-diaphanous. 
	\item For all $i,j \in I$, if $w=w_1w_2$ is an $(i,j)$-diaphanous word, then there exists $\ell \in I$ such that $w_1$ is $(i, \ell)$-diaphanous and $w_2$ is $(\ell,j)$-diaphanous. \qed
\end{itemize}
\end{fact}

\noindent
Our first lemma describe fibres of hyperplanes in quasi-median Cayley graphs of quandle products. 

\begin{lemma}\label{lem:PathFibreDiaph}
Let $(\mathcal{I}, \mathcal{G}, \mathcal{A})$ be a quandle system. Let $J_i$ denote the hyperplane of $\mathfrak{M}:= \mathfrak{M} (\mathcal{I}, \mathcal{G}, \mathcal{A})$ containing the clique $G_i$. A path $\gamma$ starting from $1$ is contained in a fibre of $J_i$ if and only if the word $w$ labelling $\gamma$ is $i$-diaphanous. Moreover, if $j \in I$ is such that $w$ is $(i,j)$-diaphanous, then the clique of $J_i$ containing the terminus of $\gamma$ is a coset of~$G_j$.
\end{lemma}

\begin{proof}
Let $\gamma$ be a path starting from $1$ and assume that $\gamma$ is contained in a fibre of $J_i$. Let $w$ denote the word labelling $\gamma$. If $\gamma$ has is reduced to a single vertex, necessarily $1$, then it is clear that $w$ is $(i,i)$-diaphanous and that the clique of $J_i$ containing the terminus of $\gamma$ is a coset of $G_i$. From now on, assume that $\gamma$ has positive length. So we can decompose $\gamma$ as a concatenation $\gamma_0 \cup e$ where $e$ denotes the last edge of $\gamma$. Accordingly, the label $w$ decomposes as $w_0s$ where $s$ is the last letter of $w$. Arguing by induction over the length of $w$, we already know that $w_0$ is $(i,j)$-diaphanous for some $j \in I$ and that the clique of $J_i$ containing the terminus of $\gamma_0$ is a coset of $G_j$. It follows from Lemma~\ref{lem:Ladder} that, geometrically, the situation is the following:
\begin{center}
\includegraphics[width=0.8\linewidth]{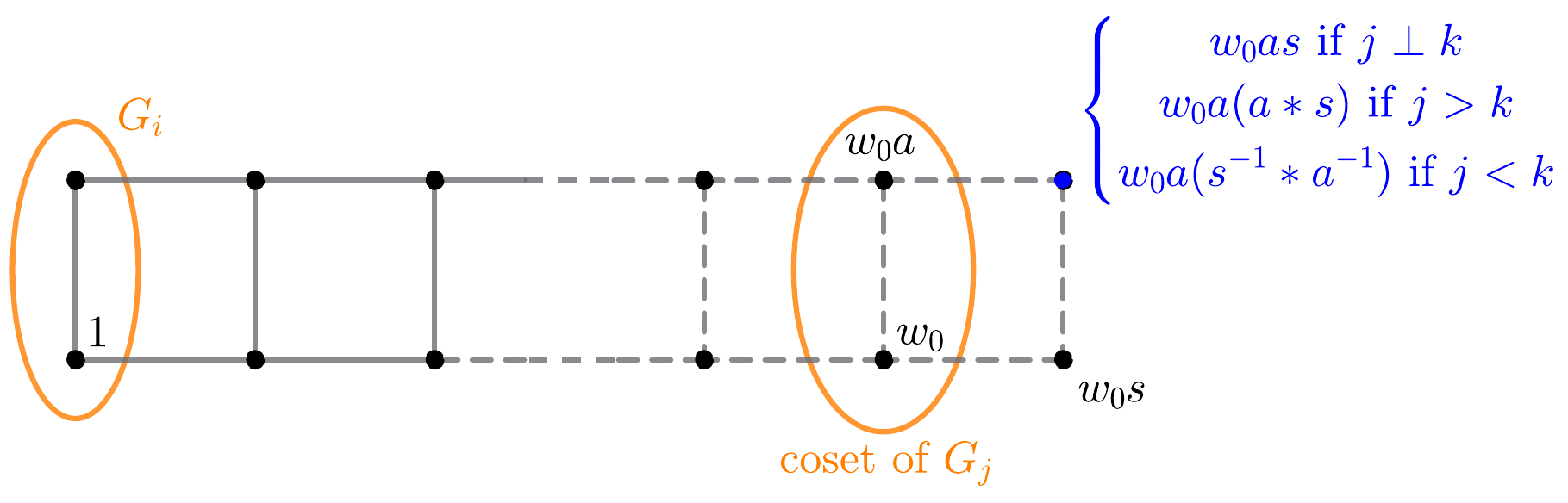}
\end{center}

\noindent
Let $k \in I$ be such that $s \in G_k$. By applying Lemma~\ref{lem:QuandleCycle} to the rightmost $4$-cycle in the figure above, we know that $j$ and $k$ are $<$- or $\perp$-comparable and that the rightmost edge is contained in a coset of $G_\ell$ where $\ell=j$ if $j \perp k$ or $j>k$ and $\ell= s^{-1} \ast j$ if $j<k$. We conclude that $s$ is $(j,\ell)$-diaphanous and that the clique of $J_i$ containing $w$ is a coset of $G_\ell$, as desired.

\medskip \noindent
Conversely, assume that $w$ is $(i,j)$-diaphanous for some $i,j \in I$. If $w$ is the empty word, then $\gamma$ is reduced to a single vertex, namely $1$, and it is clear that $\gamma$ is contained in a fibre of $J_i$ and that the clique of $J_i$ containing the terminus of $\gamma$ is $G_j=G_i$. From now on, assume that $w$ has positive length. So we can decompose $w$ as a concatenation $w_0s$ where $s$ denotes the last letter of $w$. Accordingly, $\gamma$ decomposes as $\gamma_0 \cup e$ where $\gamma_0$ is labelled by $w_0$ and where $e$ is the last edge of $\gamma$, which is then labelled by $s$. Let $k,\ell \in I$ be such that $w_0$ is $(i,k)$-diaphanous and $s \in G_\ell$. Arguing by induction over the length of $w$, we already know that $\gamma_0$ is contained in a fibre of $J_i$ and that the clique of $J_i$ containing the terminus of $\gamma_0$ is a coset of $G_k$. Notice that, since $w$ is $(i,j)$-diaphanous and that $w_0$ is $(i,k)$-diaphanous, necessarily $s$ is $(k,j)$-diaphanous. This amounts to saying that $k$ and $\ell$ are $<$- or $\perp$-comparable and that $j= s \ast k$ if $k < \ell$ and $j=k$ if $k \perp \ell$ or $k> \ell$. Thus, fixing some non-trivial element $a \in G_k$, we have the following geometric situation:
\begin{center}
\includegraphics[width=0.8\linewidth]{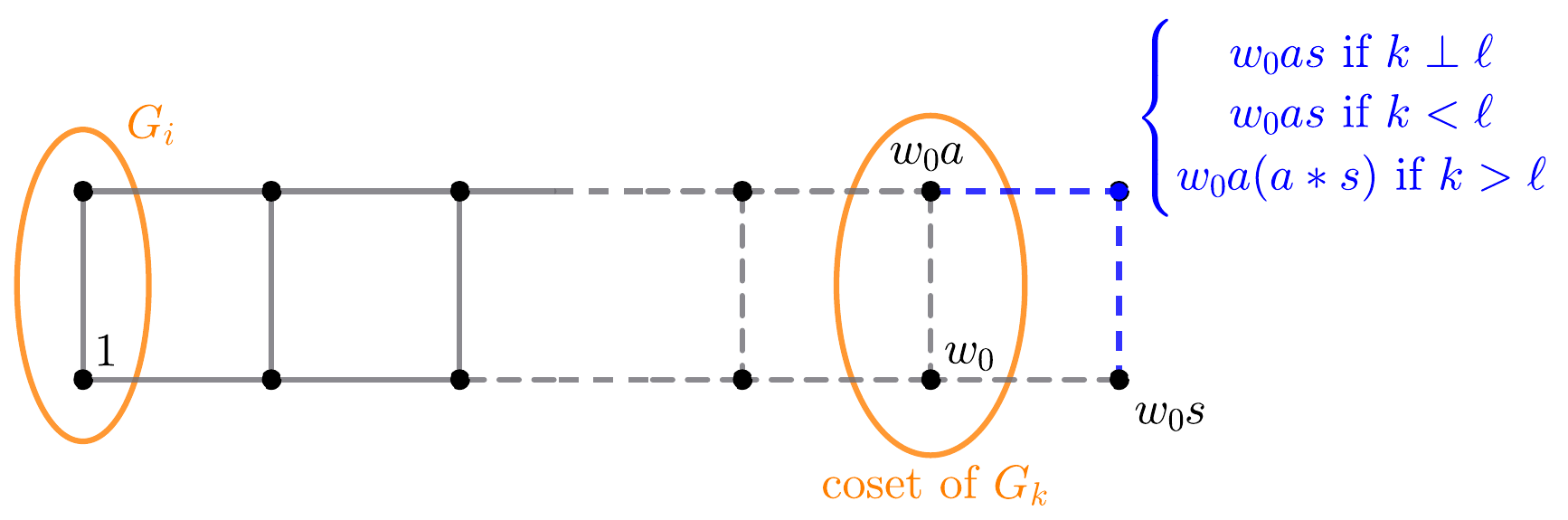}
\end{center}

\noindent
This implies that $\gamma$ is contained in a fibre of $J_i$ and that the clique of $J_i$ containing the terminus of $\gamma$ is a coset of $G_j$. 
\end{proof}

\begin{cor}\label{cor:DiaphElements}
Let $\mathfrak{Q}:= \mathfrak{Q}(\mathcal{I}, \mathcal{G}, \mathcal{A})$ be a quandle product. If an element $g \in \mathfrak{Q}$ can be represented by a $\mathfrak{Q}$-reduced word that is $i$-diaphanous (resp.\ $(i,j)$-diaphanous), then every $\mathfrak{Q}$-reduced word representing $g$ is $i$-diaphanous (resp.\ $(i,j)$-diaphanous).  
\end{cor}

\begin{proof}
Assume that an element $g \in \mathfrak{Q}$ can be represented by a $\mathfrak{Q}$-reduced word $w$ that is $i$-diaphanous. It follows from Lemma~\ref{lem:PathFibreDiaph} that the path $\gamma$ labelled by $w$ that connects $1$ to $g$ is contained in the fibre $F_i$ passing through $1$ of the hyperplane $J_i$ containing the clique $G_i$. In particular, $g$ belongs to $F_i$. It follows from the convexity of fibres (Theorem~\ref{thm:QMbig}) and the description of geodesics provided by Theorem~\ref{thm:ShortestWords} that the $\mathfrak{Q}$-reduced words representing $g$ are exactly the words labelling the geodesics in $F_i$ connecting $1$ to $g$. So the desired conclusion follows from Lemma~\ref{lem:PathFibreDiaph}. 
\end{proof}

\noindent
As a consequence of Corollary~\ref{cor:DiaphElements}, it makes sense to define $i$- and $(i,j)$-diaphanous elements in quandle products. Notice that, according to Fact~\ref{fact:Diaph}, the product of an $(i,j)$-diaphanous element with a $(j,k)$-diaphanous elements yields an $(i,k)$-diaphanous element. As a consequence, the set of all the $(i,i)$-diaphanous elements defines a subgroup, which we denote by $\mathrm{Diaph}(i)$. 

\begin{cor}\label{cor:HypStab}
Let $\mathfrak{Q}:= \mathfrak{Q}(\mathcal{I}, \mathcal{G}, \mathcal{A})$ be a quandle product, $i \in I$ an index, and let $J_i$ denote the hyperplane of $\mathfrak{M}:= \mathfrak{M}(\mathcal{I}, \mathcal{G}, \mathcal{A})$ containing the clique $G_i$. The fibre of $J_i$ passing through $1$ is $F_i:= \{ \text{$i$-diaphanous elements}\}$. The stabiliser of $J_i$ is $\langle G_i, \mathrm{Diaph}(i) \rangle$. 
\end{cor}

\begin{proof}
The description of the fibre $F_i$ follows from Lemma~\ref{lem:PathFibreDiaph} and the connexity of fibres. It is clear that $\mathrm{Diaph}(i)$ stabilises the fibre $F_i$. Conversely, if an element $g \in \mathfrak{Q}$ stabilises $F_i$, then it must be an $i$-diaphanous element such that $gG_i$ belongs to $J_i$. It follows from Lemma~\ref{lem:PathFibreDiaph} that $g$ must be $(i,i)$-diaphanous. Thus, 

\begin{fact}
The stabiliser of the fibre $F_i$ is $\mathrm{Diaph}(i)$.
\end{fact}

\noindent
Now, let $g \in \mathrm{stab}(J_i)$. The fibre $F_i$ is sent under $g$ to a fibre passing through some vertex $a \in G_i$. So $a^{-1}g$ stabilises $F_i$, and a fortiori belongs to $\mathrm{Diaph}(i)$, hence $g \in \langle G_i, \mathrm{Diaph}(i) \rangle$. Conversely, it is clear that $G_i$ and $\mathrm{Diaph}(i)$ stabilise $J_i$. 
\end{proof}

\noindent
We need a last preliminary lemma in order to prove Proposition~\ref{prop:TopicHolonomy}. Roughly speaking, it states that, given a quandle system $(\mathcal{I},\mathcal{G},\mathcal{A})$ and given two parallel cliques in the corresponding quasi-median graph, the generators labelling two parallel edges contained in these cliques can be obtained from one another by applying an isomorphism that is either the identity or that belongs to the system of groups $\mathscr{S}(\mathcal{I},\mathcal{G},\mathcal{A})$; moreover, this isomorphism only depends on the cliques and not of the particular edges under consideration. More formally:

\begin{lemma}\label{lem:ParallelEdges}
Let $\mathfrak{Q}:= \mathfrak{Q}(\mathcal{I}, \mathcal{G}, \mathcal{A})$ be a quandle product, $i,j \in I$, $a \in G_i \backslash \{1\}$, and $g \in \mathfrak{Q}$. There exists a clique containing $g$ that is parallel to $gaG_j$ if and only if $i$ and $j$ are $<$- or $\perp$-comparable. If so, then the clique is $gG_k$ where $k=a^{-1} \ast j$ if $i>j$ and $k=j$ if $i \perp j$ or $i<j$; moreover, for every $b \in G_j \backslash \{j\}$, the oriented edge of $gG_k$ parallel to $(ga,gab)$ is $(g,gc)$ where $c=a^{-1} \ast b$ if $i>j$ and $c=b$ if $i \perp j$ or $i<j$. 
\end{lemma}

\begin{proof}
The (non-)existence and description of the clique $gG_k$ is a direct consequence of the description of $4$-cycles in $\mathfrak{M}$ provided by Lemma~\ref{lem:QuandleCycle}. Assuming that this clique $gG_k$ exists, and given a non-trivial $b \in G_j$, if $c \in G_k$ is the element such that $(ga,gab)$ is parallel to $(g,gc)$, then $g,ga,gab, gc$ define a $4$-cycle, and the description of $c$ also follows from Lemma~\ref{lem:QuandleCycle}. 
\end{proof}

\begin{proof}[Proof of Proposition~\ref{prop:TopicHolonomy}.]
First, assume that the holonomy is trivial. Fix a hyperplane $J$ of $\mathfrak{M}$. Up to translating by an element of $\mathfrak{Q}$, we can assume that $J$ contains a clique $G_i$ for some $i \in I$. We know from Corollary~\ref{cor:HypStab} that $\mathrm{stab}(J)= \langle G_i, \mathrm{Diaph}(i) \rangle$. We claim that $\mathrm{Diaph}(i)$ stabilises each fibre delimited by $J$. This will be enough to conclude that the action of $\mathfrak{Q}$ on $\mathfrak{M}$ is topical-transitive.

\medskip \noindent
Consider an element of $\mathrm{Diaph}(i)$, which we will write as an $i$-diaphanous word $w=s_1 \cdots s_n$. We know from Corollary~\ref{cor:HypStab} that that $w$ stabilises the fibre of $J$ containing $1$. Given a non-trivial $g \in G_i$, we know from Lemma~\ref{lem:Ladder} that there must be a ladder connecting the edge $\{w,wg\}$ to the edge $\{1,g'\}$ for some non-trivial $g' \in G_i$. 
\begin{center}
\includegraphics[width=0.6\linewidth]{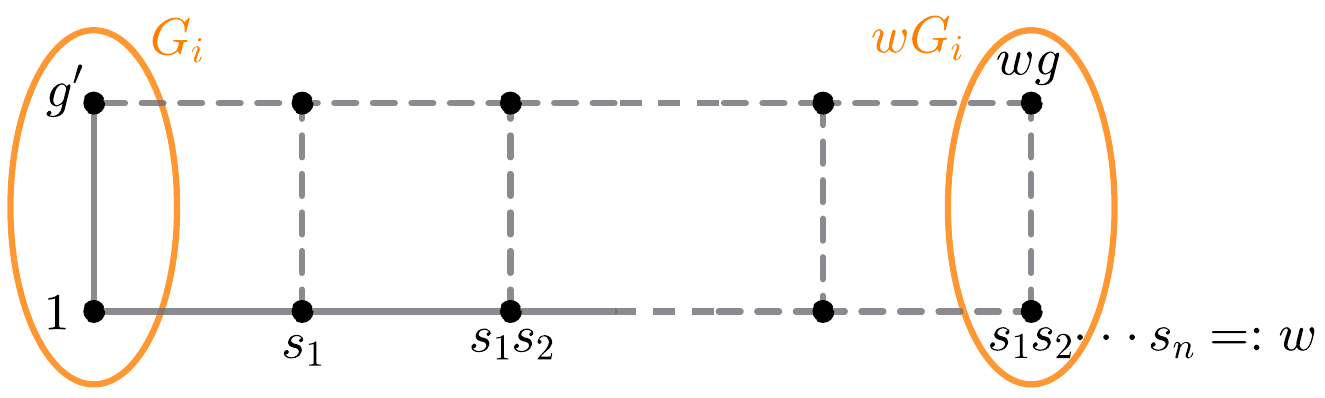}
\end{center}

\noindent
So $w$ sends the fibre of $J$ containing $g$ to the fibre containing $g'$. But we know from Lemma~\ref{lem:ParallelEdges} that, for any two consecutive parallel edges of our ladder, the labels can be obtained from one another by applying the identity or an isomorphism of the system of groups $\mathscr{S}(\mathcal{I}, \mathcal{G},\mathcal{A})$. Consequently, the label of $(1,g')$ - namely $g'$ - must be the image of the label of $(w,wg)$ - namely $g$ - under an automorphism of $G_i$ coming from $\mathrm{Hol}(i)$. Since $\mathrm{Hol}(i)$ is trivial by assumption, we conclude that $g'=g$; and a fortiori that $w$ stabilises the fibre of $J$ containing $g$. 

\medskip \noindent
Next, assume that the holonomy is non-trivial. So there exists an $i \in I$ such that $\mathrm{Hol}(i)$ contains a non-trivial automorphism, say $\varphi$ such that $\varphi(g) \neq g$ for some $g \in G_i \backslash \{1\}$. As explained in Section~\ref{section:Holonomy}, such an automorphism can be described as
$$\varphi : g \mapsto a_1 \ast (a_2 \ast ( \cdots \ast ( a_n \ast g)))$$
for some $a_1 \in G_{r(1)}, \ldots, a_n \in G_{r(n)}$. Moreover, the inequalities
$$\left\{ \begin{array}{l} r(n)>i \\ r(n-1) > a_n \ast i \\ r(n-2)> a_{n-1} \ast (a_n \ast i) \\ \vdots \\ r(1) > a_2 \ast ( \cdots \ast (a_{n-1} \ast (a_n \ast i))) \end{array} \right.$$
must be satisfied, as well as the equality $a_1 \ast (a_2 \ast ( \cdots \ast (a_n \ast i)))=i$. Set $w:= a_1^{-1} \cdots a_n^{-1}$. We deduce from these conditions that $w \in \mathrm{Diaph}(i)$; and, according to Lemma~\ref{lem:ParallelEdges}, that we have the following geometric configuration:
\begin{center}
\includegraphics[width=0.7\linewidth]{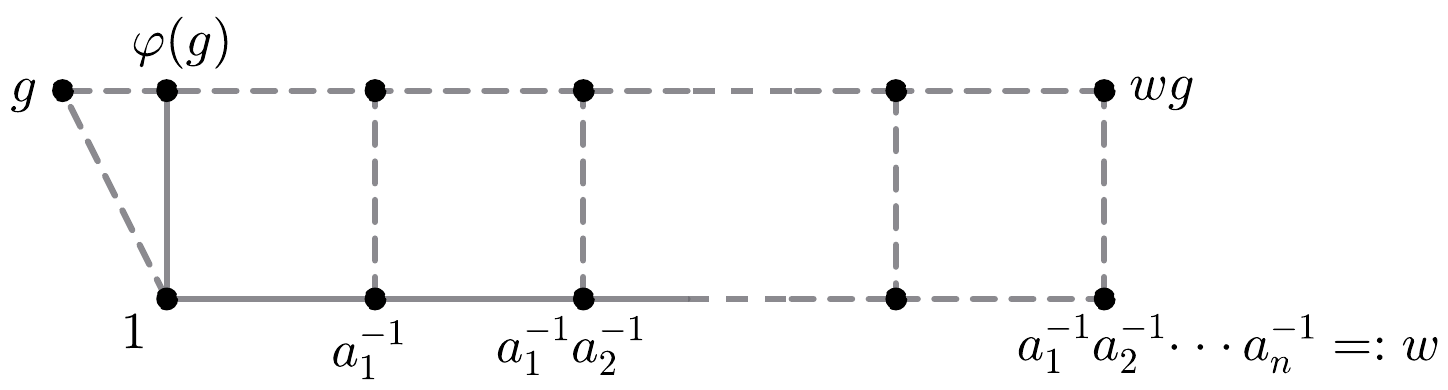}
\end{center}
Consequently, if $J_i$ denotes the hyperplane containing the clique $G_i$, then $w$ stabilises the fibre of $J_i$ containing $1$ but not the fibre containing $g$. Such an action cannot coincide with the action of an element of $G_i$ on the fibres delimited by $J_i$ since the action of $G_i$ is free. We conclude that the action of $\mathfrak{Q}$ on $\mathfrak{M}$ is not topical-transitive. 
\end{proof}

\begin{proof}[Proof of Theorem~\ref{thm:NPC}.]
It is clear that, with respect to the action of $\mathfrak{Q}$ on $\mathfrak{M}:= \mathfrak{M}(\mathcal{I},\mathcal{G},\mathcal{A})$, vertex-stabilisers are trivial. Moreover, according to Corollary~\ref{cor:QuandleClique}, clique-stabilisers are conjugates of factors. Thus, \cite[Proposition~7.4]{QM} shows that $\mathfrak{Q}$ admits a proper action on a median graph as soon as every group in $\mathcal{G}$ also admits a proper action on a median graph. The same statement about metrically proper actions follows from \cite[Proposition~5.22]{QM} by noticing that:

\begin{claim}\label{claim:FinetelyManyCliques}
If $\mathcal{I}$ is finite, then every vertex of $\mathfrak{M}$ belongs to only finitely many cliques.
\end{claim}

\noindent
It follows from Corollary~\ref{cor:QuandleClique} that every vertex $g$ of $\mathfrak{M}$ belongs to only finitely many cliques, namely $gG_i$ for $i \in I$.
\end{proof}

\subsection{Locally finite Cayley graphs}\label{section:LocallyFiniteCayley}

\noindent
Theorem~\ref{thm:NPC} shows that (metrically) proper actions on median graphs are preserved by quandle products with trivial holonomy. However, the theory developed in \cite{QM} does not allow us to show that geometric actions are preserved by quandle products. See Conjecture~\ref{conj:CAT} below and the related discussion. In this section, we prove a weaker result in this direction. 

\begin{thm}\label{thm:MedianCayl}
Let $\mathfrak{Q}:= \mathfrak{Q}(\mathcal{I}, \mathcal{G}, \mathcal{A})$ be a quandle product. For every $i \in I$, let $S_i \subset G_i$ be a generating set. Assume that, for all $i,j \in I$ and $g \in G_j$, if $i<j$ then $g \ast S_i = S_{g \ast i}$. If $\mathrm{Cayl}(G_i,S_i)$ is median (resp.\ quasi-median) for every $i \in I$, then 
$$\mathrm{Cayl} \left( \mathfrak{Q}, \bigcup\limits_{i \in I} S_i \right)$$
is median (resp.\ quasi-median). 
\end{thm}

\noindent
Of course, as generating sets, we can always take $S_i:=G_i$, in which case we recover Theorem~\ref{thm:QuandleQM}. For interesting applications, the groups in $\mathcal{G}$ will be finitely generated and we will want to find finite generated sets for the $S_i$. This is not always possible in full generality, but, with Lemma~\ref{lem:GeneratingNoHolo} below, we will see that such a collection of finite generating sets can always be found for quandle products with trivial holonomy. 

\medskip \noindent
The proof of Theorem~\ref{thm:MedianCayl} is based on \cite[Section~3]{QM}, which shows how to extend a collection of (local) metrics defined on the cliques of a quasi-median graph into a global metric that is ``quasi-median friendly''. We recall below basic definitions and properties related to this construction. (See also \cite{Contracting} for a similar construction in the more general framework of \emph{paraclique graphs}.)

\medskip \noindent
Let $X$ be a quasi-median graph. For every clique $C$ of $X$, fix a metric $\delta_C$ on $C$. We refer to the collection $\{\delta_C \mid C \text{ clique}\}$ as a \emph{system of (local) metrics}. The \emph{global (pseudo-)metric} $\delta$ on $X$ is defined by
$$\delta(a,b):= \inf \left\{ \sum\limits_ {i=1}^{n-1} \delta_{C_i}(x_i,x_{i+1}) \mid \begin{array}{c}  x_1=a, \ldots, x_n=b \text{ path in } X \\ C_i \text{ clique containing $x_i$ and $x_{i+1}$} \end{array} \right\}$$
for all $a,b \in X$. In full generality, there is no reason for the global metric to extend the local metric nor for the geometry of $\delta$ to be compatible with the quasi-median geometry of $X$. This is why, in practice, we assume that our system of local metrics satisfies some condition: the system is \emph{coherent} if, for all cliques $C$ and $D$ that belong to the same hyperplane, 
$$\delta_D(\mathrm{proj}_D(x), \mathrm{proj}_D(y)) = \delta_C(x,y) \text{ for all } x,y \in C$$
where $\mathrm{proj}_D : X \to D$ denotes the gate-projection on $D$. Global metrics of coherent systems of metrics do extend the local metrics, and they turn out to be compatible in some sense with the quasi-median geometry of $X$. Let us mention two alternative descriptions of the global metric under the assumption of coherence. 

\begin{lemma}[{\cite[Proposition~3.13]{QM}}]
Let $X$ be a quasi-median graph endowed with a coherent system of metrics $\{\delta_C \mid C \text{ clique}\}$. For every hyperplane $J$, choose a clique $C \subset J$ and define
$$\delta_J(a,b) := \delta_C( \mathrm{proj}_C(a), \mathrm{proj}_C(b)) \text{ for all } a,b \in X.$$
Then $\delta_J$ is a pseudo-metric on $X$ that does not depend on the choice of $C$. Moreover, 
$$\delta(a,b)= \sum\limits_{J \text{ hyperplane}} \delta_J(a,b) = \sum\limits_{J \text{ hyperplane separating $a$ and $b$}} \delta_J(a,b)$$
for all $a,b \in X$. 
\end{lemma}

\noindent
In order to state our second alternative description of global metrics, we will assume for simplicity that our local metrics are graph-metrics. Notice that the global metric of a system of local graph-metrics is also a graph-metric. 

\begin{lemma}[{\cite[Lemma~3.18]{QM}}]\label{lem:BrokenGeod}
Let $X$ be a quasi-median graph endowed with a coherent system of graph-metrics $\{\delta_C \mid C \text{ clique}\}$. A \emph{broken geodesic} between two vertices $a,b \in X$ is a path in $(X,\delta)$ of the form $\gamma_1 \cup \cdots \cup \gamma_{n-1}$ where $x_1=a, \ldots, x_n=b$ is a geodesic in $X$ and where each $\gamma_i$ is a $\delta_{C_i}$-geodesic in the clique $C_i$ of $X$ containing both $x_i$ and $x_{i+1}$. Then, broken geodesics are geodesics in $(X,\delta)$. 
\end{lemma}

\noindent
It is worth mentioning, however, that geodesics with respect to a global metric may not be broken geodesics. 

\medskip \noindent
Theorem~\ref{thm:MedianCayl} follows from two key observations: Cayley graphs of quandle products as in Theorem~\ref{thm:MedianCayl} can be obtained from global metrics given by coherent systems of metrics on quasi-median graphs (Proposition~\ref{prop:CayleyQ}); and global metrics given by coherent systems of (quasi-)median metrics are again (quasi-)median (Proposition~\ref{prop:QMdelta}). We start by proving the latter observation.

\begin{prop}\label{prop:QMdelta}
Let $X$ be a quasi-median graph and $\{\delta_C \mid C \text{ clique}\}$ a coherent system of metrics. If $(C,\delta_C)$ is a median (resp.\ quasi-median) graph for every clique $C$, then $(X,\delta)$ is a median (resp.\ quasi-median) graph. 
\end{prop}

\noindent
In order to prove the proposition, the following observation will be needed.

\begin{lemma}\label{lem:GatedGated}
Let $X$ be a quasi-median graph and $\{\delta_C \mid C \text{ clique}\}$ a coherent system of graph-metrics. For every $Y \subset V(X)$, if $Y$ induces a gated subgraph in $X$, then it induces a gated subgraph in $(X,\delta)$. If so, the gates taken in $X$ and $(X,\delta)$ coincide. 
\end{lemma}

\begin{proof}
Let $x \in X$ be a vertex and let $p \in Y$ denote its gate in $Y$. For every vertex $y \in Y$, we know that there exists a geodesic in $X$ connecting $x$ to $y$ that passes through $p$. Consequently, there exists a broken geodesic connecting $x$ to $y$ that passes through $p$. It follows from Lemma~\ref{lem:BrokenGeod} that there exists a geodesic in $(X,\delta)$ connecting $x$ to $y$ that passes through $p$. This proves that $p$ is the gate of $x$ in $(Y,\delta)$. 
\end{proof}

\begin{proof}[Proof of Proposition~\ref{prop:QMdelta}.]
Assume that $(C,\delta_C)$ is a graph-metric for every clique $C$. We start by verifying that, if each $(C,\delta_C)$ is weakly modular, then so is $(X,\delta)$. 

\begin{claim}\label{claim:CombTC}
If each $(C,\delta_C)$ satisfies the triangle condition, then so does $(X,\delta)$.
\end{claim}

\noindent
Let $o,x,y \in X$ be three vertices satisfying $\delta(o,x)=\delta(o,y)$ and $\delta(x,y)=1$. Necessarily, $x$ and $y$ belong to a common clique $C$ of $X$. Let $z_0 \in C$ denote the gate of $o$ in $C$. As a consequence of Lemma~\ref{lem:GatedGated}, $\delta(z_0,x)=\delta(z_0,y)$. By aplying the triangle condition in $(C,\delta_C)$, we find a vertex $z$ that is adjacent to both $x$ and $y$ in $(C,\delta_C)$ and such that $\delta_C(z_0,z) = \delta_C(z_0,x)-1$. Since $z_0$ is the gate of $z$ in the subgraph $(C,\delta_C)$ of $(X,\delta)$, we have
$$\delta(o,z)=\delta(o,z_0)+ \delta(z_0,z) = \delta(o,z_0) + \delta(z_0,x)-1 = \delta(o,x)-1.$$
Thus, $z$ is the vertex we are looking for. 

\begin{claim}\label{claim:CombQC}
If each $(C,\delta_C)$ satisfies the quadrangle condition, then so does $(X,\delta)$.
\end{claim}

\noindent
Let $o,x,y,z \in X$ be three vertices such that $\delta(o,y)=\delta(o,z)=\delta(o,x)-1$ and $\delta(y,z)=2$. We start by showing that there exists a gated subgraph $Q$ containing $x,y,z$ such that $(Q,\delta)$ is weakly modular. Since $x$ is adjacent to both $y$ and $z$ in $(X,\delta)$, necessarily $x$ is adjacent to both $y$ and $z$ in $X$. We distinguish two cases. If $x,y,z$ all belong to the same clique $C$, then we set $Q:=C$. Otherwise, the hyperplane of $X$ containing $\{x,y\}$, which necessarily separates $o$ and $y$, must be transverse to the hyperplane containing $\{x,z\}$, which necessarily separates $o$ and $z$. Therefore, the edges $\{x,y\}$ and $\{x,z\}$ are contained in a common prism $P$, which decomposes as the product of the two cliques containing $\{x,y\}$ and $\{x,z\}$. Since the quadrangle condition is stable under product, we know that $(P,\delta)$ satisfies the quadrangle condition. Then, we set $Q:=P$.

\medskip \noindent
Now, let $w_0$ denote the gate of $o$ in $Q$. As a consequence of Lemma~\ref{lem:GatedGated}, we have $\delta(w_0,y)= \delta(w_0,z)=\delta(w_0,x)-1$. Since $(Q,\delta)$ satisfies the quadrangle condition, we can find a common neighbour $w \in Q$ of $y,z$ in $(Q,\delta)$ such that $\delta(w_0,w)=\delta(w_0,x)-2$. Since
$$\delta(o,w)=\delta(o,w_0)+\delta(w_0,w)=\delta(o,w_0)+ \delta(w_0,x)-2  =\delta(o,x)-2,$$
we conclude that $w$ is the vertex we are looking for, completing the proof of Claim~\ref{claim:CombQC}. 

\medskip \noindent
It remains to understand how behave specific subgraphs that are allowed in weakly modular graphs but forbidden in (quasi-)median graphs.

\begin{claim}\label{claim:NoSubGraphs}
Any subgraph of $(X,\delta)$ isomorphic to $K_3$, $K_4^-$, or $K_{3,2}$ must be contained in $(C,\delta)$ for some clique $C$ of $X$.
\end{claim}

\noindent
This essentially follows from the fact that two adjacent vertices of $(X,\delta)$ are necessarily adjacent in $X$. In particular, every complete subgraph (e.g.\ $K_3$) of $(X,\delta)$ must be contained in a clique of $X$. A subgraph $K_4^-$ of $(X,\delta)$ must be contained in two cliques of $X$ whose intersection contains at least one edge. But, in a quasi-median graph, two such cliques must coincide, so our $K_4^-$ actually must be contained in a single clique of $X$. Finally, a $K_{3,2}$ in $(X,\delta)$ cannot be induced in $X$, so  at least two of its non-adjacent vertices must be adjacent in $X$, thus producing a $K_4^-$. It follows that all the edges of our $K_{3,2}$ belong to the same hyperplane of $X$, from which we deduce that it must be contained in some clique of $X$. Claim~\ref{claim:NoSubGraphs} is proved. 

\medskip \noindent
Since quasi-median graphs can be defined as weakly modular graphs with no induced copies of $K_4^-$ nor $K_{3,2}$, according to Theorem~\ref{thm:QMmodular}, it follows from Claims~\ref{claim:CombTC}, \ref{claim:CombQC}, and~\ref{claim:NoSubGraphs} that, if each $(C,\delta_C)$ is quasi-median, then $(X,\delta)$ is quasi-median as well. Notice that, as a consequence of Claim~\ref{claim:NoSubGraphs}, we know that $(X,\delta)$ is triangle-free if and only if each $(C,\delta_C)$ is triangle-free. Therefore, since median graphs can be defined as triangle-free quasi-median graphs according to Theorem~\ref{thm:MedianQM}, we conclude that $(X,\delta)$ is median whenever each $(C,\delta_C)$ is median. 
\end{proof}

\noindent
Now, we turn to the second key observation in order to prove Theorem~\ref{thm:MedianCayl}. 

\begin{prop}\label{prop:CayleyQ}
Let $\mathfrak{Q}:= \mathfrak{Q}(\mathcal{I}, \mathcal{G}, \mathcal{A})$ be a quandle product. For every $i \in I$, let $S_i \subset G_i$ be a generating set. Assume that, for all $i,j \in I$ and $g \in G_j$, if $i<j$ then $g \ast S_i = S_{g \ast i}$. For every clique $C$ of $\mathfrak{M}:= \mathfrak{M}(\mathcal{I}, \mathcal{G}, \mathcal{A})$, write $C=gG_i$ for some $i \in I$ and some $g \in \mathfrak{Q}$, and define
$$\delta_C : (ga,gb) \mapsto |a^{-1}b|_{S_i}.$$
Then $\{\delta_C \mid C \text{ clique}\}$ is a coherent system of metrics; and, if we denote by $\delta$ the corresponding global metric, then $(\mathfrak{M},\delta)$ is isometric to $\mathrm{Cayl}(\mathfrak{Q}, \bigcup_{i \in I} S_i)$. 
\end{prop}

\noindent
Before proving the proposition, let us record an elementary criterion that allows us to verify whether or not a system of metrics is coherent. 

\begin{lemma}\label{lem:WhenCoherent}
Let $X$ be a quasi-median graph. A system of metrics $\{ \delta_C \mid C \text{ clique}\}$ is coherent if and only if, for all oriented edges $(a,b)$ and $(x,y)$ that are opposite sides of an induced $4$-cycle, we have $\delta_C(a,b)= \delta_D(x,y)$ where $C$ (resp.\ $D$) is the clique containing $a$ and $b$ (resp.\ $x$ and $y$). 
\end{lemma}

\begin{proof}
Assume that our system of metrics is coherent, and fix two oriented edges $(a,b)$ and $(x,y)$ that are opposite sides of an induced $4$-cycle. Denote by $C$ (resp.\ $D$) the clique containing $a$ and $b$ (resp.\ $x$ and $y$). Since $x$ (resp.\ $y$) is the gate-projection of $a$ (resp.\ $b$) on $D$, the equality $\delta_C(a,b)= \delta_D(x,y)$ follows immediately from the coherence of our system of metrics.

\medskip \noindent
Conversely, assume that, for all oriented edges $(a,b)$ and $(x,y)$ that are opposite sides of an induced $4$-cycle, we have $\delta_C(a,b)= \delta_D(x,y)$ where $C$ (resp.\ $D$) is the clique containing $a$ and $b$ (resp.\ $x$ and $y$). Fix two cliques $A$ and $B$ that belong to a common hyperplane, and let $a,b \in A$ be two vertices. As a consequence of Lemma~\ref{lem:Ladder}, there exists a ladder connecting $(a,b)$ to the edge $(a',b')$ of $B$ where $a'$ (resp.\ $b'$) denotes the gate-projection of $a$ (resp.\ $b$) on $B$. Let $(x_1,y_1)=(a,b), \ldots, (x_n,y_n)$ denote the (oriented) rungs of our ladder; and, for every $1 \leq i \leq n$, let $K_i$ denote the clique containing $x_i$ and $y_i$. It follows from our assumption that
$$\delta_{K_i}(x_i,y_i) = \delta_{K_{i+1}}(x_{i+1},y_{i+1}) \text{ for every } 1 \leq i \leq n-1.$$
We conclude that
$$\delta_A(a,b)= \delta_{K_1}(x_1,y_1)= \delta_{K_n}(x_n,y_n)= \delta_B(a',b').$$
This proves that our system of metrics is coherent.
\end{proof}

\begin{proof}[Proof of Proposition~\ref{prop:CayleyQ}.]
First of all, notice that the metric $\delta_C$ does not depend on the choice of $g$. Indeed, let $h \in \mathfrak{Q}$ be another element satisfying $C=hG_i$. Since $h \in C = gG_i$, there exists some $c \in G_i$ such that $h=gc$. Then, for all $x,y \in C$, we can write $x=ha$ and $y=hb$ for some $a ,b \in G_i$. We also have $x=g(ca)$ and $y=g(cb)$. Since $|(ca)^{-1}(cb)|_{S_i} = |a^{-1}b|_{S_i}$, the claim follows. 

\medskip \noindent
In order to verify that $\{\delta_C \mid C \text{ clique}\}$ is coherent, we apply Lemma~\ref{lem:WhenCoherent}. So let $S$ be an $4$-cycle in $\mathfrak{M}$. According to Lemma~\ref{lem:QuandleCycle}, $S$ corresponds to one of the following configurations:
\begin{center}
\includegraphics[width=0.7\linewidth]{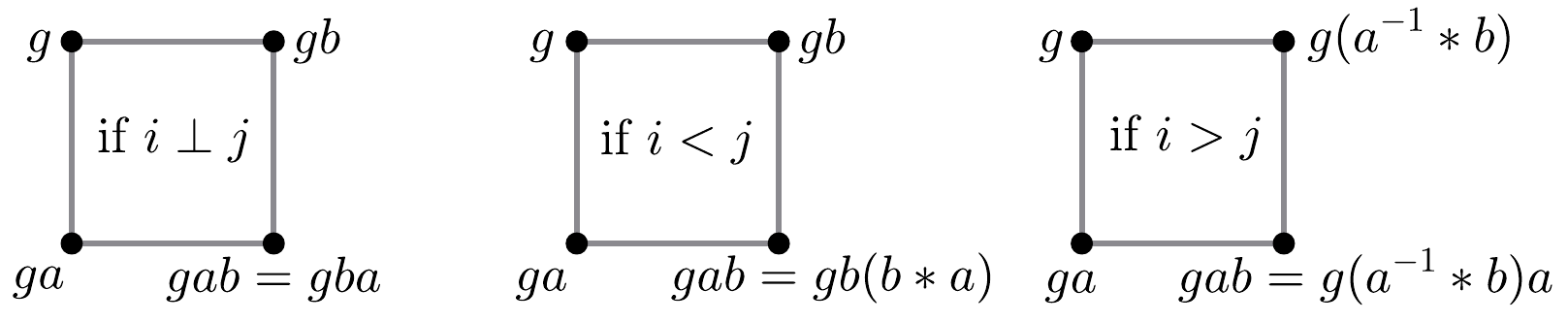}
\end{center}

\noindent
where $g \in \mathfrak{Q}$, $i,j \in I$, $a \in G_i$, and $b \in G_j$. Let $C$ (resp.\ $D$) denote the clique containing containing the top (resp.\ bottom) edge of $S$. If $i \perp j$ or $i<j$, then
$$\delta_C(g,gb) = |b|_{S_j} = \delta_D(ga,gab),$$
as desired. If $i>j$, then the equality
$$\delta_C(g,g (a^{-1} \ast b)) = |a^{-1} \ast b|_{a^{-1} \ast j} = |b|_{S_j} = \delta_D(ga,gab)$$
follows from the fact that $a^{-1} \ast S_j = S_{a^{-1} \ast j}$. Thus, coherence is proved.

\medskip \noindent
Finally, notice that, for all vertices $a,b \in \mathfrak{M}$, $\delta(a,b)=1$ holds if and only if there exists a clique $C$ containing both $a,b$ and $\delta_C(a,b)=1$. By construction, this amounts to saying that $a^{-1}b$, or its inverse, belongs to $\bigcup_{i \in I}S_i$. This shows that the graph $(\mathfrak{M},\delta)$ coincides with $\mathrm{Cayl}(\mathfrak{Q}, \bigcup_{i \in I} S_i)$. 
\end{proof}

\begin{proof}[Proof of Theorem~\ref{thm:MedianCayl}.]
Our theorem follows immediately from the combination of Propositions~\ref{prop:CayleyQ} and~\ref{prop:QMdelta}. 
\end{proof}

\noindent
We conclude this section by noticing that Theorem~\ref{thm:MedianCayl} can always be applied to quandle products of finitely generated groups when the holonomy is trivial. 

\begin{lemma}\label{lem:GeneratingNoHolo}
Let $(\mathcal{I}, \mathcal{G}, \mathcal{A})$ be a quandle system with trivial holonomy. Assume that every group in $\mathcal{G}$ is finitely generated. There exists a collection of finite generating sets $\{S_i \subset G_i \mid i \in I\}$ such that, for all $i,j \in I$ and $g \in G_j$, if $i<j$ then $g \ast S_i  = S_{g \ast i}$. 
\end{lemma}

\begin{proof}
Consider the system of groups $\mathscr{S}:=\mathscr{S}(\mathcal{I}, \mathcal{G},\mathcal{A})$ as defined in Section~\ref{section:Holonomy}. Choose a group in each connected component of $\mathcal{S}$. We get a collection of groups $\{G_i \mid i \in I_0\}$ for some $I_0 \subset I$. For every $i \in I_0$, fix a finite generating set $S_i \subset G_i$. Now, given an arbitrary $j \in I$, there exist a unique $i \in I_0$ and an isomorphism $\varphi_j : G_i \to G_j$ obtained by composing morphisms of $\mathscr{S}$. Notice that, because there is no holonomy, the isomorphism $\varphi_j$ is also uniquely determined. Then, we set $S_j:= \varphi_j(S_i)$. The collection $\{S_i \subset G_i \mid i \in I\}$ satisfies the desired condition by construction, due to the triviality of the holonomy. 
\end{proof}

\subsection{Banach space compression}

\noindent
We conclude this section by recording a few by-products of the results obtained in Sections~\ref{section:ProperCub} and~\ref{section:LocallyFiniteCayley}, once combined with \cite{QM}.

\medskip \noindent
Recall that a (discrete) group is \emph{a-T-menable} (resp.\ \emph{a-$L^p$-menable}) whenever it admits a metrically proper action on some Hilbert space (resp.\ $L^p$-space). Roughly speaking, such a property shows that the geometry of the group under consideration is compatible with the geometry of Hilbert spaces (resp.\ $L^p$-spaces). A quantitative measure of this compatibility is given by the \emph{$L^p$-compression}. The \emph{compression} of a Lipschitz map $\varphi : X\to Y$ between two metric spaces is defined as
$$\mathrm{comp}(\varphi) :=\sup \left\{ \alpha \in [0,1] \mid \exists C >0, \forall x,y \in X, d(\varphi(x),\varphi(y)) \geq C d(x,y)^\alpha \right\}.$$
Then, the \emph{$L^p$-compression} of a finitely generated group $G$ is
$$\alpha_p(G):= \sup \{ \mathrm{comp}(\varphi) \mid \varphi : G \to L^p\text{-space Lipschitz} \}.$$
The main result of this section is:

\begin{thm}\label{thm:Compression}
Let $\mathfrak{Q}:= \mathfrak{Q}(\mathcal{I}, \mathcal{G}, \mathcal{A})$ be a quandle product with $\mathcal{I}$ finite and with trivial holonomy. 
\begin{itemize}
	\item If every group in $\mathcal{G}$ is a-T-menable (resp.\ a-$L^p$-menable for some odd $p \geq 1$), then so is $\mathfrak{Q}$.
	\item If every group in $\mathcal{G}$ is finitely generated, then $$\alpha_p (\mathfrak{Q}) \geq \min \left( \frac{1}{p}, \min\limits_{G \in \mathcal{G}} \alpha_p(G) \right).$$ for every $p \geq 1$.
\end{itemize}
\end{thm}

\begin{proof}
We know from Proposition~\ref{prop:TopicHolonomy} that $\mathfrak{Q}$ acts topically-transitively on its quasi-median Cayley graph $\mathfrak{M}$. Clearly, vertex-stabilisers are trivial; and, according to Corollary~\ref{cor:QuandleClique}, clique-stabilisers are isomorphic to groups in $\mathcal{G}$. We also know from Fact~\ref{claim:FinetelyManyCliques} that every vertex of $\mathfrak{M}$ belongs to only finitely many cliques. Therefore, \cite[Proposition~5.25]{QM} (resp.\ \cite[Proposition~5.26]{QM}) applies and show that, if every group in $\mathcal{G}$ is a-T-menable (resp.\ a-$L^p$-menable for some odd $p \geq 1$), then so is $\mathfrak{Q}$. This proves the first item of our theorem. Let us prove the second item. 

\medskip \noindent
By combining Lemma~\ref{lem:GeneratingNoHolo} and Proposition~\ref{prop:CayleyQ}, we know that $\mathfrak{Q}$ has a finite generating set $S$ such that $\mathrm{Cayl}(\mathfrak{Q},S)$ can be described as $(\mathfrak{M},\delta)$ where $\delta$ is the global metric given by a coherent system of metrics $\{\delta_C \mid C \text{ clique}\}$. Moreover, for every clique $C$ of $\mathfrak{M}$, $(C,\delta_C)$ is isometric to a Cayley graph of a group in $\mathcal{G}$ (with respect to a finite generating set). Then, the desired conclusion follows from \cite[Proposition~3.30]{QM}. 
\end{proof}

\section{Iterated semidirect products}\label{section:SemiDecomposition}

\noindent
Section~\ref{section:Combination} recorded geometric aspects of quandle products that can be deduced from the quasi-median geometry described in Section~\ref{section:QMquandle}. In this section, we focus on more algebraic aspects of quandle products. Our main result in this direction is the following decomposition theorem:

\begin{thm}\label{thm:SemiDirect}
Let $\mathfrak{Q}:= \mathfrak{Q}(\mathcal{I}, \mathcal{G}, \mathcal{A})$ be a quandle product. Let $I_\mathrm{min}$ denote the set of minimal elements in $I$ and let $\mathfrak{Q}_\mathrm{min}:= \langle gG_ig^{-1} \mid i \in I_\mathrm{min}, g \in \mathfrak{Q} \rangle$. Then $\mathfrak{Q}= \mathfrak{Q}_\mathrm{min} \rtimes \langle I \backslash I_\mathrm{min} \rangle$. Moreover, $\mathfrak{Q}_\mathrm{min}$ is a graph product whose vertex-groups are isomorphic to $G_i$ for some $i \in I_\mathrm{min}$ and the action of $\langle I \backslash I_\mathrm{min} \rangle$ on $\mathfrak{Q}_\mathrm{min}$ permutes vertex-groups.
\end{thm}

\noindent
Since parabolic subgroups are quandle products according to Corollary~\ref{cor:FactorsEmb}, Theorem~\ref{thm:SemiDirect} allows one to prove that some properties are preserved under quandle products by induction on the number of factors. Examples will be given below. Notice, however, that our oposet may not contain minimal elements, so Theorem~\ref{thm:SemiDirect} does not bring any valuable information for such a quandle product of infinitely many groups. See Example~\ref{ex:Push} for such a quandle product. Nevertheless, for quandle products of finitely many groups, Theorem~\ref{thm:SemiDirect} provides a complete decomposition as iterated semidirect and graph products. 

\begin{cor}\label{cor:DecompositionQuandle}
Let $\mathfrak{Q}:= \mathfrak{Q}(\mathcal{I}, \mathcal{G}, \mathcal{A})$ be a quandle product. If $\mathcal{I}$ is finite, then 
$$\mathfrak{Q} = G_1 \rtimes ( G_2 \rtimes ( \cdots \rtimes G_n))$$
for some graph products $G_1, \ldots, G_n$ whose vertex-groups are isomorphic to groups in $\mathcal{G}$. Moreover, for each semidirect product, the action on the corresponding graph product permutes its vertex-groups. \qed
\end{cor}

\noindent
We emphasize that, in the decomposition provided by Corollary~\ref{cor:DecompositionQuandle}, the graph products may be done over infinite graphs, even when the quandle product is finitely generated.

\medskip \noindent
We now turn to the proof of Theorem~\ref{thm:SemiDirect}. Our first preliminary lemma highlights the specific role played by the minimal elements of our oposet. 

\begin{lemma}\label{lem:RotativeStab}
Let $\mathfrak{Q}:= \mathfrak{Q}(\mathcal{I}, \mathcal{G}, \mathcal{A})$ be a quandle product acting on its quasi-median Cayley graph $\mathfrak{M}:= \mathfrak{M}(\mathcal{I}, \mathcal{G}, \mathcal{A})$. For every $i \in I$, let $J_i$ denote the hyperplane containing the clique $G_i$.
\begin{itemize}
	\item For every $i \in I_\mathrm{min}$, $\mathrm{stab}_\circlearrowright(J_i)= G_i$.
	\item For all $i,j \in I_\mathrm{min}$ and $g,h \in \mathfrak{Q}$, if $gJ_i$ and $hJ_j$ are transverse, then $\mathrm{stab}_\circlearrowright(gJ_i)$ and $\mathrm{stab}_\circlearrowright (hJ_j)$ commute.
\end{itemize}
\end{lemma}

\begin{proof}
We start by proving the following claim, where we refer to a minimal generator as an element of a factor $G_i$ with $i \in I_\mathrm{min}$. 

\begin{claim}\label{claim:MinimalGenerators}
For every $i \in I_\mathrm{min}$, all the edges of $J_i$ are labelled by minimal generators.
\end{claim}

\noindent
As a consequence of Lemma~\ref{lem:QuandleCycle}, if an edge of an induced $4$-cycle is labelled by a generator $a \in G_i$, then the opposite edge is labelled either by $a$ itself or by $b \ast a$ where $b \in G_j$ for some $j>i$. Since $I_\mathrm{min}$ is a stable subset of $I$, it follows that, if an induced $4$-cycle has an edge labelled by a minimal generator, then the opposite edge is also labelled by a minimal generator. We also know from Lemma~\ref{lem:QuandleClique} that all the edges of a $3$-cycle are labelled by generators coming from the same factor, so, if a $3$-cycle has one edge labelled by a minimal generator, then all its edges are labelled by minimal generators. These observations prove Claim~\ref{claim:MinimalGenerators}. 

\medskip \noindent
Now, given an $i \in I_\mathrm{min}$, we want to prove that $\mathrm{stab}_\circlearrowright(J_i)= G_i$. It suffices to show that, given two arbitrary cliques $A,B \subset J_i$, the inclusion $\mathrm{stab}(A) \subset \mathrm{stab}(B)$. Let $g,h \in \mathfrak{Q}$ and $j,k \in I$ be such that $A=gG_j$ and $B=hG_k$. Up to replacing $h$ with another representative in $hG_k$, we can assume that the vertices $g \in A$ and $h \in B$ belong to the same fibre of $J_i$. For every $a \in G_j$, we know from Lemma~\ref{lem:Ladder} that we can find an $b \in G_k$ such that the oriented edges $(g,ga)$ and $(h,hb)$ are connected by a ladder.
\begin{center}
\includegraphics[width=0.7\linewidth]{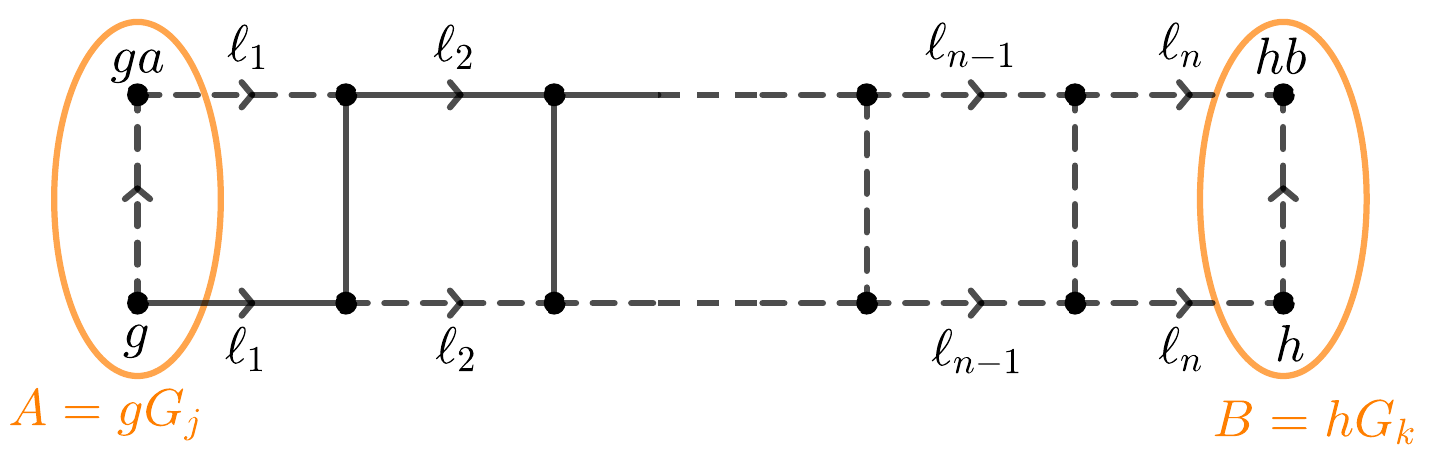}
\end{center}

\noindent
Let $\ell_1 \cdots \ell_n$ denote the word labelling the bottom path connecting $g$ to $h$ in the ladder. We know from Claim~\ref{claim:MinimalGenerators} that all the rungs of our ladders are labelled by minimal generators. It follows from Lemma~\ref{lem:QuandleCycle} that the edges of the top path connecting $ga$ to $hb$ in the ladder are labelled by $\ell_1, \ldots, \ell_n$. It follows that 
$$ag^{-1}h = a \ell_1 \cdots \ell_n = \ell_1 \cdots \ell_n b= g^{-1}h b,$$
hence $gag^{-1} = hbh^{-1}$. Thus, we have proved that 
$$\mathrm{stab}(A) = gG_jg^{-1} \leq hG_kh^{-1} = \mathrm{stab}(B),$$
as desired. The first item of our lemma is proved.

\medskip \noindent
Next, let $i,j \in I_\mathrm{min}$ and $g,h \in \mathfrak{Q}$ such that $gJ_i$ and $hJ_j$ are transverse. So there exist two intersecting cliques $A \subset gJ_i$ and $B \subset hJ_j$ that span a prism. Up to translating by an element of $\mathfrak{Q}$, we can assume for simplicity that $1 \in A \cap B$. An element $a \in \mathrm{stab}_\circlearrowright (gJ_i)$ (resp.\ $b \in \mathrm{stab}_\circlearrowright(hJ_j)$) must stabilise $A$ (resp.\ $B$), so the two neighbours $a$ and $b$ of $1$ must span an induced $4$-cycle $Q$. Let $p,q \in I$ be such that $a \in G_p$ and $b \in G_q$. We know from Claim~\ref{claim:MinimalGenerators} that $p,q \in I_\mathrm{min}$ and from Lemma~\ref{lem:QuandleCycle} that $p,q$ are $<$- or $\perp$-comparable. Hence $p \perp q$, which implies that $a$ and $b$ necessarily commute. Thus, we have proved the second item of our lemma. 
\end{proof}

\noindent
Our second, and last, preliminary lemma shows that parabolic subgroups in quandle products are gated-cocompact. 

\begin{lemma}\label{lem:ParabolicGated}
Let $(\mathcal{I}, \mathcal{G}, \mathcal{A})$ be a quandle system. For every stable $R \subset I$, the subgraph $\langle R \rangle$ in the quasi-median Cayley graph $\mathfrak{M}:= \mathfrak{M}(\mathcal{I}, \mathcal{G}, \mathcal{A})$ is gated.
\end{lemma}

\begin{proof}
First, it is clear that the subgraph $\langle R \rangle$ is connected. Next, we claim that $\langle R \rangle$ \emph{absorbs $3$-cycles}, i.e.\ every $3$-cycle with at least one edge in $\langle R \rangle$ must be entirely contained in $\langle R \rangle$. Indeed, fix a $3$-cycle $\{x,y,z\}$ with $\{x,y\} \subset \langle R \rangle$. We can write $y=xa$ with $a \in G_i$ for some $i \in I$. Since $x,y \in \langle R \rangle$, necessarily $a \in G_i$, which implies that $i \in R$ according to Corollary~\ref{cor:FactorsEmb}. Then, we deduce from Corollary~\ref{cor:QuandleClique} that $\{x,y,z\}$ is contained in $xG_i$. But $x G_i \subset \langle R \rangle$ since $x \in \langle R \rangle $and $i \in R$, concluding the proof of our claim.

\medskip \noindent
Finally, we claim that $\langle R \rangle$ is \emph{locally convex}, i.e.\ every $4$-cycle with two consecutive edges in $\langle R \rangle$ must be entirely contained in $\langle R \rangle$. Indeed, let $Q$ be such a $4$-cycle. Up to translating $Q$ with an element of $\langle R \rangle$, we assume for simplicity that $1$ is a vertex of $Q$ with two neighbours in $\langle R \langle$, say $a$ and $b$. Let $i,j \in I$ be such that $a \in G_i$ and $b \in G_j$. As previously, we know from Corollary~\ref{cor:FactorsEmb} that $i,j \in R$ since the edges $\{1,a\}$ and $\{1,b\}$ are contained in $\langle R \rangle$. According to Lemma~\ref{lem:QuandleCycle}, the fourth vertex of $Q$ is $ab$ or $ba$, but it any case it has to belong to $\langle R \rangle$. So $Q$ must be contained in $\langle R \rangle$, concluding the proof of our claim.

\medskip \noindent
Thus, we have proved that the subgraph $\langle R \rangle$ is connected, that is absorbs $3$-cycles, and that is locally convex. It follows from \cite[Proposition~2.6]{QM} that $\langle R \rangle$ is a gated subgraph. 
\end{proof}

\begin{proof}[Proof of Theorem~\ref{thm:SemiDirect}.]
It is clear that $I\backslash I_\mathrm{min}$ is a stable subset of $I$. Therefore, Lemma~\ref{lem:ParabolicGated} implies that the subgraph $\langle I \backslash I_\mathrm{min} \rangle$ is gated. Let $\mathcal{W}$ denote the collection of the hyperplanes tangent to $\langle I \backslash I_\mathrm{min} \rangle$ (i.e.\ containing the edges with exactly one endpoint in $\langle I \backslash I_\mathrm{min} \rangle$). Notice that
$$\mathcal{W}= \{ gJ_i \mid g\in \langle I \backslash I_\mathrm{min} \rangle, i \in I_\mathrm{min} \},$$
where, as usual, we denote by $J_i$ the hyperplane containing the clique $G_i$. Set 
$$\mathrm{Rot}(\mathcal{W}):= \langle \mathrm{stab}_\circlearrowright(J) \mid J \in \mathcal{W} \rangle.$$
Thanks to Lemma~\ref{lem:RotativeStab}, we can apply \cite[Theorem~3.24]{MR4033512} (see also \cite[Proposition~6.8]{FBn}) and deduce that
$$\mathfrak{Q}= \mathrm{Rot} (\mathcal{W}) \rtimes \mathrm{stab}( \langle I \backslash I_\mathrm{min} \rangle)$$
where $\mathrm{Rot}(\mathcal{W})$ decomposes as a graph product whose vertex-groups are isomorphic to $G_i$ for some $i \in \mathrm{I}_\mathrm{min}$ and where $\langle I \backslash I_\mathrm{min} \rangle$ acts on $\mathrm{Rot}(\mathcal{W})$ by permuting its vertex-groups. Since $\mathrm{stab}( \langle I \backslash I_\mathrm{min} \rangle)= \langle I \backslash I_\mathrm{min} \rangle$, this concludes the proof of our theorem. 

\medskip \noindent
It is worth mentioning that, according to \cite[Theorem~3.24]{MR4033512}, $\mathrm{Rot}(\mathcal{W})$ is a graph product over an explicit graph: the \emph{crossing graph} of $\mathcal{W}$, namely the graph whose vertex-set is $\mathcal{W}$ and whose edges connect two hyperplanes whenever they are transverse. Hence

\begin{fact}\label{fact:FiniteClique}
The underlying graph $\Gamma$ defining the graph product $\mathfrak{Q}_\mathrm{min}$ satisfies
$$\mathrm{clique}(\Gamma) \leq \dim_\square \mathfrak{M} \leq \#\{ n \geq 0 \mid \exists i_1, \ldots, i_n \in I \text{ pairwise $<$- or $\perp$-comparable} \}.$$
\end{fact}

\noindent
Recall that the clique-number $\mathrm{clique}(\cdot)$ of a graph refers to the maximal size (possibly infinite) of a complete subgraph. The upper bound on the cubical dimension of $\mathfrak{M}$ is given by Theorem~\ref{thm:QuandleQM}. 
\end{proof}

\noindent
Let us record a few concrete consequences of Theorem~\ref{thm:SemiDirect}. 

\begin{cor}\label{cor:CombApplications}
Let $\mathfrak{Q}:= \mathfrak{Q}(\mathcal{I}, \mathcal{G}, \mathcal{A})$ be a quandle product with $\mathcal{I}$ finite. 
\begin{itemize}
	\item[(i)] For every prime $p \geq 2$, $\mathfrak{Q}$ contains an element of order $p$ if and only there exists some $G \in \mathcal{G}$ containing an element of order $p$. Consequently, $\mathfrak{Q}$ is torsion-free if and only if all its factors are torsion-free.  
	\item[(ii)] $\mathfrak{Q}$ is orderable if and only if all its factors are orderable. 
	\item[(iii)] $\mathfrak{Q}$ satisfies the Tits alternative if and only if all its factors satisfy the Tits alternative.
	\item[(iv)] $\mathfrak{Q}$ has finite asymptotic dimension if and only if all its factors have finite asymptotic dimension. 
\end{itemize}
\end{cor}

\noindent
In the item (iii), we say that a group satisfies the \emph{Tits alternative} if every subgroup either contains a non-abelian free subgroup or is virtually solvable. 

\begin{proof}[Proof of Corollary~\ref{cor:CombApplications}.]
We start by proving (i). Since factors are embedded in $\mathfrak{Q}$ according to Corollary~\ref{cor:FactorsEmb}, it is clear that, if a factor contains an element of order $p$, then so does $\mathfrak{Q}$. Conversely, assume that $\mathfrak{Q}$ contains an element of order $p$. Decompose $\mathfrak{Q}$ as $\mathfrak{Q}_\mathrm{min} \rtimes \langle I \backslash I_\mathrm{min} \rangle$ as in Theorem~\ref{thm:SemiDirect}. According to Claim~\ref{claim:SemiFiniteOrder} below, either $\mathfrak{Q}_\mathrm{min}$ or $\langle I \backslash I_\mathrm{min} \rangle$ contains an element of order $p$. In the latter case, since we know from Corollary~\ref{cor:FactorsEmb} that $\langle I \backslash I_\mathrm{min} \rangle$ is a quandle product with fewer factors than $\mathfrak{Q}$, we can argue by induction and deduce that $\langle I \backslash I_\mathrm{min} \rangle$ has a factor containing an element of order $p$. But the factors of $\langle I \backslash I_\mathrm{min} \rangle$ are factors of $\mathfrak{Q}$, hence the desired conclusion follows. In the former case, it follows from Claim~\ref{claim:GPfiniteOrder} below that the graph product $\mathfrak{Q}_\mathrm{min}$ has a vertex-group containing an element of order $p$. But the vertex-groups of $\mathfrak{Q}_\mathrm{min}$ are factors of $\mathfrak{Q}$, hence the desired conclusion.

\begin{claim}\label{claim:SemiFiniteOrder}
If a semidirect product $N \rtimes H$ contains an element of order $p$, then either $N$ or $H$ contains an element of order $p$. 
\end{claim}

\noindent
Let $g \in N \rtimes H$ be an element of order $p$. If its projection $h$ to $H$ is non-trivial, then $h$ yields an element of order $p$ in $H$. Otherwise, if $g$ projects trivially to $H$, then $g$ must belong to $N$, providing an element of order $p$ in $N$. 

\begin{claim}\label{claim:GPfiniteOrder}
If a graph product contains an element of order $p$, then some vertex-group contains an element of order $p$. 
\end{claim}

\noindent
Let $\Gamma$ be a graph and $\mathcal{G}$ a collection of groups indexed by $V(\Gamma)$. According to \cite[Lemma~4.5]{GreenGP}, a finite subgroup $H \leq \Gamma \mathcal{G}$ must be contained in $\langle \Lambda \rangle$ for some complete subgraph $\Lambda \leq \Gamma$. Thus, if $\Gamma \mathcal{G}$ contains an element of order $p$, then there must exist a complete subgraph $\Lambda \leq \Gamma$ such that $\langle \Lambda \rangle$ contains an element of order $p$. Since $\langle \Lambda \rangle = \bigoplus_{u \in V(\Lambda)} G_u$, we conclude from Claim~\ref{claim:SemiFiniteOrder} that some vertex-group of $\Gamma \mathcal{G}$ contains an element of order $p$. This concludes the proof of Claim~\ref{claim:GPfiniteOrder}. 

\medskip \noindent
Item (ii) (resp.\ (iii)) is proved similarly as being orderable (resp.\ satisfying the Tits alternative) is stable under graph products \cite{MR2946302} (resp.\ \cite{MR3365774}) and semi-direct products.

\medskip \noindent
Finally, let us prove (iv). If $\mathfrak{Q}$ has finite asymptotic dimension, then the same holds for its subgroups, including its factors. Conversely, assume that every factor has finite asymptotic dimension. Decompose $\mathfrak{Q}$ as $\mathfrak{Q}_\mathrm{min} \rtimes \langle I \backslash I_\mathrm{min} \rangle$ as in Theorem~\ref{thm:SemiDirect}. Since $\langle I \backslash I_\mathrm{min} \rangle$ is a quandle product with fewer factors as $\mathfrak{Q}$ and whose factors are also factors of $\mathfrak{Q}$, we can argue by induction on the number of factors and deduce that $\langle I \backslash I_\mathrm{min} \rangle$ has finite asymptotic dimension. According to \cite{MR2213160}, it suffices to know that $\mathfrak{Q}_\mathrm{min}$ has finite asymptotic dimension in order to conclude that $\mathfrak{Q}$ has finite asymptotic dimension. The key observation is that $\mathfrak{Q}_\mathrm{min}$ is a graph product over a graph with finite clique-number (Fact~\ref{fact:FiniteClique}) and whose vertex-groups are factors of $\mathfrak{Q}$. Then, we can apply \cite{AsdimGP} to conclude. 
\end{proof}

\begin{remark}
The item (i) in Corollary~\ref{cor:CombApplications} provides valuable information on the torsion in quandle products, but it would be more satisfying to have a complete description of finite subgroups. See Problem~\ref{prob:FiniteSub} and the related discussion. 
\end{remark}

\begin{remark}
In view of the item (ii) in Corollary~\ref{cor:CombApplications}, it is worth mentioning that a quandle product of bi-orderable groups may not be bi-orderable. For instance, 
$$\langle a,b,t \mid ab=ba, \ ta=b^{-1}t, \ tb=a^{-1}t \rangle$$
is the presentation of a quandle product of three infinite cyclic groups that is not bi-orderable (since it contains a non-trivial element conjugate to its inverse). 
\end{remark}

\noindent
We conclude this section with a by-product of the results proved so far.

\begin{prop}\label{prop:QuasiRetract}
Let $\mathfrak{Q}:= \mathfrak{Q}(\mathcal{I}, \mathcal{G}, \mathcal{A})$ be a quandle product with trivial holonomy, with $\mathcal{I}$ finite, and with every group in $\mathcal{G}$ finitely generated. Every parabolic subgroup of $\mathfrak{Q}$ is a quasi-retract. 
\end{prop}

\noindent
Recall that, given a finitely generated group $G$, a subgroup $H \leq G$ is a \emph{quasi-retract} if there exists a Lipschitz map $\rho : G \to H$ such that $\rho_{|H}$ lies at finite distance from $\mathrm{id}_H$. 

\begin{proof}[Proof of Proposition~\ref{prop:QuasiRetract}.]
Let $P \leq \mathfrak{Q}$ be a parabolic subgroup. Without loss of generality, we can assume that $P$ is a standard parabolic subgroup. According to Lemma~\ref{lem:StableClosure}, there exists a stable subset $R \subset I$ such that $P= \langle R \rangle$.

\medskip \noindent
Let $S_i \in subset G_i$, $i \in I$, be finite generating sets as given by Lemma~\ref{lem:GeneratingNoHolo}. According to Theorem~\ref{thm:MedianCayl}, $\mathrm{Cayl}(\mathfrak{Q}, S:= \bigcup_{i \in I} S_i)$ can be identified with $(\mathfrak{M}, \delta)$ for some global metric $\delta$ defined on the quasi-median Cayley graph $\mathfrak{M}:= \mathfrak{M}(\mathcal{I}, \mathcal{G}, \mathcal{A})$ of $\mathfrak{Q}$. The combination of Lemmas~\ref{lem:ParabolicGated} and ~\ref{lem:GatedGated} shows that $\langle R \rangle$ defines a gated subgraph in $\mathrm{Cayl}(\mathfrak{Q},S)$. 

\medskip \noindent
Thus, the gated-projection $\mathrm{Cayl}(\mathfrak{Q},S) \to \langle R \rangle$ yields the desired quasi-retraction $\mathfrak{Q} \to P$ on our parabolic subgroup.  
\end{proof}

\section{Graph products as subgroups}

\noindent
We saw in Section~\ref{section:SemiDecomposition} that, given a quandle product $\mathfrak{Q}:= \mathfrak{Q}(\mathcal{I}, \mathcal{G}, \mathcal{A})$, conjugates of factors indexed by minimal elements of $\mathcal{I}$ usually generate graph products. In this section, we show that, under reasonable assumptions, the same holds for arbitrary factors up to taking finite-index subgroups. More precisely:

\begin{thm}\label{thm:SubGP}
Let $\mathfrak{Q}:= \mathfrak{Q}(\mathcal{I}, \mathcal{G}, \mathcal{A})$ be a quandle product. Assume that the following conditions are satisfied:
\begin{itemize}
	\item the oposet $\mathcal{I}$ is finite;
	\item every group in $\mathcal{G}$ is finitely generated and residually finite;
	\item for every $i \in I$, the kernel of $G_i \curvearrowright \bigsqcup_{j<i} G_j$ has finite-index in $G_i$. 
\end{itemize}
For all conjugates of factors $F_1, \ldots, F_n$, there are finite-index subgroups $F_i' \leq F_i$ such that $\{F_1', \ldots, F_n'\}$ is the basis of a graph product. 
\end{thm}

\noindent
Here, given a group $G$ and subgroups $H_1, \ldots, H_n \leq G$, we say that $\{H_1, \ldots, H_n\}$ is \emph{the basis of a graph product} if the identity maps $H_i \to H_i$ induce an injective morphism $\Gamma \mathcal{H} \hookrightarrow G$ where $\Gamma$ is the commutation graph of $\{H_1, \ldots, H_n\}$ (i.e.\ the graph whose vertices are $H_1, \ldots, H_n$ and whose edges connect $H_i$ and $H_j$ whenever they commute) and where $\mathcal{H}:= \{H_1, \ldots, H_n\}$.

\medskip \noindent
The proof of Theorem~\ref{thm:SubGP} is based on the following criterion, which follows from (the proof of) \cite[Theorem~8.43]{QM}.

\begin{prop}\label{prop:PPforSubGP}
Let $G$ be a group acting on a quasi-median graph $X$ with trivial vertex-stabilisers. For all hyperplanes $J_1, \ldots, J_n$ and all residually finite subgroups $S_1 \leq \mathrm{stab}_\circlearrowright(J_1), \ldots, S_n \leq \mathrm{stab}_\circlearrowright(J_n)$, there exist finite-index subgroups $\dot{S}_1 \leq S_1, \ldots, \dot{S}_n \leq S_n$ such that $\{ \dot{S}_1, \ldots, \dot{S}_n \}$ is the basis of a graph product.
\end{prop}

\begin{proof}[Proof of Theorem~\ref{thm:SubGP}.]
For every $i \in I$, let $J_i$ denote the hyperplane containing the clique $G_i$. In view of Proposition~\ref{prop:PPforSubGP}, it suffices to show that, for every $i \in I$, $G_i$ contains a finite-index subgroup $\dot{G}_i$ that lies in $\mathrm{stab}_\circlearrowright (J_i)$. We start by setting
$$K_i:= \mathrm{ker} \left( G_i \curvearrowright \bigsqcup\limits_{j<i} G_j \right) \text{ for every } i \in I.$$
By assumption, we know that each $K_i$ has finite index in $G_i$. Then, we set
$$\dot{G}_i:= \bigcap\limits_{j \in I} \bigcap\limits_{\varphi \in \mathscr{S}(i,j)} \varphi^{-1}(K_j) \text{ for every } i \in I.$$
Notice that, as a direct consequence of our definition:

\begin{fact}\label{fact:PhiDot}
For all $i,j \in I$ and $\varphi \in \mathscr{S}(i,j)$, $\varphi( \dot{G}_i)= \dot{G}_j$.
\end{fact}

\noindent
Then, let us verify that:

\begin{claim}\label{claim:FiniteIndexDot}
For every $i \in I$, $\dot{G}_i$ has finite index in $G_i$.
\end{claim}

\noindent
Because $\mathcal{I}$ is finite, there exists some $N \geq 0$ such that $K_j$ has index $\leq N$ in $G_j$ for every $j \in I$. Therefore, $\dot{G}_i$ is an intersection of subgroups with indices $\leq N$. But $G_i$ is finitely generated, so it may contain only finitely many such subgroups. Consequently, there are only finitely many distinct subgroups in the intersection defining $\dot{G}_i$, which implies that $\dot{G}_i$ indeed has finite index in $G_i$. Claim~\ref{claim:FiniteIndexDot} is proved.

\medskip \noindent
The key observation that will allow us to prove our theorem is:

\begin{claim}\label{claim:Transparent}
For all $i,j \in I$, $g \in \dot{G}_i$, and $h \in \mathfrak{Q}$, if $h$ is $(i,j)$-diaphanous then $\dot{G}_i \leq h G_j h^{-1}$. 
\end{claim}

\noindent
It follows from Lemma~\ref{lem:PathFibreDiaph} that the clique $hG_j$ belongs to $J_i$. Thus, according to Lemma~\ref{lem:Ladder}, there exists a ladder connecting the edge $\{ 1, g^{-1}\}$ to an edge of $hG_j$. In other words, we have 
\begin{center}
\includegraphics[width=0.6\linewidth]{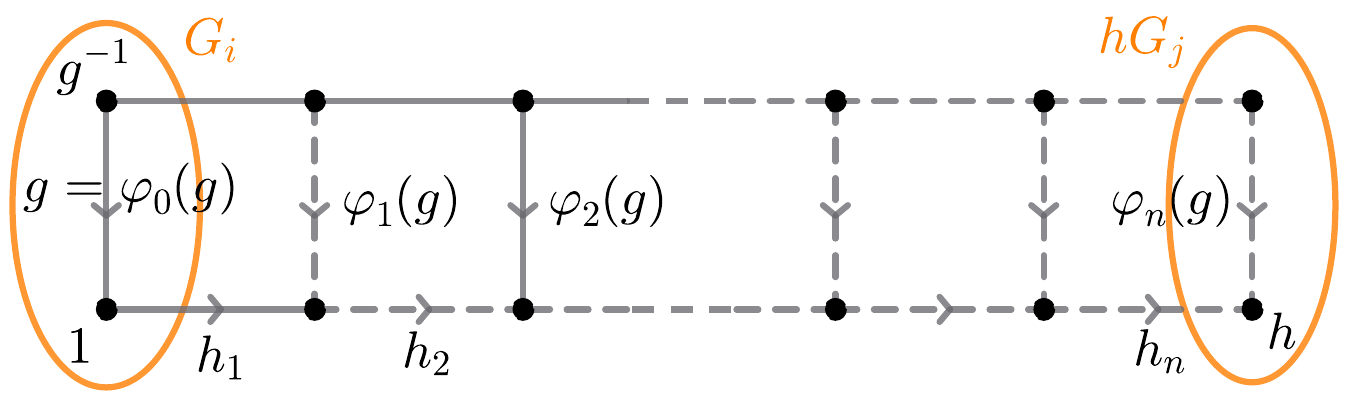}
\end{center}

\noindent
where $h_1 \cdots h_n$ is a $\mathfrak{Q}$-reduced word representing $h$. Notice that, as a consequence of Lemma~\ref{lem:ParallelEdges}, the vertical edges of our ladder are labelled by generators of the form $\varphi(g)$ with $\varphi \in \mathscr{S}(i, \ell)$; observe that, according to Fact~\ref{fact:PhiDot}, such a generator must belong to $\dot{G}_\ell$ as $g \in \dot{G}_i$. This allows us to deduce from Lemma~\ref{lem:ParallelEdges} that the edge opposite to the horizontal edge labelled by $h_k$ is labelled by either $h_k$ or $\varphi_{k-1}(g) \ast h_k = h_k$. Thus, the top horizontal line of our ladder is also labelled by the word $h_1 \cdots h_n$. This implies that $gh = h g'$ for some $g' \in G_j$. The proof of Claim~\ref{claim:Transparent} is complete. 

\medskip \noindent
We know from Lemma~\ref{lem:PathFibreDiaph} that the cliques of $J_i$ are exactly the $hG_j$ where $h$ is $(i,j)$-diaphanous. Consequently,
$$\mathrm{stab}_\circlearrowright(J_i)= \bigcap\limits_{j \in I} \bigcap\limits_{h \text{ $(i,j)$-diaphanous}} hG_jh^{-1}.$$
According to Claim~\ref{claim:Transparent}, $\mathrm{stab}_\circlearrowright (J_i)$ contains $\dot{G}_i$, which has finite index in $G_i$ according to Claim~\ref{claim:FiniteIndexDot}. This concludes the proof of our theorem. 
\end{proof}

\noindent
As a particular case of Theorem~\ref{thm:SubGP}, we get the following statement:

\begin{cor}\label{cor:SubRAAG}
Let $\mathfrak{Q}:= \mathfrak{Q}(\mathcal{I}, \mathcal{G},\mathcal{A})$ be a quandle product. Assume that $\mathcal{I}$ finite and that every group in $\mathcal{G}$ is cyclic. For all conjugates of infinite-order generators $z_1, \ldots, z_n$, there exists $N \geq 1$ such that $\{z_1^N, \ldots, z_n^N\}$ is a basis of a right-angled Artin group. \qed
\end{cor}

\section{Examples}\label{section:Examples}

\noindent
In this section, we record various examples of quandle products.

\subsection{Twisted graph products}

\noindent
We saw in Example~\ref{ex:GP} that quandle products $\mathfrak{Q}(\mathcal{I}, \mathcal{G}, \mathcal{A})$ for which $\mathcal{I}=(I, =, \perp)$, and consequently $\mathcal{A}= \emptyset$, correspond to graph products. Of course, they have trivial holonomy. 

\medskip \noindent
A related but more general situation is where $\leq$ is allowed to be distinct from $=$ but all the actions $G_i \curvearrowright \mathrm{Obj} \left( \bigsqcup_{j<i} G_j \right)$ coming from $\mathcal{A}$ are supposed to be trivial. Then, the corresponding quandle product is an example of a \emph{twisted graph product}. 

\begin{definition}
Let $\Gamma$ be an oriented graph, $\mathcal{G}= \{G_u \mid u \in V(\Gamma)\}$ a collection of groups, and $\mathcal{A}=\{ G_u \curvearrowright G_v \mid (u,v) \in \vec{E}(\Gamma)\}$ a collection of actions. The \emph{twisted graph product} $\Gamma(\mathcal{G}, \mathcal{A})$ is 
$$\left( \underset{u \in V(\Gamma)}{\ast} G_u \right) / \langle\langle ab=b(b\ast a) \text{ for all } (u,v) \in E(\Gamma), a \in G_u, b \in G_v \rangle\rangle.$$
\end{definition}

\noindent
Twisted graph products of infinite cyclic groups coincide with twisted right-angled Artin groups as introduced in \cite{MR856848, MR2672155}. Twisted graph products will be studied in full generality in a forthcoming work \cite{BuildingQM}, in connection with right-angled buildings.

\medskip \noindent
Notice that, if $\Gamma(\mathcal{G}, \mathcal{A})$ is a twisted graph product coming from a quandle product, then $\Gamma$ must be \emph{transitively oriented} (i.e.\ for all vertices $x,y,z \in V(\Gamma)$, if $(x,y), (y,z) \in \vec{E}(\Gamma)$, then $(x,z) \in \vec{E}(\Gamma)$). In fact, not every twisted graph product is (at least naturally) a quandle product: indeed, there exist twisted right-angled Artin groups with torsion, while quandle products of infinite cyclic groups are torsion-free according to Corollary~\ref{cor:CombApplications}.

\subsection{Wreath products}\label{section:Wreath}

\noindent
In this section, we focus on quandle products whose oposets contain greatest elements. 

\begin{definition}
A \emph{quandle-wreath product} is a quandle product $\mathfrak{Q}(\mathcal{I}, \mathcal{G}, \mathcal{A})$ for which there exists some $i_\mathrm{max} \in I$ such that $i \leq i_\mathrm{max}$ for every $i \in I$. 
\end{definition}

\noindent
Clearly, for such a quandle product, $\mathfrak{Q}$ decomposes as the semidirect product $\langle I \backslash \{i_\mathrm{max}\} \rangle \rtimes \langle i_\mathrm{max} \rangle$, where $\langle i_\mathrm{max} \rangle$ acts on the quandle product $\langle I \backslash \{ i_\mathrm{max}\} \rangle$ by permuting its factors according to the action of $G_{i_\mathrm{max}}$ on $\bigsqcup_{j \neq i_\mathrm{max}} G_j$ coming from $\mathcal{A}$.   

\medskip \noindent
Such a structure echoes the definition of permutation wreath products, motivating our terminology. In fact, it is not difficult to describe any permutation wreath products as quandle-wreath products. More generally, let us consider the following family of groups:

\begin{definition}[\cite{MR3393469}]
Let $H$ be a group acting on a graph $\Gamma$ and $A$ another group. The \emph{graph-wreath product} $A \wr_\Gamma H$ is $\Gamma A \rtimes H$ where $H$ acts on the graph product $\Gamma A$ by permuting the vertex-groups according to its action on the vertices $\Gamma$. 
\end{definition}

\noindent
We emphasize that, in the definition above, $H$ sends copies of $A$ to copies of $A$ identically. 

\begin{prop}
Graph-wreath products are quandle-wreath products with trivial holonomy.
\end{prop}

\begin{proof}
Let $A \wr_\Gamma H$ be a graph-wreath product. Let $\mathcal{I}$ denote the oposet $(V(\Gamma) \sqcup \{H\}, \leq ,\perp)$ where $\leq$ is defined by: $x \leq y$ if and only if $x \in V(\Gamma)$ and $y=H$; and where $\perp$ is defined by: $x \perp y$ if and only if $x,y \in V(\Gamma)$ and $\{x,y\} \in E(\Gamma)$. Let $\mathcal{G}$ be the collection of groups $\{G_u=A \text{ for every } u \in V(\Gamma), G_H=H\}$, and finally set $\mathcal{A} = \{ G_H \curvearrowright \bigsqcup_{u \in V(\Gamma)} G_u \}$ where $G_H$ permutes the $G_u$ according to its action on $V(\Gamma)$. It is clear that $(\mathcal{I}, \mathcal{G}, \mathcal{A})$ is a quandle system. Moreover, the presentation defining the corresponding quandle product $\mathfrak{Q}$ shows directly that $\mathfrak{Q} \simeq A \wr_\Gamma H$. Because $G_H=H$ sends copies of $A$ to copies of $A$ identically, it is clear that the holonomy is trivial. 
\end{proof}

\noindent
Let us mention a couple of remarkable examples of graph-wreath products.

\begin{ex}
Given two groups $A$ and $B$, define the product $A \bullet B$ from the relative presentation 
$$\langle A,B , t \mid [A, t^nBt^{-n}]=1 \text{ for every } n\geq 0 \rangle,$$
where $[X,Y]=1$ is a shorthand for: $[x,y]=1$ for all $x \in X$ and $y \in Y$. The $\bullet$-product of two groups can be defined as a diagram product \cite[Example~10.64]{QM} and as a right-angled graph of groups \cite[Section~11]{QM}.  The group $\mathbb{Z} \bullet \mathbb{Z}$ coincides with the Bestvina-Brady subgroup of $\mathbb{F}_2 \times \mathbb{F}_2$, i.e.\ the kernel of the morphism $\mathbb{F}_2 \times \mathbb{F}_2 \twoheadrightarrow \mathbb{Z}$ that sends generators to $1$. Thus, $\mathbb{Z} \bullet \mathbb{Z}$ naturally embeds into $\mathbb{F}_2 \times \mathbb{F}_2$. More generally, as noticed in \cite[Example~4.29]{MR4586831}, $A \bullet B$ embeds into $(A \ast \mathbb{Z}) \times (B \ast \mathbb{Z})$. The alternative description
$$\langle A_i \ (i \in \mathbb{Z}), B_i \ (i \in \mathbb{Z}), t \mid tA_it^{-1}=A_{i+1}, tB_it^{-1}= B_{i+1}, [A_i,B_j]=1 \text{ if } i<j \rangle$$
$$=\langle A_i \ (i \in \mathbb{Z}), B_i \ (i \in \mathbb{Z}) \mid  [A_i,B_j]=1 \text{ if } i<j \rangle \rtimes \langle t \rangle$$
of $A \bullet B$, where the $A_i$ (resp.\ $B_i$), $i \in \mathbb{Z}$, are copies of $A$ (resp.\ of $B$), shows that $A \bullet B$ is a graph-wreath product. 
\end{ex}

\begin{ex}
Given two groups $A$ and $B$, define the product $A \square B$ from the relative presentation
$$\langle A,B,t \mid [A,B]= [A, tBt^{-1}]=1 \rangle.$$
The $\square$-product of two groups can also be defined both as a diagram product \cite[Example~10.65]{QM} and as a right-angled graph of groups \cite[Example~11.40]{QM}. The group $\mathbb{Z} \square \mathbb{Z}$ was introduced in \cite{MR4071367} as the example of a cocompact diagram group that is not a right-angled Artin group. The alternative description
$$\langle A_i \ (i \in \mathbb{Z}), B_i \ (i \in \mathbb{Z}), t \mid tA_it^{-1}=A_{i+1}, tB_it^{-1}= B_{i+1}, [A_i,B_j]=1 \text{ if } j \in \{i,i+1\} \rangle$$
$$= \langle A_i \ (i \in \mathbb{Z}), B_i \ (i \in \mathbb{Z}) \mid [A_i,B_j]=1 \text{ if } j \in \{i,i+1\} \rangle \rtimes \langle t \rangle$$
of $A \square B$, where the $A_i$ (resp.\ $B_i$), $i \in \mathbb{Z}$, are copies of $A$ (resp.\ of $B$), shows that $A \square B$ is a graph wreath-product. 
\end{ex}

\noindent
In the definition of permutation wreath products, all the lamp-groups are taken to be equal. But it is possible to take distinct groups in lamps that lie in distinct orbits. More precisely, let $H$ be a group acting on a set $S$. Let $\{S_i \mid i \in I\}$ denote the $H$-orbits in $S$. For every $i \in I$, fix a group $F_i$. Then, define
$$\left( \bigoplus\limits_{i \in I} \bigoplus_{S_i} F_i \right) \rtimes H,$$
where $H$ permutes the factors in the direct sum according to its action on $S$. Formally, this is no longer a permutation wreath product, but it is a quandle product. As a funny particular case:

\begin{ex}
Given two groups $A$ and $B$, consider the semidirect product
$$\left( \bigoplus\limits_{\mathbb{Z}} (A \oplus B) \right) \rtimes \mathbb{Z}$$
where $\mathbb{Z}$ translates to the left (resp.\ right) the copies of $A$ (resp.\ $B$). An element of this group can be represented as a configuration of lamps with two lamplighters:
\begin{center}
\includegraphics[width=0.7\linewidth]{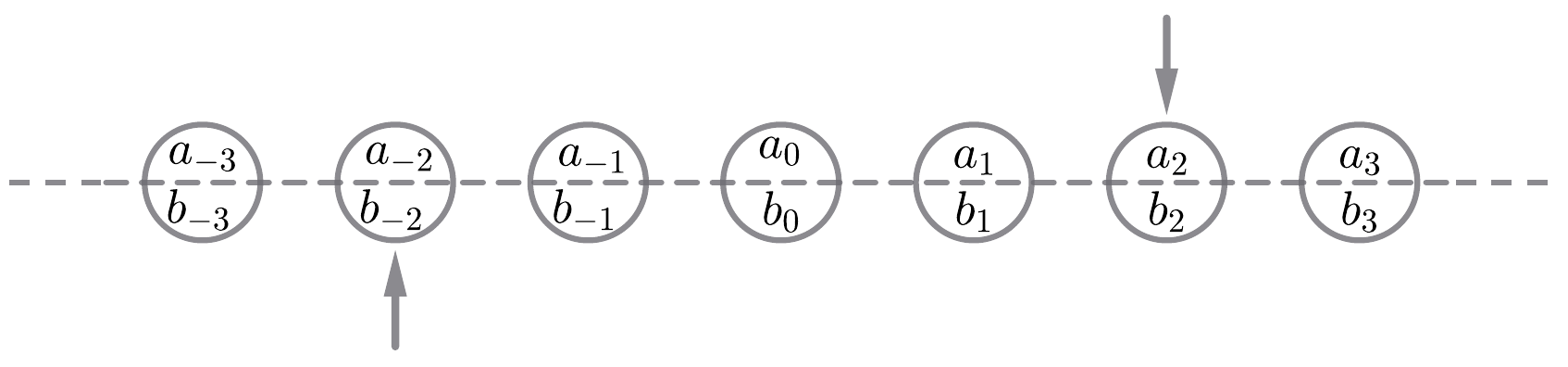}
\end{center}

\noindent
where each lamp has a top colour in $A$ and a bottom colour in $B$, and where shifting the top lamplighter to the right (resp.\ left) also shifts the bottom lamplighter to the left (resp.\ right). 
\end{ex}

\noindent
Also, in the definition of permutation wreath products, the lamp-groups are permuted identically. But we can add some holonomy. For instance:

\begin{ex}
Let $n \geq 1$ be an integer and $H$ a group with a fixed order $\leq$ invariant under conjugation. Define
$$L:= \langle a_h \ (h \in H) \mid a_h^{2n}=1 \ (h \in H), \ a_ha_k = a_k^{-1}a_h \ (h>k) \rangle.$$
Because $\leq$ is invariant under conjugation, $H$ acts on $L$ via
$$ha_kh^{-1} = a_{hk} \text{ for all } h,k \in H.$$
The semidirect product $L \rtimes H$ is not a permutation wreath product, but it is a quandle product. The example is taken from \cite[Example~4.42]{Halo}, where it is used, with $H$ finitely generated, as an example of a group quasi-isometric to $\mathbb{Z}_{2n} \wr H$ but not commensurable to any lamplighter over $H$. 
\end{ex}

\noindent
Finally, notice that the lamp-cactus group mentioned in the introduction is also a quandle-wreath product.

\subsection{Cactus products}\label{section:cactus}

\noindent
Motivated by cactus groups and various variations, including those mentioned in the introduction, we define:

\begin{definition}
A \emph{cactus system} $(\mathcal{S}, \mathcal{P}, \mathcal{G})$ is the data of 
\begin{itemize}
	\item a set $\mathcal{S}$ together with a collection $\mathcal{P} = \{ (S, \iota_S) \}$ of subsets $S \subset \mathcal{S}$, each endowed with an involution $\iota_S : S \to S$;
	\item and a collection $\mathcal{G}= \{( G_i, \dot{G}_i,\alpha_i) \mid i \in \{ \# S \mid S \in \mathcal{P} \}\}$ of groups $G_i$, each endowed with an index-two subgroup $\dot{G}_i \leq G_i$ and an automorphism $\alpha_i \in \mathrm{Aut}(G_i)$ of order $\leq 2$ stabilising $\dot{G}_i$. 
\end{itemize}
The \emph{cactus product} $\mathrm{Cact}(\mathcal{S}, \mathcal{P}, \mathcal{G})$ is the quandle product:
\begin{itemize}
	\item whose oposet is $(\mathcal{P}, \subseteq, \text{disjointness} \perp)$;
	\item whose factors are $\{ G_S:= G_{\# S} \mid S \in \mathcal{P} \}$;
	\item and whose actions are defined by $$\left\{ \begin{array}{l} b \ast A = A \text{ and } b \ast a =a \text{ if } b \in \dot{G}_B \\ b \ast A = \iota_B(A) \text{ and } b \ast a =\alpha_A (a) \text{ otherwise} \end{array} \right.$$ for all $A,B \in \mathcal{P}$ satisfying $A \subseteq B$ and $a \in G_A$, $b \in G_B$.
\end{itemize}
\end{definition}

\noindent
Verifying that the quandle relation holds is just computation. Given $A,B,C \in \mathcal{P}$ satisfying $A \subset B \subset C$ and $a \in G_A$, $b \in G_B$, $c \in G_C$, we find that 
$$\left\{ \begin{array}{l} c \ast (b \ast A) = \iota_C \circ \iota_B (A) = (\iota_C \circ \iota_B \circ \iota_C^{-1}) \circ \iota_C (A) = \iota_{\iota_C(B)} \circ \iota_C(A) \\ (c \ast b) \ast (c \ast A) = \iota_{\iota_C(B)} \circ \iota_C (A) \end{array} \right.$$
and 
$$c \ast (b \ast a)  = \left\{ \begin{array}{cl} a & \text{if $b \in \dot{G}_B$ and $c \in \dot{G}_C$} \\ \alpha_A(a) & \text{if $b \in \dot{G}_B$ and $c \notin \dot{G}_C$} \\ \alpha_A(a) & \text{if $b \notin \dot{G}_B$ and $c \in \dot{G}_C$} \\ a & \text{if $b \notin \dot{G}_B$ and $c \notin \dot{G}_C$} \end{array} \right.  = (c \ast b ) \ast (c \ast a).$$
Thus, cactus products are well-defined quandle products.

\medskip \noindent
Let us mention some examples of cactus products. We start with the easiest one: (standard) cactus groups. It is worth mentioning that the higher-dimensional and oriented cactus groups mentioned in the introduced are also cactus products of cyclic groups. For higher-dimensional cactus groups, we leave the details to the interested reader.

\begin{ex}
Set $\mathcal{S} := \{1, \ldots, n\}$, $\mathcal{P}:= \{ ([i,j], \iota_{[i,j]}) \mid 1 \leq i<j \leq n \}$ where $\iota_{[i,j]}$ denotes the central symmetry of $[i,j]$, and $\mathcal{G}:= \{ (G_i:= \mathbb{Z}_2, \dot{G}_i:= \{1\}, \alpha_i:= \mathrm{id}) \mid 2 \leq i \leq n\}$. Then, the cactus product $\mathrm{Cact}(\mathcal{S}, \mathcal{P}, \mathcal{G})$ coincides with the standard cactus groups $J_n$ described in the introduction.  
\end{ex}

\begin{ex}
Set $\mathcal{S} := \{1, \ldots, n\}$, $\mathcal{P}:= \{ ([i,j], \iota_{[i,j]}) \mid 1 \leq i<j \leq n \}$ where $\iota_{[i,j]}$ denotes the central symmetry of $[i,j]$, and $\mathcal{G}:= \{ (G_i:= \mathbb{Z}, \dot{G}_i:= 2 \mathbb{Z}, \alpha_i:= - \mathrm{id}) \mid 2 \leq i \leq n\}$. Then, the cactus product $\mathrm{Cact}(\mathcal{S}, \mathcal{P}, \mathcal{G})$ coincides with the oriented cactus groups described in the introduction.  
\end{ex}

\noindent
The cactus group naturally surjects onto the symmetric group $\mathrm{Sym}(\{1, \ldots, n \})$. In \cite{MR3567259, MR3945740}, a generalised cactus group is associated to every Coxeter group. Formally, fix a Coxeter system $(W,S)$. Denote by $\mathcal{P}^{\mathrm{irr}}_{< \infty}(S)$ the collection of the subsets $R \subset S$ such that $\langle R \rangle$ is an irreducible finite subgroup, and by $\omega_R$ the longest element of $\langle R \rangle$. Then, the \emph{cactus group over $(W,S)$} is defined by the presentation
$$\left\langle \sigma_I \ (I \in \mathcal{P}^\mathrm{irr}_{<\infty}(S)) \mid \begin{array}{c} \sigma_I^2=1  \text{ for every } I, \left[ \sigma_I, \sigma_J \right] = 1 \text{ if } I \cap J = \emptyset, \\ \sigma_I \sigma_J = \sigma_J \sigma_{\omega_J I \omega_J^{-1}} \text{ whenever } I \subset J \end{array} \right\rangle.$$
Notice that the cactus group over the symmetric group coincides with the standard cactus group. As shown by the next example, generalised cactus groups are also cactus products.

\begin{ex}
Fix a Coxeter system $(W,S)$. Set $\mathcal{S}:=S$, $\mathcal{P}:= \{ (R, \iota_R) \mid R \in \mathcal{P}^\mathrm{irr}_{< \infty}(S) \}$ where $\iota_R$ is defined by $x \mapsto \omega_R x \omega_R^{-1}$, and $\mathcal{G}:= \{ ( G_i:= \mathbb{Z}_2, \dot{G}_i:= \{1\}, \alpha_i:= \mathrm{id}) \}$. Then, the cactus product $\mathrm{Cact}(\mathcal{S}, \mathcal{P}, \mathcal{G})$ coincides with the cactus group over $(W,S)$. 
\end{ex}

\noindent
In the spirit of graph braid groups \cite{MR2701024} and the diagrammatic interpretation of the affine cactus group \cite{AffineCactus}, we can also define \emph{graph cactus groups}:

\begin{ex}
Let $\Gamma$ be a graph. Set $\mathcal{S}:= V(\Gamma)$, 
$$\mathcal{P}:= \{ (P, \iota_P) \mid P \text{ embedded path or cycle, } \iota_P \text{ central symmetry}\},$$ 
and $\mathcal{G}:= \{ (G_i:= \mathbb{Z}_2, \dot{G}_i:= \{1\}, \alpha_i:= \mathrm{id} ) \}$. We refer to the cactus product $\mathrm{Cact}(\mathcal{S}, \mathcal{P}, \mathcal{G})$ as the cactus group over $\Gamma$. 

\medskip \noindent
The cactus group over the path $P_n$ of length $n$ coincides with the standard cactus group $J_{n+1}$, and the cactus group over the cycle $C_n$ of length $n$ coincides with the affine cactus group $AJ_n$. 
\end{ex}

\subsection{Trickle groups}\label{section:Trickle}

\noindent
In this section, we relate the \emph{trickle groups} recently introduced in \cite{Trickle} to quandle products. We start by recalling their definition. 

\begin{definition}\label{def:Trickle}
A \emph{trickle graph} is a quadruple $(\Gamma, \leq , \mu , (\varphi_x)_{x \in V(\Gamma)})$, where
\begin{itemize}
	\item $\Gamma$ is a simplicial graph,
	\item $\leq$ is a partial order on $V(\Gamma)$,
	\item $\mu$ is a labelling $\mu : V(\Gamma) \to \mathbb{N}_{\geq 2} \cup \{\infty\}$ of the vertices,
	\item $\varphi_x : \mathrm{star}_\Gamma(x) \to \mathrm{star}_\Gamma(x)$ is an automorphism of $\mathrm{star}_\Gamma(x)$ for every $x \in V(\Gamma)$.
\end{itemize}
For $x,y \in V(\Gamma)$, the notation $x \parallel y$ means that $x$ and $y$ are not $\leq$-comparable. We set $E_\parallel (\Gamma):= \{ \{x,y\} \in E(\Gamma) \mid x \parallel y\}$. The quadruple $(\Gamma, \leq, \mu, (\varphi_x)_{x \in V(\Gamma)})$ must satisfy the following conditions.
\begin{itemize}
	\item[(a)] For all $x,y \in V(\Gamma)$, if $x<y$ then $\{x,y\} \in E(\Gamma)$.
	\item[(b)] For all $x,y,z \in V(\Gamma)$, if $\{x,y\} \in E_\parallel(\Gamma)$ and $z \leq y$, then $\{x,z\} \in E_\parallel (\Gamma)$.
	\item[(c)] For all $x \in V(\Gamma)$ and $y,z \in \mathrm{star}_\Gamma(x)$, we have $z \leq y$ if and only if $\varphi_x(z) \leq \varphi_x(y)$.
	\item[(d)] For all $x \in V(\Gamma)$ and $y \in \mathrm{star}_\Gamma(x)$, if $\varphi_x(y) \neq y$ then $y<x$.
	\item[(e)] For every $x \in V(\Gamma)$, if $\mu(x)$ is finite, then the order of $\varphi_x$ is finite and divides $\mu(x)$.
	\item[(f)] For all $x \in V(\Gamma)$ and $y \in \mathrm{star}_\Gamma(x)$, $\mu(\varphi_x(y))= \mu(y)$.
	\item[(g)] For all $x,y,z \in V(\Gamma)$, if $z<y<x$ then $\varphi_x \circ \varphi(y)(z)= \varphi_{y'} \circ \varphi_x(z)$ where $y':= \varphi_x(y)$. 
\end{itemize}
\end{definition}

\noindent
Our next result shows that trickle groups are quandle products, but it also characterises precisely which quandle products are trickle groups. Thus, one gets an alternative definition of trickle groups thanks to the perspective of quandle products. 

\begin{prop}\label{prop:QuandleVsTrickle}
A group is a trickle group $\mathfrak{T}(\Gamma, \leq , \mu, (\varphi_x)_{x \in V(\Gamma)})$ if and only if it is a quandle product of $\{ \mathbb{Z}_{\mu(x)}, x \in V(\Gamma)\}$ with trivial holonomy.
\end{prop}

\begin{proof}
Let $(\Gamma, \leq , \mu, (\varphi_x)_{x \in V(\Gamma)})$ be a trickle graph. Our goal is to define a quandle system such that the presentations defining the trickle and quandle groups yield the same group.

\medskip \noindent
First, consider the oposet $\mathcal{I}:= (V(\Gamma), \leq , \perp)$ where $\perp$ is defined as follows: for all $u,v \in V(\Gamma)$, $u \perp v$ holds if and only if $\{u,v \} \in E(\Gamma)$ but $u$ and $v$ are not $\leq$-comparable. Notice that $\mathcal{I}$ is indeed an oposet as a consequence of (b). Then, consider the collection of groups $\mathcal{G}:= \{ G_u \simeq \mathbb{Z}_{\mu(u)} \mid u \in V(\Gamma)\}$. For convenience, we fix a generator $z_u$ of $G_u$ for every $u \in V(\Gamma)$. Finally, given a vertex $u \in V(\Gamma)$, consider the action $G_u \curvearrowright \bigsqcup_{v<u} G_v$ defined by
$$\left\{ \begin{array}{l} z_u \ast v: = \varphi_u(v) \\ z_u \ast z_v := z_{\varphi_u(v)} \end{array} \right. \text{ for every } v < u.$$
We need to justify that this indeed defines an action. Notice that, as a consequence of (c), if $v < u$ then $\varphi_u(v)< \varphi_u(u)$. But it follows from (d) that $\varphi_u(u)=u$, hence $\varphi_u(v)<u$. Therefore, $\varphi_u$ induces a bijection of $\{ v \mid v<u\}$. Moreover, (e) implies that, if $\mu(u)$ is finite, then $\mathrm{ord}(\varphi_u)$ is finite and divides $\mu(u)= \mathrm{ord}(z_u)$. Consequently, $z_u \ast v := \varphi_u(v)$ defines an action of $G_u$ on $\mathrm{Obj}(\bigsqcup_{v<u} G_v)$. Then, we know from (f) that $\mathrm{ord}(z_{\varphi_u(v)}) = \mathrm{ord}(z_v)$ for every $v<u$. In other words, $z_u$ permutes the isotropy groups $G_v$ of $\bigcup_{v<u} G_v$ by preserving the isomorphism type. We conclude that our action $G_u \curvearrowright \bigsqcup_{v<u} G_v$ is well-defined.

\medskip \noindent
Notice that, according to (c), the action $G_u \curvearrowright \mathrm{Obj}(\bigsqcup_{v<u} G_v)$ preserves the order $\leq$, and necessarily the orthogonality $\perp$. 

\medskip \noindent
Let $u,v,w \in V(\Gamma)$ satisfying $u<v<w$. Notice that, according to (a), our vertices are pairwise adjacent in $\Gamma$. We have
$$z_w \ast ( z_v \ast u) = \varphi_w \circ \varphi_v(u) \overset{\text{(f)}}{=} \varphi_{\varphi_w(v)} \circ \varphi_w(u) = \varphi_w(v) \ast \varphi_w(u) = (z_w \ast z_v) \ast (z_w \ast u),$$
and similarly
$$z_w \ast ( z_v \ast z_u)= z_{\varphi_w \circ \varphi_v(u)} \overset{\text{(f)}}{=}  z_{\varphi_{\varphi_w(v)} \circ \varphi_w(u)} = (z_w \ast z_v) \ast ( z_w \ast z_u).$$
Thus, we have proved that $(\mathcal{I}, \mathcal{G}, \mathcal{A}:= \{ G_u \curvearrowright \bigsqcup_{v<u} G_v, u \in V(\Gamma)\})$ is a quandle system. The holonomy is trivial because, given two vertices $u<v$, the stabiliser of $G_v$ for the action of $G_u$ on $\bigsqcup_{w<u} G_w$ is non-trivial unless $\varphi_u$ fixes $v$, in which case $z_u$ induces the identity on $G_v$. 

\medskip \noindent
The fact that the presentation given by our trickle graph coincides with the presentation given by our quandle system follows from (d). 

\medskip \noindent
Conversely, let $(\mathcal{I}, \mathcal{G}, \mathcal{A})$ be a quandle system with trivial holonomy such that every group in $\mathcal{G}$ is cyclic. Our goal is to define a trickle graph such that the presentations defining the trickle and quandle groups coincide.

\medskip \noindent
Let $\Gamma$ denote the graph whose vertex-set is $I$ and whose edges connect two vertices whenever they are $\leq$- or $\perp$-comparable. We endow $V(\Gamma)=I$ with the order given by the oposet $\mathcal{I}$. Define the labelling $\mu : i \mapsto |G_i|$. Finally, for every edge $\{i,j\} \in E(\Gamma)$, set
$$\varphi_i(j):= \left\{ \begin{array}{l} j \text{ if } i \leq j \text{ or } i\perp j \\ k \text{ such that } i \ast j \in G_k \text{ if } i>j \end{array} \right..$$
Notice that $\varphi_i$ induces an automorphism of $\mathrm{star}_\Gamma(i)$ because $G_i$ has a well-defined action on $\mathrm{Obj}(\bigsqcup_{j<i} G_j)$. 

\medskip \noindent
Now, we need to verify that all the conditions from Definition~\ref{def:Trickle} are satisfied. Condition (a) is clear and (b) holds because $\mathcal{I}$ is an oposet. Conditions (c), (e), and (f) hold because $G_i$ has a well-defined action on the groupoid $\bigcup_{j<i} G_j$ that preserves $\leq$. In order to verify (d), let $i \in V(\Gamma)$ and $j \in \mathrm{star}_\Gamma(i)$. We know by construction of $\Gamma$ that $i \leq j$, or $j \leq i$, or $i \perp j$. But we also know by construction of $\varphi_i$ that $\varphi_i(j) = j$ whenever $j \geq i$ or $j \perp i$. Consequently, if $\varphi_i(j) \neq j$ then we must have $i>j$, as desired. Finally, (g) holds because of the quandle relation.

\medskip \noindent
Thus, we have proved that $(\Gamma, \leq , \mu, (\varphi_i)_{i \in V(\Gamma)})$ is a trickle graph. The fact that the presentation given by this trickle graph coincides with the presentation given by our quandle system follows from the triviality of the holonomy (see Remark~\ref{remark:NoHolonomy}). 
\end{proof}

\noindent
The next statement summarises what can be deduced about trickle groups from our results about quandle products thanks to Proposition~\ref{prop:QuandleVsTrickle}.

\begin{cor}\label{cor:Trickle}
Let $(\Gamma, \leq , \mu , (\varphi_x)_{x \in V(\Gamma)})$ be a trickle graph with $\Gamma$ finite and let $\mathfrak{T}$ denote the corresponding trickle group. 
\begin{itemize}
	\item[(i)] The Cayley graph $\mathrm{Cayl}(\mathfrak{T}, V(\Gamma))$ is quasi-median. Consequently, $\mathfrak{T}$ is cocompactly cubulable. 
	\item[(ii)] For every prime $p \geq 2$, $\mathfrak{T}$ contains an element of order $p$ if and only if there exists $x \in V(\Gamma)$ such that $\mu(x)$ is disible by $p$. Consequently, $\mathfrak{T}$ is torsion-free if and only if $\mu(x)= \infty$ for every $x \in V(\Gamma)$. 
	\item [(iii)] $\mathfrak{T}$ decomposes as $$G_1 \rtimes ( G_2 \rtimes (\cdots \rtimes G_n))$$ where each $G_i$ is a graph product of finite groups.
	\item[(iv)] If $\mu(x)=\infty$ for every $x \in V(\Gamma)$, then $\mathfrak{T}$ is orderable.
	\item[(v)] For all conjugates of infinite-order generators $g_1, \ldots, g_m$, there exists $N \geq 1$ such that $\{g_1^N, \ldots, g_m^N\}$ is the basis of a right-angled Artin group.
\end{itemize}
\end{cor}

\begin{proof}
Theorem~\ref{thm:MedianCayl} implies that $\mathrm{Cayl}(\mathfrak{T}, V(\Gamma))$ is quasi-median. Since every quasi-median graph can be turned canonically into a median graph (see \cite[Proposition~4.16]{QM}), it follows that $\mathfrak{T}$ acts geometrically on a median graph, which amounts to saying that $\mathfrak{T}$ acts geometrically on a CAT(0) cube complex. This proves (i).

\medskip \noindent
Then, (ii)  and (iv) are consequences of Corollary~\ref{cor:CombApplications}; (iii) follows from Corollary~\ref{cor:DecompositionQuandle}; and (v) is a particular case of Corollary~\ref{cor:SubRAAG}. 
\end{proof}

\section{Open questions}

\noindent
In this final section, we record several natural questions which we leave open.

\paragraph{Other examples.} Recall that, to every oriented graph $\Gamma$ and every labelling $\lambda : \vec{E}(\Gamma ) \to V(\Gamma)$, is associated a \emph{LOG group} defined by the presentation
$$\langle V(\Gamma) \mid u \lambda(e) = \lambda(e) v \text{ for every } (u,v) \in \vec{E}(\Gamma) \rangle.$$
In view of the similarity between such presentations and the presentations defining quandle products, it is natural to ask which LOG groups are quandle products. In full generality, the question is probably too difficult to handle. 

\medskip \noindent
An important class of LOG groups is given by knot groups, as described by their Wirtinger presentations. Hence the more specific question:

\begin{question}
Which knot groups are quandle products of cyclic groups?
\end{question}

\noindent
Another interesting source of LOG groups is given by diagram groups associated to complete semigroup presentations \cite{MR1396957}. 

\begin{question}
Which diagram groups are quandle products of cyclic groups? 
\end{question}

\noindent
It is worth mentioning that Thompson's group $F$, a notable example of diagram group, turns out to be a trickle group \cite{Trickle}, and consequently a quandle product of infinite cyclic groups. However, its commutator subgroup $F'$, which is also a diagram group \cite{MR1725439}, is not a quandle product of infinite cyclic groups. Indeed, $F'$ is simple and every quandle product of infinite cyclic groups surjects onto $\mathbb{Z}_2$. For the same reason, $F'$ is already not a LOG group.

\paragraph{Quandle products of monoids.} Thompson's group $F$ admits the well-known infinite presentation
$$\langle x_0,x_1, \ldots \mid x_nx_k = x_kx_{n+1} \text{ for all } k<n \rangle.$$
Despite the fact that $F$ is a quandle product of infinite cyclic groups, as mentioned above, this is not the presentation of a quandle product. Indeed, $\langle x_0 \rangle$ does not permute $\langle x_1 \rangle, \langle x_2 \rangle, \ldots$, the ``negative translates'' are missing. The construction in \cite{Trickle} essentially amounts to adding these missing generators (namely $x_0^k x_1x_0^{-k}$ for $k >0$) to the presentation. Loosely speaking, we fill a gap between $x_0$ and $x_1$ by adding infinitely many generators. But, for the same reason, we have to fill a gap between $x_i$ and $x_{i+1}$ for every $i \geq 1$. And, then, we have to fill a gap between any two consecutive new generators. And so on and so forth. Eventually, we find the infinite presentation of a quandle product whose generators are labelled by $\mathbb{Z}[1/2]$. 

\medskip \noindent
However, $\langle x_0 \rangle$ does permute $\langle x_1 \rangle, \langle x_2 \rangle, \ldots$, but as a monoid. Thus, the monoid presentation define by the infinite presentation above would correspond to a quandle product of monoids. Our geometric perspective on Cayley graphs of quandle products of groups seems to extend to quandle products of monoids. For instance, we can recover the fact $\mathrm{Cayl}(F, \{x_1,x_2, \ldots\})$ is a median graph (which is already implicit in \cite{MR752825}). 

\begin{question}
Is there a good theory of quandle products of monoids and their groups?
\end{question}

\noindent
We already know that there is a tight connection between $F$ and its monoid. Does this phenomenon extend to arbitrary quandle products of monoids?

\paragraph{Intersection of parabolic subgroups.} It is a common phenomenon in group theory that intersections of parabolic subgroups are again parabolic subgroups. In the framework we are interested in, we already know that it holds for graph products \cite{MR3365774} and at least for standard parabolic subgroups in quandle products according to Corollary~\ref{cor:WP}. 

\begin{question}
In a quandle product, is the intersection between two arbitrary parabolic subgroups again a parabolic subgroup?
\end{question}

\noindent
Motivated by the proof of \cite[Theorem~1.7]{Mediangle}, we suggest the sketch of a geometric proof as follows. Let $\mathfrak{Q}:=\mathfrak{Q}(\mathcal{I}, \mathcal{G}, \mathcal{A})$ be a quandle product and $g\langle R \rangle g^{-1},h \langle S \rangle h^{-1} \leq \mathfrak{Q}$ two parabolic subgroups. With respect to the action of $\mathfrak{Q}$ on its quasi-median Cayley graph $\mathfrak{M}:= \mathfrak{M}(\mathcal{I}, \mathcal{G}, \mathcal{A})$, our parabolic subgroups coincide with the stabilisers of the gated subgraphs $g \langle R \rangle$ and $h \langle S \rangle$. Let $P$ denote the gated-projection of $h \langle S \rangle$ on $g \langle R \rangle$. Fix a point $k \in P$ and let $T \subset I$ be such that every generator labelling an edge of $P$ belongs to $G_t$ for some $t \in T$. The inclusion $P \subset k \langle T \rangle$ is clear. We can expect the reverse inclusion to hold. If so, $P= k \langle T \rangle$. Then, the inclusion
$$g \langle R \rangle g^{-1} \cap h \langle S \rangle h^{-1} \subset \mathrm{stab}(P)= k \langle T \rangle k^{-1}$$
is clear. It remains to verify the reverse inclusion to conclude that the intersection of our two parabolic subgroups is again parabolic.

\paragraph{Residual finiteness.} A natural question, which is also asked in \cite{Trickle} for trickle groups, is whether quandle products of finitely many residually finite groups with trivial holonomy are residually finite. (Notice that asking for finitely many factors and for trivial holonomy is necessary as Thompson's group $F$ and wreath products $F \wr \mathbb{Z}$ with $F$ non-abelian are not residually finite.) In this perspective, \cite[Proposition~7.5]{QM} allows us to show that a group is cocompact special \cite{MR2377497} whenever it admits an action on a quasi-median graph with good properties. 

\medskip \noindent
The main difficulty is to find a finite-index subgroup that avoids bad configurations of hyperplanes. Such configurations are easy to identify, thanks to the explicit structure of our quasi-median graph and its hyperplanes, but how to find a finite-index subgroup that avoids this obstruction? For a cactus group $J_n$, which we think as a quandle product $\mathfrak{Q}:= \mathfrak{Q}(\mathcal{I}, \mathcal{G}, \mathcal{A})$ where $I:= \{ [i,j] \subset [n] \mid i < j \}$, we can enrich $\mathcal{I}$ by adding the singletons $\{i\}$ to $I$ in order to get a new oposet $\mathcal{I}^+$ whose minimal elements are $\leq$- or $\perp$-comparable to all the other elements of the oposet. Then, these minimal elements yields a stable subset $I_\mathrm{min} \subset I^+$ on which $\mathfrak{Q}$ naturally acts. The kernel of this action coincides with the pure cactus group, which is cocompact special. Can this construction be generalised? As a more concrete question, we ask:

\begin{question}\label{question:CocompactSpecial}
Is a quandle product of finitely many finite groups virtually cocompact special?
\end{question}

\noindent
Notice that a positive answer would imply that a quandle product with trivial holonomy of finitely many residually finite groups (e.g.\ a trickle group defined by a finite trickle graph) is residually finite. 

\medskip \noindent
A more direct strategy to prove the residual finiteness of a quandle product $\mathfrak{Q}:= \mathfrak{Q}(\mathcal{I}, \mathcal{G}, \mathcal{A})$, with $\mathcal{I}$ finite and with $G$ finite for every $G \in \mathcal{G}$, is the following. Fix a non-trivial element $g \in \mathfrak{Q}$. In the quasi-median Cayley graph $\mathfrak{M}:= \mathfrak{M}(\mathcal{I}, \mathcal{G}, \mathcal{A})$, choose a finite gated subgraph $Y$ containing both $1$ and $g$. Now, consider the subgroup $H \leq \mathfrak{Q}$ generated by all the cliques tangent to $Y$ (i.e.\ with a vertex in $Y$ but edge in $Y$). This is a finite-index subgroup of $\mathfrak{Q}$, but does it contain $g$? 

\medskip \noindent
The decomposition provided by Corollary~\ref{cor:DecompositionQuandle} is also promising. Typically, residual finiteness is not stable under semidirect products, but the semidirect products in Corollary~\ref{cor:DecompositionQuandle} are far from being arbitrary. As a preliminary question (which corresponds to the case where the holonomy is trivial), we ask: 

\begin{question}
When is a graph-wreath product residually finite?
\end{question}

\noindent
We refer the reader to Section~\ref{section:Wreath} for the definition of graph-wreath products.

\paragraph{CAT(0)ness.} In view of Theorems~\ref{thm:NPC} and~\ref{thm:MedianCayl}, it is natural to ask whether acting geometrically on a median graph or on a CAT(0) space is preserved by quandle products (of finitely many groups without holonomy). The machinery developed in \cite{QM} applies to graph products, but not to quandle product verbatim. The main reason is that, for the actions of quandle products on their quasi-median Cayley graphs, the stabiliser of a prism does not decompose as the product of the stabilisers of its cliques. However, we expect that this is only a technical difficulty, and that the conclusion should still hold. 

\begin{conj}\label{conj:CAT}
A quandle product with trivial holonomy of finitely many groups that are cocompactly cubulable (resp.\ CAT(0)) is cocompactly cubulable (resp.\ CAT(0)). 
\end{conj}

\paragraph{Finite subgroups.} Building blocks of quasi-median graphs are cliques, prisms, and hyperplanes. Corollary~\ref{cor:QuandleClique} describes clique-stabilisers in our quasi-median graphs and Corollary~\ref{cor:HypStab} hyperplane-stabiliser. However, the structure if prism-stabilisers is more delicate. 

\begin{question}
Let $\mathfrak{Q}(\mathcal{I}, \mathcal{G}, \mathcal{A})$ be a quandle product acting on its quasi-median Cayley graph $\mathfrak{M}(\mathcal{I}, \mathcal{G}, \mathcal{A})$. What is the algebraic structure of prism-stabilisers?
\end{question}

\noindent
A better understanding of this question would bring valuable information on the structure of finite subgroups, on distorted elements or even distorted abelian subgroups, on the structure of polycyclic subgroups, and on finiteness properties of quandle products. As concrete questions:

\begin{problem}\label{prob:FiniteSub}
Describe finite subgroups in quandle products. 
\end{problem}

\begin{question}
Is a quandle product with trivial holonomy of finitely many groups of type $F_n$ also of type $F_n$?
\end{question}

\noindent
We emphasize that, even for cactus groups, the structure of finite subgroups is not fully understood. See \cite[Conjecture~8.5]{MR4874027}.

\paragraph{Asymptotic geometry.} With the notable exception of asymptotic dimension, we have not discussed the large-scale geometry of quandle products. A basic question in this direction is:

\begin{question}\label{question:QuandleHyp}
When is a finitely generated quandle product hyperbolic?
\end{question}

\noindent
In view of \cite[Theorem~5.17]{QM}, a preliminary step would be to determine when quasi-median Cayley graphs of quandle products are hyperbolic. In the same vein, we can ask

\begin{question}
When is a quandle product acylindrically hyperbolic?
\end{question}

\noindent
The question may be more accessible than Question~\ref{question:QuandleHyp}, as it would be sufficient to find WPD contracting isometries in quasi-median Cayley graphs of quandle products. 

\medskip \noindent
Then, in another direction and inspired by \cite[Proposition~8.7]{MR4808711}, we ask:

\begin{question}
Let $\mathfrak{Q}:= \mathfrak{Q}(\mathcal{I}, \mathcal{G}, \mathcal{A})$ be a quandle product with trivial holonomy and with $\mathcal{I}$ finite. Does $\mathfrak{Q}$ quasi-isometrically embedded into a product of finitely many trees of spaces whose vertex-spaces are quasi-isometric to groups in $\mathcal{G}$?
\end{question}

\noindent
We know from \cite[Proposition~5.42]{MR4808711} a positive answers for graph products. A positive answer in full generality would provide information, for instance, about Dehn functions of quandle products.

\paragraph{Algorithmic problems.} As illustrated in \cite{MR4874027} for cactus groups, admitting a (quasi-)median Cayley graph allows one to solve explicitly and efficiently various algorithmic properties. In this direction, a natural question to ask is:

\begin{question}
Let $\mathfrak{Q}$ be a quandle product of finitely many cyclic groups (with trivial holonomy). Solve explicitly the conjugacy problem in $\mathfrak{Q}$. 
\end{question}

\noindent
Notice that, when the holonomy is trivial, we already know that the conjugacy problem is solvable, since the group is cocompactly cubulable according to Corollary~\ref{cor:Trickle}. But the goal is to exploit the quasi-median geometry of the Cayley graph in order to find a solution as simple as possible.

\addcontentsline{toc}{section}{References}

\bibliographystyle{alpha}
{\footnotesize\bibliography{QuandleProd}}

\newcommand{\etalchar}[1]{$^{#1}$}
\begin{thebibliography}{BCC{\etalchar{+}}13}

\bibitem[Abr00]{MR2701024}
A.~Abrams.
\newblock {\em Configuration spaces and braid groups of graphs}.
\newblock ProQuest LLC, Ann Arbor, MI, 2000.
\newblock Thesis (Ph.D.)--University of California, Berkeley.

\bibitem[AM15]{MR3365774}
Y.~Antol\'in and A.~Minasyan.
\newblock Tits alternatives for graph products.
\newblock {\em J. Reine Angew. Math.}, 704:55--83, 2015.

\bibitem[BCC{\etalchar{+}}13]{MR3062742}
B.~Bre\v{s}ar, J.~Chalopin, V.~Chepoi, T.~Gologranc, and D.~Osajda.
\newblock Bucolic complexes.
\newblock {\em Adv. Math.}, 243:127--167, 2013.

\bibitem[BCL24]{MR4701892}
P.~Bellingeri, H.~Chemin, and V.~Lebed.
\newblock Cactus groups, twin groups, and right-angled {A}rtin groups.
\newblock {\em J. Algebraic Combin.}, 59(1):153--178, 2024.

\bibitem[BG84]{MR752825}
K.~Brown and R.~Geoghegan.
\newblock An infinite-dimensional torsion-free {${\rm FP}\sb{\infty }$}\ group.
\newblock {\em Invent. Math.}, 77(2):367--381, 1984.

\bibitem[BGP24]{Trickle}
P.~Bellingeri, E.~Godelle, and L.~Paris.
\newblock Trickle groups.
\newblock {\em arxiv:2412.04932}, 2024.

\bibitem[BM15]{AsdimGP}
G.~Bell and D.~Moran.
\newblock On constructions preserving the asymptotic topology of metric spaces.
\newblock {\em The North Carolina Journal of Mathematics and Statistics},
  1:46--57, 2015.

\bibitem[BMW94]{MR1297190}
H.-J. Bandelt, H.~Mulder, and E.~Wilkeit.
\newblock Quasi-median graphs and algebras.
\newblock {\em J. Graph Theory}, 18(7):681--703, 1994.

\bibitem[Bon16]{MR3567259}
C.~Bonnaf\'e.
\newblock Cells and cacti.
\newblock {\em Int. Math. Res. Not. IMRN}, (19):5775--5800, 2016.

\bibitem[Car12]{MR2885229}
J.~Carter.
\newblock A survey of quandle ideas.
\newblock In {\em Introductory lectures on knot theory}, volume~46 of {\em Ser.
  Knots Everything}, pages 22--53. World Sci. Publ., Hackensack, NJ, 2012.

\bibitem[CE10]{MR2672155}
M.~Clancy and G.~Ellis.
\newblock Homology of some {A}rtin and twisted {A}rtin groups.
\newblock {\em J. K-Theory}, 6(1):171--196, 2010.

\bibitem[CG25]{Contracting}
L.~Ciobanu and A.~Genevois.
\newblock Contracting elements and conjugacy growth in coxeter groups, graph
  products, and further groups.
\newblock {\em arxiv:2504.15636}, 2025.

\bibitem[Che25]{AffineCactus}
H.~Chemin.
\newblock Combinatorics of affine cactus groups.
\newblock {\em arxiv:2501.16270}, 2025.

\bibitem[Chi12]{MR2946302}
I.~Chiswell.
\newblock Ordering graph products of groups.
\newblock {\em Internat. J. Algebra Comput.}, 22(4):1250037, 14, 2012.

\bibitem[DJ90]{MR1127191}
N.~Dershowitz and J.-P. Jouannaud.
\newblock Rewrite systems.
\newblock In {\em Handbook of theoretical computer science, {V}ol.\ {B}}, pages
  243--320. Elsevier, Amsterdam, 1990.

\bibitem[DS06]{MR2213160}
A.~Dranishnikov and J.~Smith.
\newblock Asymptotic dimension of discrete groups.
\newblock {\em Fund. Math.}, 189(1):27--34, 2006.

\bibitem[Gen19]{MR4033512}
A.~Genevois.
\newblock Embeddings into {T}hompson's groups from quasi-median geometry.
\newblock {\em Groups Geom. Dyn.}, 13(4):1457--1510, 2019.

\bibitem[Gen20]{MR4071367}
A.~Genevois.
\newblock Contracting isometries of {${\rm CAT}(0)$} cube complexes and
  acylindrical hyperbolicity of diagram groups.
\newblock {\em Algebr. Geom. Topol.}, 20(1):49--134, 2020.

\bibitem[Gen22]{Mediangle}
A.~Genevois.
\newblock Rotation groups, mediangle graphs, and periagroups: a unified point
  of view on {C}oxeter groups and graph products of groups.
\newblock {\em arxiv:2212.06421}, 2022.

\bibitem[Gen23]{MR4586831}
A.~Genevois.
\newblock Special cube complexes revisited: a quasi-median generalization.
\newblock {\em Canad. J. Math.}, 75(3):743--777, 2023.

\bibitem[Gen24]{MR4808711}
A.~Genevois.
\newblock Automorphisms of graph products of groups and acylindrical
  hyperbolicity.
\newblock {\em Mem. Amer. Math. Soc.}, 301(1509):vi+127, 2024.

\bibitem[Gen25a]{MR4874027}
A.~Genevois.
\newblock Cactus groups from the viewpoint of geometric group theory.
\newblock {\em Topology Proc.}, 66:59--103, 2025.

\bibitem[Gen25b]{QM}
A.~Genevois.
\newblock Cubical-like geometry of quasi-median graphs and applications to
  geometric group theory.
\newblock {\em PhD Thesis, arXiv:1712.01618}, 2025.

\bibitem[Gen25c]{FBn}
A.~Genevois.
\newblock Flat braid groups, right-angled a{}rtin groups, and commensurability.
\newblock {\em arxiv:2502.17917}, 2025.

\bibitem[Gen25d]{BuildingQM}
A.~Genevois.
\newblock Right-angled buildings as quasi-median graphs.
\newblock {\em in preparation}, 2025.

\bibitem[Gre90]{GreenGP}
E.~Green.
\newblock Graph products of groups.
\newblock {\em PhD thesis, University of Leeds}, 1990.

\bibitem[GS97]{MR1396957}
V.~Guba and M.~Sapir.
\newblock Diagram groups.
\newblock {\em Mem. Amer. Math. Soc.}, 130(620):viii+117, 1997.

\bibitem[GS99]{MR1725439}
V.~Guba and M.~Sapir.
\newblock On subgroups of the {R}. {T}hompson group {$F$} and other diagram
  groups.
\newblock {\em Mat. Sb.}, 190(8):3--60, 1999.

\bibitem[GT24]{Halo}
A.~Genevois and R.~Tessera.
\newblock Lamplighter-like geometry of groups.
\newblock {\em arxiv:2401.13520}, 2024.

\bibitem[Hag23]{MR4645691}
F.~Haglund.
\newblock Isometries of {$CAT(0)$} cube complexes are semi-simple.
\newblock {\em Ann. Math. Qu\'e.}, 47(2):249--261, 2023.

\bibitem[HW08]{MR2377497}
F.~Haglund and D.~Wise.
\newblock Special cube complexes.
\newblock {\em Geom. Funct. Anal.}, 17(5):1551--1620, 2008.

\bibitem[Kam02]{MR2002606}
S.~Kamada.
\newblock Knot invariants derived from quandles and racks.
\newblock In {\em Invariants of knots and 3-manifolds ({K}yoto, 2001)},
  volume~4 of {\em Geom. Topol. Monogr.}, pages 103--117. Geom. Topol. Publ.,
  Coventry, 2002.

\bibitem[KM16]{MR3393469}
P.~Kropholler and A.~Martino.
\newblock Graph-wreath products and finiteness conditions.
\newblock {\em J. Pure Appl. Algebra}, 220(1):422--434, 2016.

\bibitem[Los19]{MR3945740}
I.~Losev.
\newblock Cacti and cells.
\newblock {\em J. Eur. Math. Soc. (JEMS)}, 21(6):1729--1750, 2019.

\bibitem[Mul80]{MR605838}
H.~Mulder.
\newblock {\em The interval function of a graph}, volume 132 of {\em
  Mathematical Centre Tracts}.
\newblock Mathematisch Centrum, Amsterdam, 1980.

\bibitem[New42]{MR7372}
M.~Newman.
\newblock On theories with a combinatorial definition of ``equivalence.''.
\newblock {\em Ann. of Math. (2)}, 43:223--243, 1942.

\bibitem[Pri86]{MR856848}
S.~Pride.
\newblock On {T}its' conjecture and other questions concerning {A}rtin and
  generalized {A}rtin groups.
\newblock {\em Invent. Math.}, 86(2):347--356, 1986.

\end{thebibliography}

\Address

%

\end{document}